\documentclass[a4paper, 10pt, reqno]{amsart}

\usepackage{comment}

\usepackage{amsmath}
\usepackage{graphicx}
\usepackage{rotating}
\usepackage{caption}
\usepackage{fancyhdr}
\usepackage{multirow}
\usepackage{amsmath,amsthm,amsfonts,amssymb,amscd}
\usepackage{hyperref}
\usepackage[nameinlink,noabbrev,capitalise]{cleveref}

\usepackage{algorithm}
\usepackage{algpseudocode}  

\usepackage{float}
\usepackage{epsfig}
\usepackage{subcaption}
\usepackage[utf8]{inputenc}
\usepackage[T1]{fontenc}

\usepackage{xcolor}
\usepackage{mdwlist}
\usepackage{amsaddr}

\usepackage{bbm}
\usepackage{enumitem}

\newtheorem{theorem}{Theorem}[section]
\newtheorem{lemma}[theorem]{Lemma}
\newtheorem{corollary}[theorem]{Corollary}
\newtheorem{definition}[theorem]{Definition}
\newtheorem{proposition}[theorem]{Proposition}
\newtheorem{example}[theorem]{Example}
\newtheorem{remark}[theorem]{Remark}

\crefname{theorem}{Theorem}{Theorems}
\Crefname{theorem}{Theorem}{Theorems}
\crefname{lemma}{Lemma}{Lemmas}
\Crefname{lemma}{Lemma}{Lemmas}
\crefname{corollary}{Corollary}{Corollaries}
\Crefname{corollary}{Corollary}{Corollaries}
\crefname{definition}{Definition}{Definitions}
\Crefname{definition}{Definition}{Definitions}
\crefname{proposition}{Proposition}{Propositions}
\Crefname{proposition}{Proposition}{Propositions}
\crefname{example}{Example}{Examples}
\Crefname{example}{Example}{Examples}
\crefname{remark}{Remark}{Remarks}
\Crefname{remark}{Remark}{Remarks}
\crefname{claim}{Claim}{Claims}
\Crefname{claim}{Claim}{Claims}
\crefname{problem}{Problem}{Problems}
\Crefname{problem}{Problem}{Problems}
\crefname{algorithm}{Algorithm}{Algorithms}
\Crefname{algorithm}{Algorithm}{Algorithms}

\AddToHook{env/lemma/begin}{\crefalias{theorem}{lemma}}
\AddToHook{env/corollary/begin}{\crefalias{theorem}{corollary}}
\AddToHook{env/definition/begin}{\crefalias{theorem}{definition}}
\AddToHook{env/proposition/begin}{\crefalias{theorem}{proposition}}
\AddToHook{env/example/begin}{\crefalias{theorem}{example}}
\AddToHook{env/remark/begin}{\crefalias{theorem}{remark}}
\AddToHook{env/claim/begin}{\crefalias{theorem}{claim}}
\AddToHook{env/problem/begin}{\crefalias{theorem}{problem}}

\newcommand{\Rd}{\mathbb R^d}
\newcommand{\R}{\mathbb R}
\newcommand{\N}{\mathbb N}

\newcommand{\cP}{\Pcal}

\newcommand{\rmd}{\mathrm{d}}

\DeclareMathOperator{\supp}{supp}

\DeclareMathOperator*{\argmax}{arg\,max}
\DeclareMathOperator*{\argmin}{arg\,min}
\newcommand{\Cpl}{\mathrm{Cpl}}
\newcommand{\cpl}{\Cpl}
\newcommand{\MCov}{{\rm MCov}}
\newcommand{\Ecal}{\mathcal{E}}
\newcommand{\Tcal}{\mathcal{T}}
\newcommand{\Wcal}{\mathcal{W}}
\newcommand{\Pcal}{\mathcal{P}}

\newcommand{\tr}{\mathrm{tr}}

\newcounter{aspt}
\newcommand{\aspttag}{}

\crefname{aspt}{}{}
\Crefname{aspt}{}{}

\newcommand{\asptitem}[2]{
  \refstepcounter{aspt}
  \renewcommand{\aspttag}{#2}
  \item[{\rm (#1)}]
}

\usepackage[
  backend=biber,
  bibencoding=utf8,
  style=numeric-comp,
  url=false,
  doi=false
]{biblatex}
\begin{document}
\pagestyle{plain}

\title{A weak transport approach to the Schrödinger--Bass bridge}
\author{Manuel Hasenbichler$^1$, Gudmund Pammer$^2$, Stefan Thonhauser$^3$}

\address{
$^1$Institute of Statistics, Graz University of Technology\\
\texttt{manuel.hasenbichler@tugraz.at}\smallskip\\
$^2$Institute of Statistics, Graz University of Technology\\
 \texttt{gudmund.pammer@tugraz.at}\smallskip\\
$^3$Institute of Statistics, Graz University of Technology\\
\texttt{stefan.thonhauser@tugraz.at}
}

\begin{abstract}
We study the Schrödinger--Bass problem\footnotemark[1], a one-parameter family of semimartingale optimal transport problems indexed by $\beta>0$, whose limiting regimes interpolate between the classical Schrödinger bridge, the Brenier--Strassen problem, and, after rescaling, the martingale Benamou--Brenier (Bass) problem.

For each $\beta>0$, we prove that the dynamic Schrödinger--Bass problem is equivalent to a static weak optimal transport (WOT) problem with explicit cost $C_{\mathrm{SB}}^\beta$.
This yields primal and dual attainment, as well as a structural characterization of the optimal semimartingales.
The cost $C_{\mathrm{SB}}^\beta$ is constructed via infimal convolution and deconvolution of the Schrödinger cost with the Wasserstein distance.
In a broader setting, we show that such infimal convolutions preserve the WOT structure and inherit continuity, coercivity, and stability of both values and optimizers with respect to the marginals.

Building on this formulation, we propose a dual ascent algorithm for numerical computation.
We establish monotone improvement of the dual objective and convergence of the iteration to the unique optimizer under suitable integrability assumptions.
\end{abstract}

\maketitle

\footnotetext[1]{The authors first learned of the Schrödinger—Bass problem through a presentation by Huyên Pham. The present work was carried out independently of \textcite{henrylabordere2026bridgingschrodingerbasssemimartingale} and approaches the problem from a different perspective using weak transport techniques.}

\section{Introduction}\label{sec:introduction} 
Semimartingale transport extends classical optimal transport by allowing mass to be transported through stochastic dynamics rather than along deterministic trajectories. This leads, on the one hand, to a richer class of transport mechanisms and, on the other, to a variational framework for constructing semimartingales with prescribed structural properties.
Foundational contributions to the subject include the seminal work of Mikami and Thieullen~\cite{MiTh6,MiTh8}, who treated semimartingale transport with controlled drift and fixed diffusion, and the extension of \textcite{TanTouzi2013}, allowing simultaneous control of drift and diffusion. 
It also includes martingale variants, in which the drift is fixed to zero and only the diffusion coefficient is optimized, beginning with the works of \textcite{BeiglHenryPenkner2013,galichon2014stochastic,beiglbock2016problem}. 
It has found important applications in mathematical finance~\cite{conze2021bass,MR4941918,GuoLoeper2021,JoLoOb26}, and is closely connected to functional inequalities~\cite{ElLeSh20} and stochastic localization techniques in convex geometry~\cite{LeVe24,MiSh24}.

\subsection{The Schr\"odinger--Bass problem}\label{subsec:main.the_SB_problem}
For $\beta > 0$, \textcite{alouadi2026lightsbbmbridgingschrodingerbass} introduced the \emph{Schr\"odinger--Bass problem}
\begin{equation}\label{eq:main.SB.dynamic}\tag{SB$\beta$}
    V_{\mathrm{SB}}^\beta(\mu,\nu)
    :=
    \inf_{\substack{X_0\sim \mu,\, X_1\sim \nu\\ \rmd X_t=a_t \rmd t+b_t\rmd B_t}}
    \mathbb E\!\left[
        \int_0^1
        \left(
            \frac12 |a_t|^2
            +
            \frac{\beta}{2}|b_t-I_d|_{\mathrm{HS}}^2
        \right)\,dt
    \right]
\end{equation}
where $B = (B_t)_{t \in [0,1]}$ is standard Brownian motion, and $a = (a_t)_{t \in [0,1]} \subset \Rd$ and $b = (b_t)_{t \in [0,1]} \subset \R^{d\times d}$ are progressively measurable processes.

As $\beta\uparrow\infty$, deviations of the diffusion coefficient from the identity are increasingly penalized, and the corresponding limiting problem is the \emph{Schr\"odinger problem}
\begin{equation}\label{eq:schroedinger_problem}\tag{SB$\infty$}
    V_{\rm EOT}(\mu,\nu)
    :=
    \inf_{\substack{X_0\sim\mu,\ X_1\sim\nu\\[1pt]\rmd X_t = a_t\rmd t + \rmd B_t}}
    \mathbb E\!\left[\int_0^1 \frac12 |a_t|^2\,\rmd t\right].
\end{equation}

On the other hand, if one rescales by $\beta^{-1}$ and lets $\beta\downarrow0$, any nonzero drift becomes prohibitively expensive, and the limiting problem is the \emph{martingale Benamou--Brenier problem}, or \emph{Bass problem}, of \textcite{BaBeiHueKae2020}:
\begin{equation}\label{eq:bass_problem}\tag{mBB}
    V_{\rm mBB}(\mu,\nu)
    :=
    \inf_{\substack{X_0\sim \mu,\, X_1\sim \nu\\ \rmd X_t=b_t\rmd B_t}}
    \mathbb E\!\left[\int_0^1 \frac12|b_t-I_d|_{\rm HS}^2\,\rmd t\right].
\end{equation}
In this way, the Schr\"odinger--Bass problem~\eqref{eq:main.SB.dynamic} interpolates between these two classical semimartingale transport problems.

The two limiting problems share a striking structural feature: both admit canonical static formulations as \emph{weak optimal transport} (WOT) problems in the sense of \textcite{gozlan2017kantorovich}, and in both cases these formulations can be leveraged to construct the optimal semimartingales.
Writing $H$ for the relative entropy and $\gamma_x$ for the Gaussian law with unit variance and mean $x$, the Schr\"odinger problem~\eqref{eq:schroedinger_problem} satisfies
\[
    V_{\rm EOT}(\mu,\nu)
    =
    \inf_{\pi\in\Cpl(\mu,\nu)}
    \int H(\pi_x\,|\,\gamma_x)\,\mu(\rmd x),
\]
where $\Cpl(\mu,\nu)$ is the set of couplings with marginals $\mu$ and $\nu$, and $(\pi_x)_x$ denotes the disintegration of $\pi$ with respect to $\mu$; see
\textcite{Foellmer1985,Leonard2014}.
Likewise, denoting $\Wcal_2^2$ the Wasserstein-$2$ distance, the Bass problem~\eqref{eq:bass_problem} admits the static formulation
\[
    V_{\rm mBB}(\mu,\nu)
    =
    \inf_{\pi\in\Cpl_M(\mu,\nu)}
    \int \frac12 \Wcal_2^2(\pi_x,\gamma_x)\,\mu(\rmd x),
\]
as shown by \textcite{BaBeiHueKae2020}. Here $\Cpl_M(\mu,\nu)$ denotes the set of martingale couplings of $\mu$ and $\nu$.

In both cases, the corresponding primal optimizers uniquely determine the laws of the optimal semimartingales. 
Moreover, the rich duality theory for weak optimal transport (see, e.g., \textcite{backhoff2019existence,beiglböck2025fundamentaltheoremweakoptimal}) yields corresponding dual formulations from whose optimizers the controls of the optimal semimartingales may be recovered.
For the Schr\"odinger problem~\eqref{eq:schroedinger_problem}, we refer for instance to \textcite{MiTh6}; for the Bass problem~\eqref{eq:bass_problem}, see \textcite{backhoff2023existence}.

\subsection{Main results}\label{subsec:main.results}
Our main result is that the Schr\"odinger--Bass problem~\eqref{eq:main.SB.dynamic} admits three further equivalent formulations: a standard weak transport formulation~\eqref{eq:main.SB.wot}, a variational formulation~\eqref{eq:main.SB.var}, and an explicit dual formulation~\eqref{eq:main.SB.dual}.
More precisely, for every $\beta>0$ and $\mu,\nu\in\Pcal_2(\R^d)$, it can be written as a static standard weak transport problem:
\begin{equation}\label{eq:main.SB.wot}\tag{SB$\beta$w}
    V_{\rm SB}^\beta(\mu,\nu)
    =
    \min_{\pi\in\Cpl(\mu,\nu)}
    \int C_{\rm SB}^\beta(x,\pi_x)\,\mu(\rmd x).
\end{equation}
Here $C_{\rm SB}^\beta\colon \R^d\times\Pcal_2(\R^d)\to\R$ is an explicit, continuous weak transport cost satisfying a quadratic growth bound. 
In addition, the problem admits the variational formulation
\begin{equation}\label{eq:main.SB.var}\tag{SB$\beta$v}
    V_{\rm SB}^\beta(\mu,\nu)
    =
    \max_{\alpha \in \Pcal_2(\R^d)}
    \left\{
        -\frac{\beta}{2}\Wcal_2^2(\mu,\alpha)
        +
        \min_{\rho \in \Pcal_2(\R^d)}
        \left\{
            V_{\rm EOT}(\alpha,\rho)
            +
            \frac{\beta}{2}\Wcal_2^2(\rho,\nu)
        \right\}
    \right\},
\end{equation}
and an explicit dual representation for every $\beta > 0$. Writing
\[
    q_\beta(x):=\frac{\beta}{2}|x|^2,
    \qquad
    \mathcal T^\beta_t[f]
    :=
    -\log\bigl(\exp(-q_\beta\Box(-f))*\gamma_{0;t}\bigr),
\]
where $\gamma_{y;t}$ is the Gaussian density with mean $y$ and covariance $t I_d$, this dual problem takes the form
\begin{equation}\label{eq:main.SB.dual}\tag{SB$\beta$d}
    V_{\rm SB}^\beta(\mu,\nu)
    =
    \max_{\substack{f\in L^1(\nu),\\ \text{$\beta$-semiconcave}}}
    \left\{
        \int f\,\rmd\nu
        -
        \int q_\beta\Box\bigl(-\mathcal T^\beta_1[f]\bigr)\,\rmd\mu
    \right\}.
\end{equation}
Here $q_\beta\Box\bigl(-\mathcal T^\beta_1[f]\bigr)$ denotes the infimal convolution of $q_\beta$ and $-\mathcal T^\beta_1[f]$. 
These formulations are equivalent and uniquely attained in the relevant sense; we refer to \Cref{thm:SB} for the precise statement.

\subsection{The Schrödinger--Bass system}\label{subsec:main.the_SB_system}
An important consequence of \cref{thm:SB} is that the optimizers of the variational and dual formulations determine the optimal semimartingale ${X=(X_t)_{t\in[0,1]}}$ attaining \eqref{eq:main.SB.dynamic}. Let $\alpha$, $\rho$ be the optimizers of the variational problem \eqref{eq:main.SB.var}, and let $f$ denote the optimizer of the dual problem~\eqref{eq:main.SB.dual}. 

We first relate $f$ to the dynamic optimizer of the entropic transport problem $V_{\rm EOT}(\alpha,\rho)$. To this end, set $g_1:=\exp\bigl(-q_\beta\Box(-f)\bigr)$ and $g_t:=g_1*\gamma_{0;1-t}$ for $t \in [0,1]$ so that $z\mapsto g_1(z)/g_0(y)$ is a Gibbs density with respect to $\gamma_{y}:= \gamma_{y;1}$. Then, by \cref{thm:SB}, $V_{\rm EOT}(\alpha,\rho)$ is attained by the F\"ollmer process
\[
    \rmd Y_t = \nabla \log g_t(Y_t)\,\rmd t + \rmd B_t, \quad Y_0 \sim \alpha.
\]

A further consequence of \cref{thm:SB} is that $X_t := Y_t + \frac1\beta \nabla \log g_t(Y_t)$ is the \emph{Schrödinger--Bass bridge}, that is, the optimal semimartingale attaining \eqref{eq:main.SB.dynamic}. 
By It\^{o}'s formula, $X$ satisfies
\begin{equation}\label{eq:main.SB.bridge.withY}
    \rmd X_t = \nabla \log g_t(Y_t)\,\rmd t + \bigl(I_d+\tfrac{1}{\beta}\nabla^2 \log g_t(Y_t)\bigr)\,\rmd B_t, \quad X_0 = \nabla v^*_0(Y_0).
\end{equation}
Equivalently, the transformation from $Y$ to $X$ can be expressed in terms of convex potentials. For $t\in[0,1)$, set 
\[
    v_t := q_1 - \tfrac{1}{\beta} q_\beta \Box \bigl(-\Tcal^\beta_{1-t}[f]\bigr),
    \qquad 
    v_1:=q_1-\tfrac{1}{\beta}f,
\]
where the latter is understood by continuity.
Their convex conjugates satisfy $v^*_t = q_1 - \frac1\beta \Tcal^\beta_{1-t}[f]$, and hence $\nabla (\beta v_t^* - q_\beta ) = \nabla \log g_t$. 
Consequently, $X_t = \nabla v_t^*(Y_t)$ for $t \in [0,1]$, while the Fenchel--Legendre duality yields $Y_t=\nabla v_t(X_t)$ for $t\in[0,1)$. 
Thus \eqref{eq:main.SB.bridge.withY} may be written as
\begin{equation}\label{eq:main.SB.bridge} \tag{SB$\beta$b}
    \rmd X_t = \beta \bigl(X_t - \nabla v_t(X_t)\bigr)\,\rmd t + \bigl(\nabla^2 v_t(X_t)\bigr)^{-1}\,\rmd B_t, \quad X_0 \sim \mu.
\end{equation}

At the terminal time, $v_1$ and $v_1^*$ are the dual potentials attaining the \emph{maximal covariance} between $\rho$ and $\nu$:
\[
    \MCov(\rho,\nu) := \int q_1\,\rmd\rho + \int q_1\,\rmd\nu - \Wcal_2^2(\rho,\nu) = \inf_{\psi \text{ cvx}} \left\{\int \psi^*\,\rmd \rho + \int \psi \,\rmd \nu\right\}. 
\]
In particular, because $\rho$ is absolutely continuous with respect to the Lebesgue measure on $\Rd$, the Brenier map $\nabla v_1^*$ transports $\rho$ to $\nu$, that is, ${(\nabla v_1^*)_\# \rho = \nu}$. The reverse relation, however, need not hold: if $\nu$ charges sets of Lebesgue measure zero, then $v_1$ need not induce a Monge map from $\nu$ to $\rho$. Accordingly, the identity $Y_t=\nabla v_t(X_t)$ holds for $t\in[0,1)$ but may fail at $t=1$.

\begin{figure}[t]
    \centering

    \includegraphics[width=0.475\linewidth]{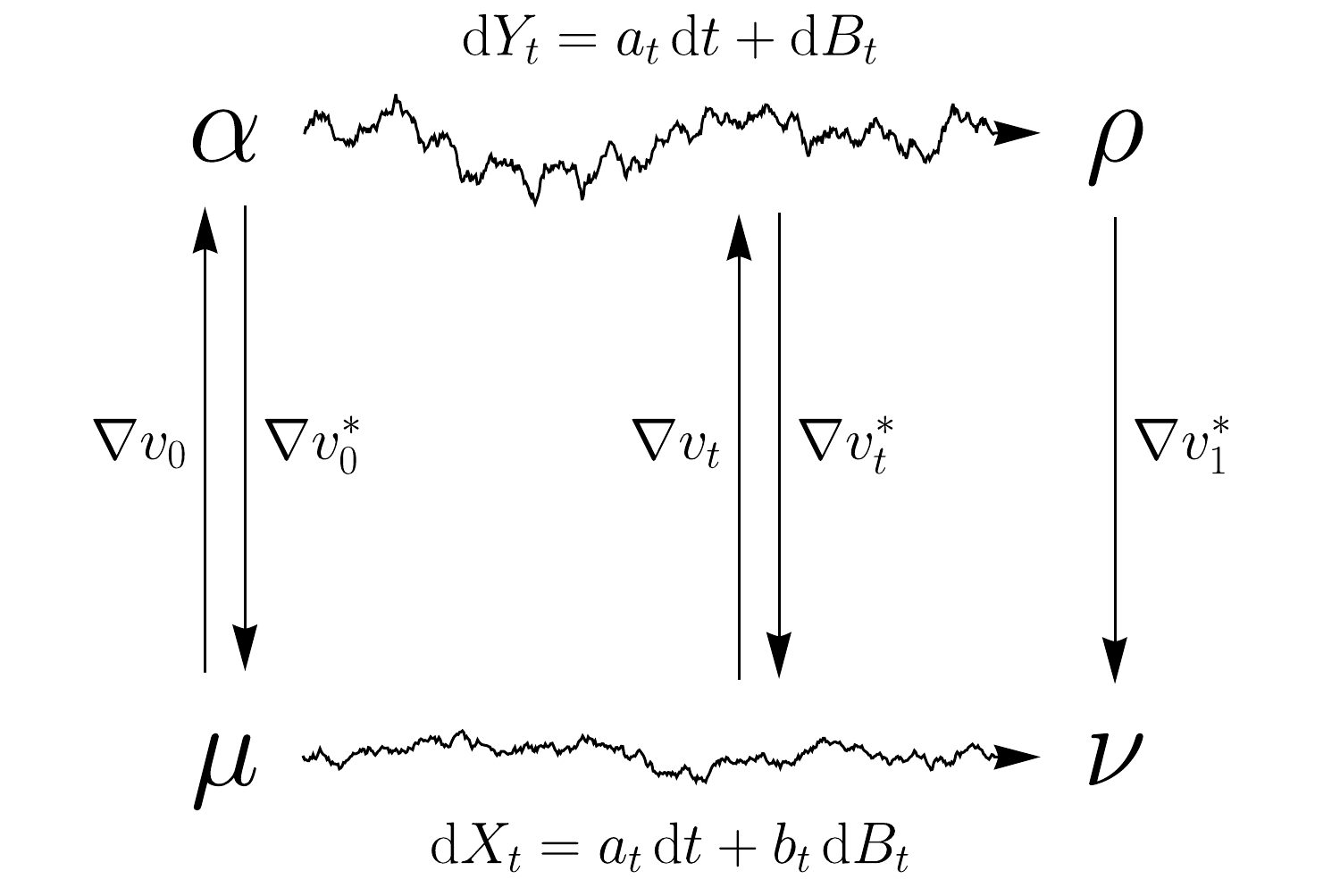}
    
    \caption{
        Schematic illustration of the Schrödinger--Bass system.
        The drift and diffusion coefficients are given by ${a_t := \nabla \log g_t(Y_t)}$ and ${b_t := I_d + \frac1\beta \nabla^2 \log g_t(Y_t)}$, respectively.
    }
    \label{fig:SB_system}
\end{figure}

We emphasize that, conditionally on $X_0=x$, there are two equivalent ways to generate $X_1$:
One may either start at $x$ and evolve the Schrödinger--Bass bridge $X$ according to \eqref{eq:main.SB.bridge}, or start the F\"ollmer process $Y$ from $\nabla v_0(x)$ and then apply the map $\nabla v_1^*$ to $Y_1$.
We refer to this relationship as the \emph{Schr\"odinger--Bass system}; see \cref{fig:SB_system} for a schematic illustration. 
Moreover, \Cref{thm:SB} shows that this system is in fact uniquely determined: If $\alpha,\rho\in\cP_2(\R^d)$ and $v_0,v_1:\R^d\to\R$ are convex functions satisfying the Schr\"odinger--Bass system, then $v_0$ and $v_1$ are unique up to additive constants, while $\alpha$ and $\rho$ are unique in $\cP_2(\R^d)$.
In particular, the dual optimizer in \eqref{eq:main.SB.dual} is recovered as $q_\beta-\beta v_1$.

\subsection{Schrödinger--Bass bridges between Gaussians: an illustration}\label{subsec:main.SB_btw_Gaussians}
If $\mu$ and $\nu$ are Gaussian, then the Schrödinger--Bass bridge can be determined analytically from the characterization provided by the Schrödinger--Bass system; see \cref{ex:SB.bridge.Gaussians} for the derivation. Even for fixed Gaussian marginals, varying $\beta$ produces bridges with distinct dynamics on $[0,1]$, as illustrated by the following example.

\begin{figure}[t]
    \centering

    \newcommand{\basePanelFraction}{0.34}
    \newcommand{\fullPanelWidth}{400} 
    \newcommand{\trimmedPanelWidth}{352}
    \newcommand{\basePanelSize}{\fpeval{\basePanelFraction}\textwidth}
    \newcommand{\trimmedPanelSize}{\fpeval{ \basePanelFraction * \trimmedPanelWidth / \fullPanelWidth }\textwidth}

    \begin{subfigure}[b]{\basePanelSize}
        \centering
        \includegraphics[width=\textwidth]{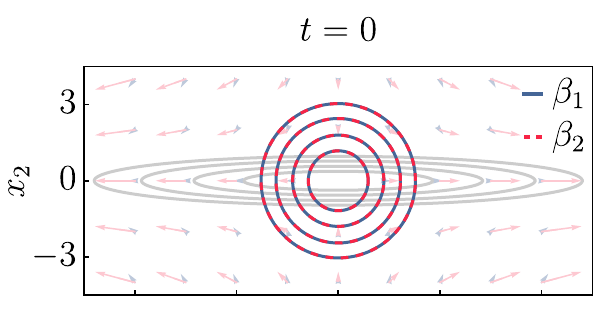}
    \end{subfigure}
    \begin{subfigure}[b]{\trimmedPanelSize}
        \centering
        \includegraphics[width=\textwidth]{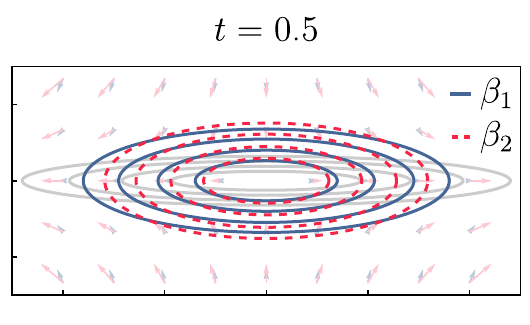}
    \end{subfigure}
    \begin{subfigure}[b]{\trimmedPanelSize}
        \centering
        \includegraphics[width=\textwidth]{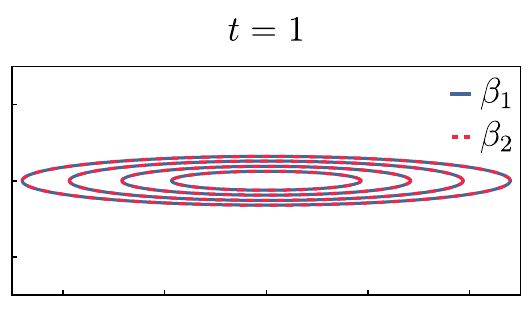}
    \end{subfigure}

    \vspace{-0.85em}

    \begin{subfigure}[b]{\basePanelSize}
        \centering
        \includegraphics[width=\textwidth]{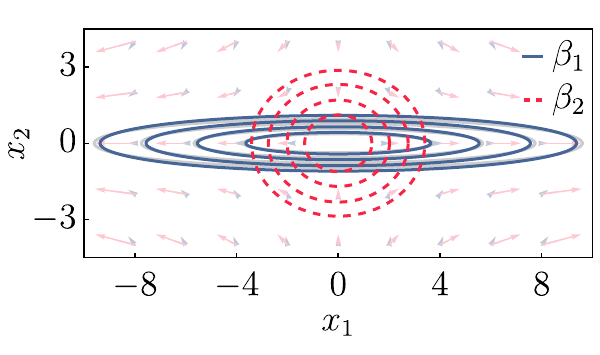}
    \end{subfigure}
    \begin{subfigure}[b]{\trimmedPanelSize}
        \centering
        \includegraphics[width=\textwidth]{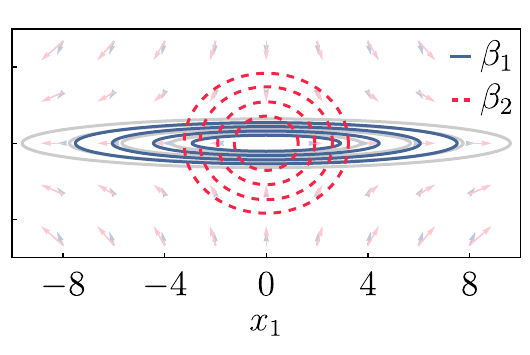}
    \end{subfigure}
    \begin{subfigure}[b]{\trimmedPanelSize}
        \centering
        \includegraphics[width=\textwidth]{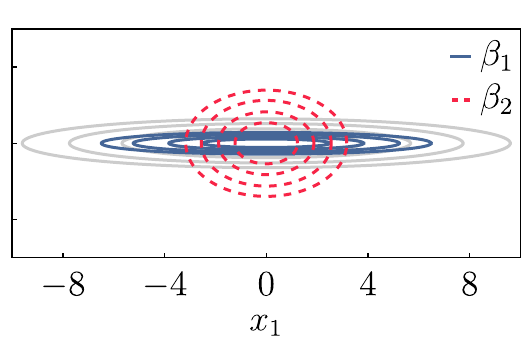}
    \end{subfigure}

    \vspace{-0.5em}
    
    \caption{Schrödinger--Bass bridges from $\mathcal N(0, I_2)$ to $\mathcal N(0, \operatorname{diag}(10,0.1))$ for ${\beta_1 = 0.4}$ and ${\beta_2 = 15}$.
    Top row: probability contours of the marginal law $\operatorname{Law}(X_t^\beta)$. Bottom row: ellipses associated with the diffusion covariance $b^\beta(t)b^\beta(t)^T$. The arrows indicate the drift field $x\mapsto a^\beta(t)^Tx$. The columns correspond to $t=0,\frac12,1$, and the gray ellipses to the target covariance. The contours enclose probabilities $0.5,0.8,0.95$, and $0.99$.}
    \label{fig:SB.Gaussian.panels}
\end{figure}

Let $\mu=\mathcal N(0,I_2)$ and $\nu=\mathcal N\left(0,\operatorname{diag}(\sigma_1^2,\sigma_2^2)\right)$ with $\sigma_1^2 = 10$ and $\sigma_2^2  = 0.1$.
In this case, the problem separates along the coordinate directions, and the Schrödinger--Bass bridge satisfies
\[
    \rmd X^\beta_t  = a^\beta(t)^TX^\beta_t\,\rmd t + b^\beta(t)\,\rmd B_t,
\]
where $a^\beta(t) := \operatorname{diag}\bigl(a^\beta_1(t), a^\beta_2(t)\bigr)$, $b^\beta(t) := \operatorname{diag}\bigl(b_1^\beta(t),b_2^\beta(t)\bigr)$ and, for $i\in\{1,2\}$,
\[
    a_i^\beta(t) = \frac{r_i^\beta}{1-(1-t)r_i^\beta},
    \qquad
    b_i^\beta(t) = \frac{\beta\bigl(1-(1-t)r_i^\beta\bigr)} {\beta\bigl(1-(1-t)r_i^\beta\bigr)-r_i^\beta}.
\]
Here $r_i^\beta$ is the root satisfying $r_i^\beta<\frac{\beta}{\beta+1}$ of
\[
    \sigma_i^2 = \frac{1}{(1-r_i^\beta)^2} + \frac{\beta^2}{(\beta-r_i^\beta)\bigl(\beta-(\beta+1)r_i^\beta\bigr)}.
\]
These bridges are illustrated in \cref{fig:SB.Gaussian.panels} for two different choices of $\beta$.

\subsection{The Schrödinger--Bass algorithm}\label{subsec:main.the_SB_algorithm}
In general, Schrödinger--Bass bridges cannot be determined analytically, and we therefore rely on a numerical procedure for their computation.
In the limiting regimes, efficient algorithms are available for the computation of the corresponding dual optimizers. For the Schrödinger problem~\eqref{eq:schroedinger_problem}, the \emph{Sinkhorn algorithm} is a standard and widely used scheme for computing the Schrödinger potentials; see \textcite{Leonard2014,MaGe20}. For the Bass problem~\eqref{eq:bass_problem}, the \emph{Martingale Sinkhorn algorithm} plays the analogous role for the computation of the \emph{Bass potential}; see \textcite{conze2021bass,MR4941918,hasenbichler2025martingalesinkhornalgorithm}.

Likewise, the Schrödinger--Bass system leads to an alternating scheme for the computation of the dual optimizer in~\eqref{eq:main.SB.dual}. Starting from an arbitrary $\beta$-semiconcave function $f_0\in L^1(\nu)$, for instance $f_0=q_\beta$, one iteratively obtains a sequence of $\beta$-semiconcave, $\nu$-integrable functions $(f_i)_{i\in\N}$ by following the Schrödinger--Bass system; see \cref{fig:SB_system}. More precisely, we propose the following scheme.

\begin{algorithm}[h]
    \caption{The Schrödinger--Bass algorithm}
    \label{alg:SB}
    \begin{algorithmic}[1]
        \Require $\mu, \nu \in \Pcal_2(\R^d)$, $\beta > 0$, and a $\beta$-semiconcave function $f_0 \in L^1(\nu)$
        
        \State $i \gets 1$
        
        \Repeat
            \State $u_i \gets q_1 - \frac{1}{\beta} q_\beta \Box (-\Tcal^\beta_1[f_{i-1}])$
            
            \State $\alpha_i \gets (\nabla u_i)_\#\mu$
            
            \State $\displaystyle 
                f_i \in \argmax_{\substack{f \in L^1(\nu), \\ \beta\text{-semiconcave}}}
                \left\{
                    \int f \,\rmd\nu + \int \Tcal^\beta_1[f] \,\rmd\alpha_i
                \right\}
            $
            
            \State $i \gets i+1$
        \Until{convergence}
    \end{algorithmic}
\end{algorithm}

A key property of \cref{alg:SB} is that it increases the dual value
\[
    \mathcal D_\beta[f]
    :=
    \int f\,\rmd \nu - \int q_\beta \Box(-\Tcal^\beta_1[f])\,\rmd\mu,
    \qquad
    f \in L^1(\nu),\ \beta\text{-semiconcave},
\]
at every iteration. This monotonicity is a consequence of the variational formulation~\eqref{eq:main.SB.var}, since the algorithm alternates between the two optimization problems arising therein; see \cref{lem:strict_ascent}. In particular, under suitable conditions, \cref{thm:convergence_SB} shows that, after normalization, the iterates $(f_i)_{i\in\N}$ obtained by \cref{alg:SB} converge to the dual optimizer of \eqref{eq:main.SB.dual} in the appropriate sense. Consequently, \cref{thm:convergence_SB} yields a constructive proof of the existence of optimizers for each of the four formulations of the Schrödinger--Bass problem, namely \eqref{eq:main.SB.dynamic}, \eqref{eq:main.SB.wot}, \eqref{eq:main.SB.var}, and \eqref{eq:main.SB.dual}.

From a computational point of view, the Schrödinger--Bass algorithm also suggests a natural route to numerical implementation. Indeed, analogously to the implementation of the Martingale Sinkhorn algorithm developed by \textcite{hasenbichler2025martingalesinkhornalgorithm}, one may parametrize the dual functions $f$, $u$, and $v^*$ and replace the exact updates in \cref{alg:SB} by approximate solutions of the two optimization problems. Therefore the weak transport perspective on the Schrödinger--Bass problem~\eqref{eq:main.SB.dynamic} naturally leads to an iterative scheme that can be implemented using standard tools from computational optimal transport. We do not pursue this direction here, since the corresponding implementation follows the same paradigm as for the Martingale Sinkhorn algorithm in the Bass case; see \textcite[Section~4]{hasenbichler2025martingalesinkhornalgorithm}.

\subsection{Related literature}\label{subsec:main.literature}
The Schr\"odinger--Bass problem~\eqref{eq:SB.dynamic} is a semimartingale transport problem in the sense of \textcite{TanTouzi2013}, with running cost
\[
    l(a,b):=\frac12 |a|^2+\frac{\beta}{2}|b-I_d|_{\mathrm{HS}}^2,
    \qquad
    (a,b)\in \R^d\times \R^{d\times d}.
\]
Since $l$ is independent of both the path and the time variable, the problem is of Markovian type. This suggests a PDE characterization of the static dual problem \eqref{eq:main.SB.dual}.
Indeed, \eqref{eq:main.SB.dual} may be rewritten as
\[
    -V_{\mathrm{SB}}^\beta(\mu,\nu)= \min_{ \substack{ \psi \in L^1(\nu), \\ \beta\text{-semiconvex}}}\left\{\int \psi_0\,\rmd\mu - \int \psi\,\rmd\nu\right\},
\]
where $\psi_0 := q_\beta \Box \bigl(-\Tcal^\beta_1[\psi]\bigr)$ is the negative $C_{\rm SB}^\beta$-transform of $\psi$; see \cref{rem:SB.c_transform}.
At the same time, by the dynamic definition of $C_{\rm SB}^\beta$ in \eqref{eq:SB.cost.dynamic}, this transform admits the control representation
\[
    \psi_0(x)=\inf_{\substack{X_0=x,\\ \rmd X_s=a_s\rmd s+b_s\rmd B_s}} \mathbb E\!\left[\int_0^1 l(a_s,b_s)\,\rmd s+\psi(X_1) \right].
\]
More generally, for $t\in[0,1]$, define
\[
    \psi_t(x):=\inf_{\substack{X_t=x,\\ \rmd X_s=a_s\rmd s+b_s\rmd B_s}} \mathbb E \!\left[\int_t^1 l(a_s,b_s)\,\rmd s + \psi(X_1) \right].
\]
Exactly as for $\psi_0$, the function $\psi_t$ admits the static representation $\psi_t = q_\beta \Box \bigl(-\Tcal^\beta_{1-t}[\psi]\bigr)$. 
In particular, $\psi_1=\psi$. The associated HJB equation is then
\begin{align}\label{eq:HJB}
    \partial_t \psi_t(x)+\inf_{(a,b)\in \R^d\times \R^{d\times d}}
    \left\{
    a\cdot \nabla_x\psi_t(x)+\frac12 \operatorname{Tr}\bigl(bb^\top \nabla_x^2 \psi_t(x)\bigr)+\frac12 |a|^2+\frac{\beta}{2}|b-I_d|_{\mathrm{HS}}^2
    \right\}=0
\end{align}
for $(t,x)\in(0,1)\times\R^d$, with terminal condition $\psi_1 = \psi$.

We remark that this semimartingale problem is not covered by the framework of \textcite{TanTouzi2013}, since the present running cost does not satisfy the coercivity assumptions imposed there. In the Schr\"odinger--Bass setting, however, the explicit static representation of $C_{\mathrm{SB}}^\beta$, together with the weak transport analysis carried out in the following sections, yields the corresponding static duality relation and existence of primal and dual optimizers for every $\beta > 0$. In particular, the family $(\psi_t)_{t\in[0,1]}$ is the value function of the associated Markovian control problem and solves the HJB equation~\eqref{eq:HJB}.

\subsection{Organization of the paper}\label{subsec:main.organization}
The paper is organized as follows: 
\Cref{sec:infimal_convolution_WOT} develops the notions of \emph{infimal convolution} and \emph{deconvolution} for weak optimal transport problems. These provide the weak optimal transport tools applied in \Cref{sec:the_SB_problem} to establish our main results on the Schrödinger--Bass problem~\eqref{eq:main.SB.dynamic}.
\Cref{sec:convergence_SB_algorithm} proves the convergence of \cref{alg:SB} under suitable conditions. 
\Cref{sec:SB_limits} shows that the Schrödinger~\eqref{eq:schroedinger_problem} and Bass~\eqref{eq:bass_problem} problems can be recovered in the limits $\beta \uparrow \infty$ and $\beta \downarrow 0$, respectively. 
Auxiliary results are collected in \Cref{sec:aux_results}.

\subsection{Notation}\label{subsec:main.notation}

\begin{itemize}[leftmargin=*,itemsep=0.25em,topsep=0.25em]
    \item Every Polish space $\mathcal X$ is equipped with a fixed complete
    separable metric $d_{\mathcal X}$ inducing its topology. Products carry
    the product topology. We write $\Pcal(\mathcal X)$ for the Borel
    probability measures on $\mathcal X$ and, for $p\in[1,\infty)$,
    \[
        \Pcal_p(\mathcal X)
        :=
        \left\{
            \mu\in\Pcal(\mathcal X):
            \int d_{\mathcal X}(x,x_0)^p\,\mu(\rmd x)<\infty
            \text{ for some }x_0\in\mathcal X
        \right\}.
    \]
    Unless specified otherwise, $\Pcal(\mathcal X)$ carries the weak topology $\tau_w$, and $\Pcal_p(\mathcal X)$ the topology induced by $\Wcal_p$.
    Other topologies are indicated by $(\Pcal(\mathcal X),\tau)$, $(\Pcal_p(\mathcal X),\tau)$.

    \item For $\mu\in\Pcal(\mathcal X)$ and $\nu\in\Pcal(\mathcal Y)$, set
    \[
        \Cpl(\mu,\nu)
        :=
        \left\{
            \pi\in\Pcal(\mathcal X\times\mathcal Y):
            \operatorname{pr}^{\mathcal X}_\#\pi=\mu,\ 
            \operatorname{pr}^{\mathcal Y}_\#\pi=\nu
        \right\},
    \]
    where $\operatorname{pr}^{\mathcal X}$ and $\operatorname{pr}^{\mathcal Y}$ are the coordinate
    projections. For $\mu,\nu\in\Pcal_p(\mathcal X)$,
    \[
        \Wcal_p^p(\mu,\nu)
        :=
        \inf_{\pi\in\Cpl(\mu,\nu)}
        \int d_{\mathcal X}(x,y)^p\,\pi(\rmd x,\rmd y).
    \]
    If $\pi\in\Cpl(\mu,\nu)$, $(\pi_x)_{x\in\mathcal X}$ denotes a disintegration with respect to its first marginal. 
    For a probability kernel $(\kappa_x)_{x\in\mathcal X}$, write
    \[
        \mu\otimes\kappa_\bullet(\rmd x,\rmd y) :=\mu(\rmd x)\kappa_x(\rmd y).
    \]

    \item For $\mu,\nu \in \Pcal_1(\Rd)$, let
    \[
        \bar\mu:=\int_{\Rd}x\,\mu(\rmd x),
        \qquad
        \Cpl_M(\mu,\nu)
        :=
        \left\{
            \pi\in\Cpl(\mu,\nu):
            \bar\pi_x=x\quad\mu\text{-a.e.}
        \right\}.
    \]
    We write $\supp(\mu)$ for the support of $\mu$.
    For $P\in\Pcal(\Pcal(\mathcal Y))$, its intensity $I(P)\in\Pcal(\mathcal Y)$
    is defined by
    \[
        I(P)(f)
        :=
        \int_{\Pcal(\mathcal Y)} \int_{\mathcal Y}f\,\rmd\rho\,P(\rmd\rho), \qquad f\in C_b(\mathcal Y).
    \]

    \item For $x\in\Rd$ and $\sigma^2>0$, $\gamma_{x;\sigma^2}$ denotes both
    the Gaussian law with mean $x$ and covariance $\sigma^2I_d$ and its
    density. Set $\gamma_x:=\gamma_{x;1}$ and $\gamma:=\gamma_0$.

    \item For $p\in[1,\infty)$, set
    \[
        L_{b,p}(\mathcal X)
        :=
        \left\{
            f:\mathcal X\to\R: \
            -K\bigl(1+d_{\mathcal X}(\cdot,x_0)^p\bigr)\le f\le K \
            \text{ for some }K>0,\ x_0\in\mathcal X
        \right\},
    \]
    and let $C_{b,p}(\mathcal X)$ be its subset of continuous functions.

    \item For $\mu,\nu\in\Pcal(\mathcal X)$, define
    \[
        H(\mu\mid\nu)
        :=
        \begin{cases}
            \displaystyle
            \int\log\!\left(\frac{\rmd\mu}{\rmd\nu}\right)\,\rmd\mu,
            & \mu\ll\nu,\\[0.4em]
            +\infty, & \text{otherwise}.
        \end{cases}
    \]
    For $\beta>0$, $x\in\Rd$, and $\mu,\nu\in\Pcal_2(\Rd)$, set
    \[
        q_\beta(x):=\frac{\beta}{2}|x|^2,
        \qquad
        W_\beta(\mu,\nu):=\frac{\beta}{2}\Wcal_2^2(\mu,\nu).
    \]

    \item For convex $\psi:\Rd\to\R\cup\{+\infty\}$, let
    \[
        \operatorname{dom}\psi:=\{\psi<+\infty\},
        \qquad
        \psi^*(y):=\sup_{x\in\Rd}\{x\cdot y-\psi(x)\}.
    \]
    Such $\psi$ is proper if $\psi>-\infty$ and $\operatorname{dom}\psi\neq\emptyset$. A concave function $\varphi$ is proper if $-\varphi$ is proper.
    
    For $f,g:\Rd\to\R\cup\{+\infty\}$,
    \[
        (f\Box g)(x):=\inf_{y\in\Rd}\{f(x-y)+g(y)\},
    \]
    while, for integrable $f,g$,
    \[
        f*g(x):=\int_{\Rd}f(x-y)g(y)\,\rmd y.
    \]
    
    \item $\operatorname{ri}(S)$ and $\operatorname{co}(S)$ denote the relative interior and convex hull of $S\subset\Rd$.
\end{itemize}

\section{Infimal convolution of weak transport problems}\label{sec:infimal_convolution_WOT}
At the core of our approach to the Schrödinger--Bass problem lies the minimization problem appearing in its variational formulation~\eqref{eq:main.SB.var},
\[
    \min_{\rho \in \cP_2(\Rd)} \left\{ V_{\rm EOT}(\alpha, \rho) + \frac{\beta}{2}\Wcal_2^2(\rho,\nu)\right\}.
\]
More generally, let $\mathcal X, \mathcal Y, \mathcal Z$ be Polish spaces, let $p \in [1,\infty)$, and let 
\[
    V: \cP_p(\mathcal X)\times \cP_p(\mathcal Y) \to \R\cup\{+\infty\},
    \qquad 
    W: \cP_p(\mathcal Y)\times \cP_p(\mathcal Z) \to \R\cup\{+\infty\}
\]
be two weak transport problems.
For $\mu \in \cP_p(\mathcal X)$, $\nu \in \cP_p(\mathcal Z)$ we call
\[
    V\Box W(\mu,\nu)
    :=
    \inf_{\rho \in \cP_p(\mathcal Y)}\bigl\{V(\mu,\rho) + W(\rho,\nu)\bigr\}
\]
the \emph{infimal convolution} of $V$ and $W$.
Under suitable assumptions, we first show that $V\Box W$ is again a weak transport problem and identify its $C$-transform, thereby obtaining an explicit dual formulation; see \cref{prop:properties_inf_conv,thm:fundamental_inf_conv}. At the end of this section, these results are used to derive the dual formulation~\eqref{eq:main.SB.dual} of the Schrödinger--Bass problem.

For the convergence analysis of \cref{alg:SB} in \cref{sec:convergence_SB_algorithm}, we further require stability of values and minimizers under perturbations of the marginals. 
Under the assumptions of \cref{thm:stability}, $V \Box W$ is continuous, and the corresponding minimizers $\rho$ depend continuously on the marginals.

We also note that similar variational structures arise in other areas. In distributionally robust optimization, the results of \textcite{blanchet2019quantifying,gao2023distributionally} lead to problems of the form
\[
    \bar I := \inf_{\lambda \geq 0} \left\{ \lambda \theta +\sup_\rho \left\{\int f \,\rmd \rho - \lambda W(\rho,\nu) \right\}\right\},
\]
where $\theta > 0$, $f \in L^1(\nu)$ is upper semicontinuous and $W$ is an optimal transport problem. Set ${V(\mu,\rho):=-\int f\,\rmd\rho}$. Then the inner problem is, up to sign, an infimal convolution of weak transport problems. 
Moreover,
\[
    - \bar I = \sup_{\alpha \in \cP_2(\R^+)} \bigl\{-\MCov(\delta_\theta,\alpha) + V\Box W(\alpha,\nu)\bigr\},
\]
which is an instance of the \emph{deconvolution} studied in \cref{subsec:deconvolution}.

Infimal convolutions also underlie variational discretizations of Wasserstein gradient flows.
Given an energy functional $\mathcal E: \cP_2(\mathcal X) \to \R\cup\{+\infty\}$, the Jordan-Kinderlehrer-Otto (JKO) scheme~\cite{jordan1998variational} recursively defines, for $h > 0$,
\[
    \rho_h^{k+1} \in \argmin_{\rho \in \cP_2(\mathcal X)}\left\{\mathcal{E}(\rho) + \frac{1}{2h}\Wcal_2^2(\rho, \rho^k_h)\right\},
\]
so that each step is an infimal convolution of $\mathcal E$ with $\frac{1}{2h}\Wcal_2^2$.
This construction is the Wasserstein analogue of the implicit Euler, or minimizing-movement, scheme and forms the basis of the general theory of gradient flows in Wasserstein and metric spaces; for a comprehensive treatment, see \textcite{ambrosio2005gradient}.

\subsection{Standard weak transport costs and structural assumptions}
We begin by recalling a standard class of cost functions for weak transport problems.

\begin{definition}[Standard weak transport costs]\label{def:std_WOT_cost}
A function $C : \mathcal X \times \Pcal_p(\mathcal Y) \to \R \cup \{+\infty\}$ is called a standard weak transport cost function if it satisfies the following properties:
\begin{enumerate}[label=(\roman*)]
    \item $C$ is lower semicontinuous;
    \item $C$ is bounded from below;
    \item for each $x \in \mathcal X$, the map $\rho \mapsto C(x,\rho)$ is convex;
    \item $C$ is proper, i.e., the effective domain of $C$,
    \[
        \operatorname{dom}(C) := \bigl\{ (x,\rho) \in \mathcal X \times \mathcal P_p(\mathcal Y) : C(x,\rho) < \infty \bigr\},
    \]
    is non-empty.
\end{enumerate}

Moreover, we say that $C$ admits a $p$-growth bound, if there exist $c>0$, $x_0 \in \mathcal{X}$, $y_0 \in \mathcal Y$ such that
\begin{equation}\label{def:p-growth.bound}
    C(x,\rho) \leq c \left( 1 + d_\mathcal{X}(x_0,x)^p + \int d_\mathcal{Y}(y_0,y)^p \,\rho(\rmd y)\right), \qquad \forall \, (x,\rho) \in \mathcal X\times\cP_p(\mathcal Y).
\end{equation}
\end{definition}

Given a standard weak transport cost $C$, the associated standard weak transport problem is
\[
    \inf_{\pi \in \cpl(\mu,\rho)}
    \int C(x,\pi_x)\,\mu(\rmd x),
    \qquad
    \mu \in \Pcal_p(\mathcal X),\ \rho \in \Pcal_p(\mathcal Y),
\]
where $(\pi_x)_{x \in \mathcal X}$ denotes a regular disintegration of $\pi$ with respect to its first marginal. 
For a function $f:\mathcal Y \to \R$, we denote by
\[
    f^C(x) := \inf_{\substack{\rho \in \cP_p(\mathcal Y), \\ f \in L^1(\rho)}} \left\{ C(x,\rho) - \int f \,\rmd\rho \right\}
\]
its corresponding $C$-transform.
In the following, unless explicitly stated otherwise, $V$ and $W$ denote weak transport problems with cost functions
\[
    C_V : \mathcal X \times \Pcal_p(\mathcal Y) \to \R \cup \{+\infty\},
    \qquad
    C_W : \mathcal Y \times \Pcal_p(\mathcal Z) \to \R \cup \{+\infty\}.
\]

We first address the question under which conditions $V\Box W$ is again a standard weak transport problem, and which regularity properties the induced cost function inherits.
A principal issue is compactness in the minimization over $\rho$. 
For example, when studying the variational formulation of the Schrödinger--Bass problem~\eqref{eq:main.SB.var}, we have $\mathcal Y=\mathcal Z=\Rd$ and $W=\frac\beta2\Wcal_2^2$. Bounds on $W$ then give bounds on the second moments of $\rho$. 
On noncompact spaces, however, this is not enough for relative compactness in $(\cP_p(\mathcal Y),\Wcal_p)$. Indeed, convergence in $\Wcal_p$ is equivalent to weak convergence together with convergence of the $p$-th moments. 
The same bounds do give relative compactness in weaker Wasserstein topologies, for instance in $\Wcal_r$ for $r<p$. This motivates the following coercivity
assumption for $C_V+W$.

\begin{definition}[Coercivity and continuity assumptions]\label{def:coerc_cont_aspts}
    Let $C_V:\mathcal X\times\cP_p(\mathcal Y)\to\R\cup\{+\infty\}$ be a standard weak transport cost and let $W:\cP_p(\mathcal Y)\times\cP_p(\mathcal Z)\to\R\cup\{+\infty\}$ be a standard weak transport problem.
    
    \begin{enumerate}
        \asptitem{Crc}{Crc} \label{aspt:coercivity} \textbf{Coercivity assumption.}
        We say that $(C_V,W)$ satisfies the coercivity assumption for $r \in [1,p]$ if for each compact $K \subseteq \mathcal X \times \cP_p(\mathcal Z)$, the map
        \[
            \rho \longmapsto \inf_{(x,\eta)\in K}\bigl\{C_V(x,\rho)+W(\rho,\eta)\bigr\}
        \]
        has compact sublevel sets in $(\cP_p(\mathcal Y),\Wcal_r)$. 
        
        \asptitem{Cnt}{Cnt} \label{aspt:continuity} \textbf{Continuity assumption.}
        We say that $(C_V,W)$ satisfies the continuity assumption if, for every $\rho$ with $(x,\rho) \in {\rm dom}(C_V)$ for some $x \in \mathcal X$, the maps $x \mapsto C_V(x,\rho)$ and $\eta \mapsto W(\rho,\eta)$ are continuous on $\mathcal X$ and on $\cP_p(\mathcal Z)$, respectively.
    \end{enumerate}
\end{definition}

Since $W$ is a standard weak transport problem, $W$ is lower semicontinuous on $(\mathcal{P}_p(\mathcal Y),\tau_w)\times \cP_p(\mathcal Z)$, and hence also on $(\mathcal{P}_p(\mathcal Y),\Wcal_r)\times \cP_p(\mathcal Z)$ for all $r \in [1,p]$; see \cite[Theorem 2.9]{backhoff2019existence}. Therefore the map 
\[
    (x,\rho,\eta) \mapsto C_V(x,\rho) + W(\rho,\eta)
\]
is lower semicontinuous on $\mathcal X \times \mathcal{P}_p(\mathcal Y) \times \cP_p(\mathcal Z)$.
In general, however, $C_V$ need not be lower semicontinuous on $\mathcal X \times (\mathcal{P}_p(\mathcal Y),\Wcal_r)$, and so $C_V+W$ may not be lower semicontinuous on $\mathcal X \times (\mathcal{P}_p(\mathcal Y),\Wcal_r) \times \cP_p(\mathcal Z)$. The following lemma shows that this lower semicontinuity nevertheless follows from \Cref{aspt:coercivity}.

\begin{lemma}\label{lem:coercivity_gives_lsc}
    Let $(C_V,W)$ satisfy the coercivity assumption~\Cref{aspt:coercivity} for some $r \in [1,p]$. Then
    \[
        (x,\rho,\eta) \longmapsto C_V(x,\rho) + W(\rho,\eta)
    \]
    is lower semicontinuous on $\mathcal X \times (\mathcal{P}_p(\mathcal Y), \Wcal_r) \times \cP_p(\mathcal Z)$. Moreover,
    \[
        (\mu,\rho,\nu) \longmapsto V(\mu,\rho) + W(\rho,\nu)
    \]
    is lower semicontinuous on $(\cP_p(\mathcal X), \tau_w) \times (\mathcal{P}_p(\mathcal Y), \Wcal_r) \times \cP_p(\mathcal Z)$.
\end{lemma}
\begin{proof}
    We defer the proof of this lemma to the appendix; see \cref{lem:appendix:coercivity_gives_lsc}.
\end{proof}

\subsection{Structural properties and duality}\label{subsec:struc_props_and_duality}
We now turn to the first structural results for infimal convolutions and their dual formulation. 
The following result establishes that, under the coercivity assumption~\Cref{aspt:coercivity}, the infimal convolution of two weak transport problems is again a weak transport problem, with a naturally induced cost function obtained by pointwise infimal convolution. 
It also gives the corresponding dual representation.

\begin{proposition}[Properties of the infimal convolution]\label{prop:properties_inf_conv}
    Let $V$, $W$ be proper standard weak transport problems and assume that $(C_V,W)$ satisfies \Cref{aspt:coercivity} for some $r \in [1,p]$.
    Define ${C_{V\Box W}:\mathcal X\times\cP_p(\mathcal Z)\to\R\cup\{+\infty\}}$ by
    \begin{equation}\label{eq:CVBoxW}
        C_{V \Box W}(x,\eta) := \inf_{\rho \in \cP_p(\mathcal Y)} \bigl\{ C_V(x,\rho) + W(\rho,\eta) \bigr\}.
    \end{equation}
    Then $C_{V\Box W}$ is a standard weak transport cost, and the infimum in \eqref{eq:CVBoxW} is attained for every $(x,\eta)\in \mathcal X\times\cP_p(\mathcal Z)$. 
    If, in addition, $(C_V,W)$ satisfies the continuity assumption~\Cref{aspt:continuity}, then $C_{V\Box W}$ is continuous on $\mathcal X\times\cP_p(\mathcal Z)$.
    
    Moreover, for every $\mu \in \cP_p(\mathcal X)$ and
    $\nu \in \cP_p(\mathcal Z)$,
    \begin{equation} \label{eq:VBoxW_primal}
        V \Box W(\mu,\nu)
        =
        \inf_{\pi \in \cpl(\mu,\nu)} \int_{\mathcal X} C_{V \Box W}(x,\pi_x)\, \mu(\rmd x).
    \end{equation}
    Finally, for every $f\in L_{b,p}(\mathcal Z)$,
    \begin{equation} \label{eq:VBoxW_conjugate}
        f^{C_{V\Box W}}
        \ge 
        \bigl(-f^{C_W}\bigr)^{C_V},
    \end{equation}
    and
    \begin{equation} \label{eq:VBoxW_dual}
        V \Box W(\mu,\nu)
        =
        \sup_{f \in L_{b,p}(\mathcal Z)} \left\{ \int_{\mathcal Z} f(z)\,\nu(\rmd z) + \int_{\mathcal X} \bigl(-f^{C_W}\bigr)^{C_V}(x)\,\mu(\rmd x) \right\}.
    \end{equation}
\end{proposition}

\begin{proof}   
    \emph{Part 1: $C_{V\Box W}$ is a standard weak transport cost.} \\
    Boundedness from below and convexity are inherited from $C_V$ and $W$, and so it remains to show lower semicontinuity. Let $(x_k,\eta_k) \to (x^\circ,\eta^\circ)$ in $\mathcal X \times \cP_p(\mathcal Z)$. 
    If $C_{V\Box W}(x_k,\eta_k) \to \infty$ there is nothing to prove. Passing to a subsequence, we hence assume that $C_{V\Box W}(x_k,\eta_k) \to c \in \R$.
    For each $k\in\N$, choose $\rho_k\in\cP_p(\mathcal Y)$ such that
    \[
        C_V(x_k,\rho_k)+W(\rho_k,\eta_k) \le C_{V\Box W}(x_k,\eta_k)+\frac1k,
    \]
    and set $K:=\{(x^\circ,\eta^\circ)\}\cup\{(x_k,\eta_k):k\in\N\}$. By \Cref{aspt:coercivity}, the set
    \[
        \Bigl\{\rho\in\cP_p(\mathcal Y):\inf_{(x,\eta)\in K}\{C_V(x,\rho)+W(\rho,\eta)\}\le c+1\Bigr\}
    \]
    is compact in $(\cP_p(\mathcal Y),\Wcal_r)$. Therefore, along a subsequence, we have $\rho_k \to \rho^\circ$ in $(\cP_p(\mathcal Y), \Wcal_r)$. By \Cref{lem:coercivity_gives_lsc},
    \begin{equation}\label{eq:prop:VBoxW_is_lsc}
        C_{V\Box W}(x^\circ,\eta^\circ)
        \le 
        C_V(x^\circ,\rho^\circ)+W(\rho^\circ,\eta^\circ) 
        \le 
        \liminf_{k\to\infty} \{C_V(x_k,\rho_k)+W(\rho_k,\eta_k)\} = c.
    \end{equation}
    Thus $C_{V\Box W}$ is lower semicontinuous and the same argument yields also attainment of the infimum in \eqref{eq:CVBoxW}.
    
    If, in addition, the pair $(C_V,W)$ satisfies \Cref{aspt:continuity}, then
    \[
        \limsup_{k \to \infty} C_{V \Box W}(x_k,\eta_k) \le \limsup_{k \to \infty} C_V(x_k,\rho^\circ) + W(\rho^\circ,\eta_k)= C_V(x^\circ,\rho^\circ) + W(\rho^\circ,\eta^\circ).
    \]
    Together with \eqref{eq:prop:VBoxW_is_lsc}, this proves that $C_{V \Box W}$ is continuous on $\mathcal{X}\times\Pcal_p(\mathcal{Z})$.

    \emph{Part 2: primal representation.} \\
    Fix $\delta>0$ and $y_0\in\mathcal Y$, and define
    \[
        W^\delta(\rho,\nu):=W(\rho,\nu)+\delta\int_\mathcal{Y} d_{\mathcal Y}(y_0,y)^p\,\rho(\rmd y),
        \qquad
        C_{W^\delta}(y,\eta):=C_W(y,\eta)+\delta\,d_{\mathcal Y}(y_0,y)^p.
    \]
    We prove \eqref{eq:VBoxW_primal} first with $W$ replaced by $W^\delta$. By definition and disintegration,
    \[
        V\Box W^\delta(\mu,\nu)
        =
        \inf_{\rho\in\cP_p(\mathcal Y)}
        \inf_{\pi\in\cpl(\mu,\rho),\,\pi'\in\cpl(\rho,\nu)}
        \int_{\mathcal X}
        \left(
        C_V(x,\pi_x)+\int_{\mathcal Y} C_{W^\delta}(y,\pi'_y)\,\pi_x(\rmd y)
        \right)\mu(\rmd x).
    \]
    For $\pi$ and $\pi'$ admissible, set $\hat\pi_x:=\int_{\mathcal Y}\pi'_y\,\pi_x(\rmd y) \in \mathcal{P}(\mathcal{Z})$ such that $\hat\pi:=\mu\otimes\hat\pi_\bullet\in\cpl(\mu,\nu)$.
    Then
    \[
        \int_{\mathcal Y} C_{W^\delta}(y,\pi'_y)\,\pi_x(\rmd y) \ge W^\delta(\pi_x,\hat\pi_x),
    \]
    and therefore
    \[ 
        V\Box W^\delta(\mu,\nu) 
        \ge 
        \inf_{\hat\pi\in\cpl(\mu,\nu)}\int_{\mathcal X} 
        \inf_{\rho \in \cP_p(\mathcal Y)}\left\{
        C_V(x,\rho)+W^\delta(\rho,\hat{\pi}_x)
        \right\}\,\mu(\rmd x) 
        = 
        \inf_{\hat\pi\in\cpl(\mu,\nu)}
        \int_{\mathcal X} C_{V\Box W^\delta}(x,\hat\pi_x)\,\mu(\rmd x). 
    \]
    
    Conversely, fix $\varepsilon>0$. The claim is immediate unless there exists
    $\hat\pi\in\cpl(\mu,\nu)$ such that
    \[
        \int_{\mathcal X} C_{V\Box W^\delta}(x,\hat\pi_x)\,\mu(\rmd x)<\infty.
    \]
    We fix such a $\hat\pi$ and choose a universally measurable $\varepsilon$-selector
    $K:\mathcal X\times\cP_p(\mathcal Z)\to\cP_p(\mathcal Y)$ for $C_{V\Box W^\delta}$ (see e.g.~\cite[Proposition 7.50]{BertsekasShreve1978}),
    that is,
    \begin{equation}\label{prop:VboxW:proof:primal_repr:eps_selector}
        \begin{aligned}
            C_{V\Box W^\delta}(x,\eta)+\varepsilon
            &\ge
            C_V(x,K(x,\eta))+W^\delta(K(x,\eta),\eta) \\
            &= C_V(x,K(x,\eta))+W(K(x,\eta),\eta)+\delta\,\int_\mathcal{Y}d_\mathcal{Y}(y_0,y)^p\,K(x,\eta)(\rmd y)
        \end{aligned}
    \end{equation}
    for all $(x,\eta)\in\mathcal X\times\cP_p(\mathcal Z)$.
    Set $\pi_x:=K(x,\hat\pi_x)$, $\pi:=\mu\otimes\pi_\bullet$ and $\rho:=\int \pi_x\,\mu(\rmd x)$.
    Since $\int_\mathcal{X} C_{V\Box W^\delta}(x,\hat\pi_x)\,\mu(\rmd x)$ is finite, \eqref{prop:VboxW:proof:primal_repr:eps_selector} implies that $\rho\in\cP_p(\mathcal Y)$. Moreover, we have
    \[
    \int_{\mathcal X} C_{V\Box W^\delta}(x,\hat\pi_x)\,\mu(\rmd x)+\varepsilon
    \ge
    \int_{\mathcal X}\bigl(C_V(x,\pi_x)+W^\delta(\pi_x,\hat\pi_x)\bigr)\,\mu(\rmd x).
    \]
    Since $\nu=\int \hat\pi_x\,\mu(\rmd x)$, convexity (see \Cref{lem:convexity_WOT}) and lower semicontinuity of $(\rho,\nu)\mapsto W^\delta(\rho,\nu)$ give
    \[
    W^\delta(\rho,\nu)\le \int_{\mathcal X} W^\delta(\pi_x,\hat\pi_x)\,\mu(\rmd x).
    \]
    It follows that
    \[
    \int_{\mathcal X} C_{V\Box W^\delta}(x,\hat\pi_x)\,\mu(\rmd x)+\varepsilon
    \ge
    \int_{\mathcal X} C_V(x, \pi_x)\,\mu(\rmd x)+W^\delta(\rho,\nu)
    \ge
    V\Box W^\delta(\mu,\nu).
    \]
    Taking the infimum over $\hat\pi\in\cpl(\mu,\nu)$ and letting $\varepsilon\downarrow0$, we obtain
    \[
    V\Box W^\delta(\mu,\nu)
    =
    \inf_{\pi\in\cpl(\mu,\nu)}
    \int_{\mathcal X} C_{V\Box W^\delta}(x,\pi_x)\,\mu(\rmd x)
    \]
    for all $\delta > 0$.
    Finally, we have $W^\delta\downarrow W$ pointwise as $\delta\downarrow0$, and therefore $V\Box W^\delta\downarrow V\Box W$ and $C_{V\Box W^\delta}\downarrow C_{V\Box W}$ pointwise. 
    Hence, for every $\pi\in\cpl(\mu,\nu)$,
    \[
        \int_{\mathcal X} C_{V\Box W^\delta}(x,\pi_x)\,\mu(\rmd x)
        \mathrel{\Big\downarrow}
        \int_{\mathcal X} C_{V\Box W}(x,\pi_x)\,\mu(\rmd x),
    \]
    and taking infima over $\pi\in\cpl(\mu,\nu)$ yields \eqref{eq:VBoxW_primal}.
    
    \emph{Part 3: conjugacy and duality.} \\
    Let $f \in L_{b,p}(\mathcal Z)$. Then
    \begin{align*}
        f^{C_{V\Box W}}(x)
        &=
        \inf_{\rho\in\cP_p(\mathcal Y),\,\eta\in\cP_p(\mathcal Z)}
        \left\{C_V(x,\rho)+W(\rho,\eta)-\int f \,\rmd\eta\right\}\\
        &=
        \inf_{\rho\in\cP_p(\mathcal Y)}
        \left\{
        C_V(x,\rho)
        +
        \inf_{\pi'\in\Cpl_p(\rho,*)}
        \iint_{\mathcal{Y}\times\mathcal{Z}} \bigl(C_W(y,\pi'_y)-f(z)\bigr)\,\pi'(\rmd y,\rmd z)
        \right\},
    \end{align*}
    where $\Cpl_p(\rho,*) := \{\pi\in\Cpl(\rho,\eta):\eta\in\cP_p(\mathcal Z)\}$. Fix $\rho\in\cP_p(\mathcal Y)$. For every $\pi'\in\Cpl_p(\rho,*)$, we have
    \[
        f^{C_W}(y) =
        \inf_{\eta\in\cP_p(\mathcal Z)}
        \left\{C_W(y,\eta)-\int_\mathcal{Z} f\,\rmd\eta\right\}
        \le
        C_W(y,\pi'_y)-\int_\mathcal{Z} f\,\rmd\pi'_y,
    \]
    for every $y \in \mathcal Y$.
    As $f^{C_W}$ is bounded from below, we can integrate both sides and get
    \begin{equation}\label{prop:VboxW:proof:c-conjugate:easy-dir}
        \int_\mathcal{Y} f^{C_W}\,\rmd\rho
        \le
        \inf_{\pi'\in\Cpl_p(\rho,*)}
        \iint_{\mathcal{Y}\times\mathcal{Z}} \bigl(C_W(y,\pi'_y)-f(z)\bigr)\,\pi'(\rmd y,\rmd z).
    \end{equation}
    Consequently, $f^{C_{V\Box W}} \ge (-f^{C_W})^{C_V}$.

    For $\delta > 0$ set $f_\delta := f - \delta\,d_\mathcal{Z}(z_0,\cdot)^p \in L_{b,p}(\mathcal{Z})$ for some $z_0 \in \mathcal Z$. We claim that, for all $\rho \in \cP_p(\mathcal{Y})$, 
    \[
        \int_\mathcal{Y} f_\delta^{C_W}\,\rmd\rho
        =
        \inf_{\pi'\in\Cpl_p(\rho,*)}
        \iint_{\mathcal{Y}\times\mathcal{Z}} \bigl(C_W(y,\pi'_y)-f_\delta(z)\bigr)\,\pi'(\rmd y,\rmd z),
    \]
    and hence $f_\delta^{C_{V\Box W}} = (-f_\delta^{C_W})^{C_V}$ for all $\delta > 0$.
    
    If $\int_\mathcal{Y} f_\delta^{C_W}\,\rmd\rho=+\infty$ we have equality in \eqref{prop:VboxW:proof:c-conjugate:easy-dir}, and so we assume that ${\int_\mathcal{Y} f_\delta^{C_W}\,\rmd\rho<\infty}$.
    Since $f_\delta$ is bounded from above, the map $(y,\eta)\mapsto C_W(y,\eta)-\int_\mathcal{Z} f_\delta\,\rmd\eta$ is Borel measurable on $\mathcal Y\times\cP_p(\mathcal Z)$. 
    Measurable selection therefore yields, for every $\varepsilon>0$, a universally measurable map ${\varphi^\varepsilon\colon \mathcal Y\to\cP_p(\mathcal Z)}$ such that, for all $y \in \mathcal{Y}$,
    \begin{equation}\label{prop:VboxW:proof:c-transform:eps_selector}
        C_W\bigl(y,\varphi^\varepsilon(y)\bigr)
        -\int_\mathcal{Z} f_\delta(z)\,\varphi^\varepsilon(y,\rmd z)
        \le
        f_\delta^{C_W}(y)+\varepsilon.
    \end{equation}
    Since $f_\delta$ is bounded from above $\int_\mathcal{Y} f_\delta^{C_W}\,\rmd\rho<\infty$, this yields
    \[
        \delta\int_\mathcal{Y}\!\int_\mathcal{Z} d_{\mathcal Z}(z_0,z)^p\,\varphi^\varepsilon(y,\rmd z)\,\rho(\rmd y)<\infty.
    \]
    Hence, for $\pi^\varepsilon:=\rho\otimes\varphi^\varepsilon_\bullet$, we have $\pi^\varepsilon\in\Cpl_p(\rho,*)$, and
    \begin{align*}
        \iint_{\mathcal{Y}\times\mathcal{Z}} \bigl(C_W(y,\pi^\varepsilon_y)-f_\delta(z)\bigr)\,\pi^\varepsilon(\rmd y,\rmd z)
        &=
        \int_\mathcal{Y} \left(C_W\bigl(y,\varphi^\varepsilon(y)\bigr)-\int_\mathcal{Z} f_\delta(z)\,\varphi^\varepsilon(y,\rmd z)\right)\rho(\rmd y)\\
        &\le
        \int_\mathcal{Y} f_\delta^{C_W}\,\rmd\rho+\varepsilon.
    \end{align*}
    Taking the infimum over $\pi'\in\Cpl_p(\rho,*)$ and then letting $\varepsilon\downarrow0$, we obtain the reverse inequality in \eqref{prop:VboxW:proof:c-conjugate:easy-dir}. Consequently, for all $\delta > 0$,
    \[
        f_\delta^{C_{V\Box W}}(x)
        =
        \inf_{\rho\in\cP_p(\mathcal Y)}
        \left\{
        C_V(x,\rho)+\int_\mathcal{Y} f_\delta^{C_W}(y)\,\rho(\rmd y)
        \right\}
        =
        (-f_\delta^{C_W})^{C_V}(x).
    \]
    This shows \eqref{eq:VBoxW_conjugate}. 

    For $\delta > 0$, define
    \[
        W_\delta(\rho,\nu) := W(\rho,\nu) + \delta \int h\,\rmd\nu,
        \qquad
        C_{W_\delta}(y,\eta) := C_W(y,\eta) + \delta \int h\,\rmd\eta. 
    \]
    Then $f_\delta^{C_W}= f^{C_{W_\delta}}$, and strong duality (see \cite[Theorem 3.1]{backhoff2019existence}) gives
    \begin{align*}
        V\Box W_{\delta}(\mu,\nu) &= \sup_{f \in L_{b,p}(\mathcal Z)} \left\{\int f\,\rmd \nu + \int (-f_\delta^{C_{W}})^{C_V}\,\rmd\mu\right\},
        \\
        V\Box W (\mu,\nu) &= \sup_{f \in L_{b,p}(\mathcal Z)} \left\{\int f\,\rmd \nu + \int f^{C_{V \Box W}}\,\rmd\mu\right\}
        \\
        &\ge \sup_{f \in L_{b,p}(\mathcal Z)} \left\{\int f\,\rmd \nu + \int (-f^{C_W})^{C_V}\,\rmd\mu\right\}.
    \end{align*}
    To show the converse inequality, let $(\delta_n)_{n\in\N}$ with $\delta_n \downarrow 0$ and $f_{n} \in L_{b,p}(\mathcal Z)$ such that, for all $n \in \N$,
    \[
        V\Box W_{\delta_n}(\mu,\nu) \le \frac1n + \int f_{\delta_n,n}\,\rmd\nu + \int (-f_{\delta_n,n}^{C_{W}})^{C_V}\,\rmd\mu,
    \]
    where $f_{\delta_n,n} := f_n - \delta_n\,d_\mathcal{Z}(z_0,\cdot)^p \in L_{b,p}(\mathcal{Z})$.
    By the monotonicity of the infimum, $V\Box W_{\delta_n}(\mu,\nu) \downarrow V\Box W(\mu,\nu)$ and so
    \[
        V\Box W(\mu,\nu) \le \lim_{n \to \infty} \left(\int f_{\delta_n,n}\,\rmd\nu + \int (-f_{\delta_n,n}^{C_{W}})^{C_V}\,\rmd\mu\right).
    \]
    Therefore, \eqref{eq:VBoxW_dual} holds.
\end{proof}

\begin{remark}\label{rem:properties_inf_conv:additional_remarks}\leavevmode
    \begin{enumerate}
        \item In the setting of \Cref{prop:properties_inf_conv}, fix $(x,\eta) \in \mathcal X \times \cP_p(\mathcal Z)$ and assume that the map
        \[
            \rho \longmapsto C_V(x,\rho) + W(\rho,\eta)
        \]
        is strictly convex on $\cP_p(\mathcal Y)$. Then the minimizer attaining $C_{V \Box W}(x,\eta)$ is unique.

        \item The dual representation \eqref{eq:VBoxW_dual} may equivalently be written with $C_{b,p}(\mathcal Z)$ in place of $L_{b,p}(\mathcal Z)$, by the standard weak transport duality applied to the standard weak transport cost $C_{V\Box W}$.

        \item As is clear from the proof of \Cref{prop:properties_inf_conv}, convexity of $\rho \mapsto C_V(x,\rho)$ is not used in the derivation of \eqref{eq:VBoxW_primal}--\eqref{eq:VBoxW_dual}. Hence these statements remain valid under assumptions on $C_V$ weaker than those of a standard weak transport cost, even though $V\Box W$ need not itself be a standard weak transport problem.
    \end{enumerate}
\end{remark}

\Cref{prop:properties_inf_conv} identifies $V\Box W$ as a standard weak optimal transport problem with cost $C_{V\Box W}$. The following lemma gives a compactness property of minimizing sequences in $\cP_p(\mathcal Y)$ for $V\Box W$ which will be used in the stability analysis in \cref{subsec:stability}. We emphasize that tightness alone does not imply attainment of the infimum in $\cP_p(\mathcal Y)$.

\begin{lemma}[Tightness of minimizing sequences]\label{lem:VBoxW.tightness.min.sequ}
    In the setting of \Cref{prop:properties_inf_conv}, assume in addition that $C_{V\Box W}$ admits a $p$-growth bound. Then every minimizing sequence $(\rho_k)_{k \in \N} \subset \cP_p(\mathcal Y)$ for
    \[
        V \Box W(\mu,\nu) = \inf_{\rho \in \cP_p(\mathcal Y)} \left\{V(\mu,\rho)+ W(\rho,\nu)\right\}
    \]
    is tight.
\end{lemma}

\begin{proof}
    Let $(\rho_k)_{k \in \N} \subset \mathcal \cP_p(\mathcal Y)$ be such that $V(\mu,\rho_k) + W(\rho_k,\nu) \to V\Box W(\mu,\nu)$.
    To show that this sequence is tight, we define measures $(\zeta^k)_{k \in \N}$ in ${\cP\bigl(\mathcal X \times (\cP_p(\mathcal Y), \Wcal_r)\times \cP_p(\mathcal Z)\bigr)}$ as follows: Fix $k \in \N$, and let $\xi^k \in \Cpl(\mu,\rho_k)$ and $\chi^k \in \Cpl(\rho_k,\nu)$ be optimal for $V(\mu,\rho_k)$ and $W(\rho_k,\nu)$, respectively. Set 
    \[
        \Gamma^k(\rmd x, \rmd y, \rmd z) := \xi^k(\rmd x, \rmd y)\,\chi^k_y(\rmd z), 
        \qquad
        \pi^k := (\operatorname{pr}^{\mathcal X},\operatorname{pr}^{\mathcal Z})_\#\Gamma^k \in \Cpl(\mu,\nu).
    \]
    Then
    \begin{equation}\label{ineq:proof:VBoxW.tightness.min.sequ}
        \begin{aligned}
            V\Box W(\mu,\nu)
            &\le
            \int C_{V\Box W}(x,\pi_x^k)\,\mu(\rmd x)
            \le
            \int C_V(x,\xi_x^k) + W(\xi_x^k,\pi^k_x)\,\mu(\rmd x) 
            \\
            &\le
            \int \left(C_V(x,\xi_x^k) + \int C_W(y,\chi^k_y)\,\xi_x^k(\rmd y)\right)\,\mu(\rmd x)
            = V(\mu,\rho_k) + W(\rho_k,\nu).
        \end{aligned}
    \end{equation}
    For $\mu$-a.e.\ $x$, $C_{V\Box W}(x,\pi^k_x) \le C_V(x,\xi_x^k) + W(\xi_x^k,\pi^k_x)$, and so \eqref{ineq:proof:VBoxW.tightness.min.sequ} gives
    \[
        0 = \lim_{k \to \infty} C_V(\cdot,\xi_\cdot^k) + W(\xi_\cdot^k,\pi^k_\cdot) - C_{V\Box W}(\cdot,\pi^k_\cdot)\quad \text{ in } L^1(\mu).
    \]
    Finally, set
    \[
        \zeta^k := \bigl(x \mapsto (x, \xi^k_x,\pi^k_x)\bigr)_\#\mu,
    \]
    and note that $\rho_k = \int \operatorname{pr}^\mathcal{Y}_\#\zeta^k_x \,\mu(\rmd x)$, $\nu = \int \operatorname{pr}^\mathcal{Z}_\#\zeta^k_x \,\mu(\rmd x)$.
    
    We claim that $(\zeta^k)_{k \in \N}$ is tight. Fix $\varepsilon > 0$. By \cite[Lemma 2.4]{backhoff2019existence}, the sequence $\bigl((x \mapsto \pi_x^k)_\#\mu\bigr)_{k \in \N}$ is precompact in $\cP_p(\cP_p(\mathcal Z))$ since $\{\nu\}$ is precompact in $\cP_p(\mathcal Z)$. 
    Thus there exist compact sets $K_1 \subseteq \mathcal X$ and $K_3 \subseteq \cP_p(\mathcal Z)$ such that
    \[
        \inf_{k \in \N} \zeta^k\bigl(K_1 \times \cP_p(\mathcal Y) \times K_3\bigr) \ge 1 - \frac\varepsilon2.
    \]
    Moreover, there exists $\delta > 0$ such that, for every $k \in \N$,
    \[
        \mu\!\left(\left\{x : C_V(x,\xi_x^k) + W(\xi_x^k,\pi^k_x) - C_{V\Box W}(x,\pi^k_x) > \delta\right\} \right) \le \frac{\varepsilon}{2}.
    \]
    Since $C_{V\Box W}$ admits a $p$-growth bound, the set
    \[
        K_2 := \left\{ \rho \in \cP_p(\mathcal Y): \inf_{(x,\eta) \in K_1 \times K_3} \bigl\{ C_V(x,\rho) + W(\rho,\eta) \bigr\} \le \sup_{(x,\eta) \in K_1 \times K_3} C_{V\Box W}(x,\eta) + \delta \right\}
    \]
    is compact in $(\cP_p(\mathcal Y),\Wcal_r)$ by \Cref{aspt:coercivity}. In particular, on the set
    \[
        \left\{x \in K_1, \ \pi_x^k \in K_3, \ C_V(x,\xi_x^k) + W(\xi_x^k,\pi^k_x) - C_{V\Box W}(x,\pi^k_x) \le \delta\right\}
    \]
    we have $\xi_x^k\in K_2$. Consequently,
    \[
        \inf_{k \in \N} \zeta^k(K_1 \times K_2 \times K_3) \ge 1 - \epsilon,
    \]
    and so $(\zeta^k)_{k\in\N}$ is tight. It follows from \cite[Lemma 2.3]{backhoff2019existence} that $(\rho_k)_{k\in\N}$ is tight.
\end{proof}

The preceding results give tractable conditions under which the infimal convolution of weak transport problems is again a standard weak transport problem and its induced cost is continuous. 
A key point is that the associated $C$-transform is obtained by composing the $C$-transforms of the two underlying problems. This makes the infimal convolution particularly useful in the sequel. 
We record these consequences in the form in which they will be used to study the variational formulation of the Schrödinger--Bass problem~\eqref{eq:main.SB.var}.

\begin{theorem}[Fundamental theorem of the infimal convolution]\label{thm:fundamental_inf_conv}
    In the setting of \Cref{prop:properties_inf_conv}, assume in addition that $C_{V\Box W}$ is continuous and admits a $p$-growth bound.
    Then, the following hold for every $\mu \in \cP_p(\mathcal X)$ and $\nu\in\cP_p(\mathcal{Z})$:
    \begin{enumerate}[label = (\roman*)]
        \item \label{it:thm.fundamental.1} \textbf{Primal attainment.} The infimal convolution $V \Box W$ is a standard WOT problem, and the infimum is attained, that is,
        \[
            V\Box W(\mu,\nu) = \min_{\pi \in \cpl(\mu,\nu)} \int_\mathcal{X} C_{V \Box W}(x,\pi_x) \,\mu(\rmd x).
        \]
        
        \item \label{it:thm.fundamental.2} \textbf{Strong duality.} $V \Box W$ admits the dual representation
        \begin{align*}
             V \Box W(\mu,\nu)
             &=\sup_{f \in C_{b,p}(\mathcal Z)} \left\{\int_\mathcal{Z} f \, \rmd\nu + \int_\mathcal{X} (-f^{C_W})^{C_V} \,\rmd\mu \right\}.
        \end{align*}
    \end{enumerate}
    If, in addition, $W(\rho,\nu) < +\infty$ for all $\rho \in \cP_p(\mathcal Y)$, then:
    \begin{enumerate}[label = (\roman*),resume*]
        \item \label{it:thm.fundamental.3} \textbf{Dual attainment.} The dual problem is attained, that is,
        \begin{align*}
             V \Box W(\mu,\nu)
             &= \max_{f \in L^1(\nu)} \left\{\int_\mathcal{Z} f \, \rmd\nu + \int_\mathcal{X} (-f^{C_W})^{C_V} \, \rmd\mu\right\}.
        \end{align*}
        
        \item \label{it:thm.fundamental.4} \textbf{Complementary slackness.} Let $\rho^\circ \in \cP_p(\mathcal Y)$ and $f^\circ \in L^1(\nu)$. Then $\rho^\circ$ is optimal for
        \[
            \inf_{\rho \in \cP_p(\mathcal Y)} \bigl\{V(\mu,\rho)+W(\rho,\nu)\bigr\}
        \]
        and $f^\circ$ is optimal for the dual problem in \ref{it:thm.fundamental.3} if and only if
        \begin{equation}\label{eq:thm.fundamental.complementary.slackness}
            \begin{aligned}
                V(\mu,\rho^\circ) &= \int_\mathcal{Y} -(f^\circ)^{C_W} \,\rmd\rho^\circ + \int_\mathcal{X} (-(f^\circ)^{C_W})^{C_V} \, \rmd\mu, \\
                W(\rho^\circ,\nu) &= \int_\mathcal{Y} (f^\circ)^{C_W} \, \rmd\rho^\circ + \int_\mathcal{Z} f^\circ \,\rmd\nu.
            \end{aligned}
        \end{equation}
    \end{enumerate}
\end{theorem}

\begin{proof}
    Since $V \Box W$ admits a standard weak transport cost $C_{V\Box W}$, primal attainment follows from \cite[Theorem 2.9]{backhoff2019existence}.
    By \Cref{prop:properties_inf_conv}, the problem $V \Box W$ admits the dual representation stated in \ref{it:thm.fundamental.2}. 
    
    To show \ref{it:thm.fundamental.3}, assume in addition $W(\rho,\nu) < \infty$ for all $\rho \in \cP_p(\mathcal Y)$.
    By \cite[Theorem 1.2]{beiglböck2025fundamentaltheoremweakoptimal}, there exists $f\in L^1(\nu)$ such that 
    \[ 
        V\Box W(\mu,\nu) = \int f\,\rmd\nu + \int f^{C_{V\Box W}}\,\rmd\mu . 
    \] 
    For all $x\in\mathcal X$, 
    \[ 
        f^{C_{V\Box W}}(x) \le C_{V\Box W}(x,\nu) - \int f\,\rmd\nu,
    \] 
    and since $C_{V\Box W}$ satisfies a $p$-growth bound, $f^{C_{V\Box W}}\in[-\infty,\infty)$.
    Fix $z_0\in\mathcal Z$. For $\delta>0$, define 
    \[
        W_\delta(\rho,\eta) := W(\rho,\eta) + \delta\int d_\mathcal Z(z_0,z)^p\,\eta(\rmd z), 
        \qquad 
        C_{W_\delta}(y,\eta) := C_W(y,\eta) + \delta\int d_\mathcal Z(z_0,z)^p\,\eta(\rmd z). 
    \] 
    As in the proof of \Cref{prop:properties_inf_conv}, we have $f^{C_{V\Box W_\delta}} = (-f^{C_{W_\delta}})^{C_V}$.
    Moreover, with $f_\delta := f-\delta \, d_\mathcal Z(z_0,\cdot)^p$ we have $f_\delta^{C_W}=f^{C_{W_\delta}}$, and hence 
    \[ 
        f^{C_{V\Box W_\delta}} = (-f_\delta^{C_W})^{C_V}. 
    \]
    Therefore, for every $\delta>0$, 
    \begin{align*}
        V\Box W_\delta(\mu,\nu) 
        &= 
        \max_{g\in L^1(\nu)} \left\{ \int g\,\rmd\nu + \int g^{C_{V\Box W_\delta}}\,\rmd\mu \right\} 
        \ge 
        \int f\,\rmd\nu + \int (-f_\delta^{C_W})^{C_V}\,\rmd\mu .
    \end{align*}
    By monotonicity of the infimum, as $\delta\downarrow0$, 
    \[ 
        V\Box W_\delta \mathrel{\Big\downarrow} V\Box W, \qquad C_{V\Box W_\delta} \mathrel{\Big\downarrow} C_{V\Box W}, \qquad f^{C_{V\Box W_\delta}} \mathrel{\Big\downarrow} f^{C_{V\Box W}},
    \] 
    and hence
    \begin{equation}\label{eq:proof:fundamental_inf_conv:dual.sandwich}
        \lim_{\delta \downarrow 0} V\Box W_\delta(\mu,\nu) 
        \ge 
        \int f\,\rmd\nu + \lim_{\delta \downarrow 0} \int (-f_\delta^{C_W})^{C_V}\,\rmd\mu 
        = 
        \int f\,\rmd\nu + \lim_{\delta \downarrow 0} \int f^{C_{V\Box W_\delta}}\,\rmd\mu 
        = 
        V\Box W(\mu,\nu).
    \end{equation} 
    For every $\rho\in\cP_p(\mathcal Y)$ and every $\delta>0$, 
    \[ 
        \int f_\delta^{C_W}\,\rmd\rho \le W(\rho,\nu) - \int f_\delta\,\rmd\nu < +\infty. 
    \] 
    Thus, by monotone convergence, $\int f_\delta^{C_W}\,\rmd\rho \downarrow \int f^{C_W}\,\rmd\rho$ as $\delta\downarrow0$, and monotonicity of the infimum gives $(-f_\delta^{C_W})^{C_V} \downarrow (-f^{C_W})^{C_V}$. Finally, 
    \[
        (-f_\delta^{C_W})^{C_V}(x) 
        \le 
        \inf_{\rho\in\cP_p(\mathcal Y)} \left\{ C_V(x,\rho) + W(\rho,\nu) - \int f\,\rmd\nu \right\} 
        = 
        C_{V\Box W}(x,\nu) - \int f\,\rmd\nu,
    \] 
    and because $C_{V\Box W}(\cdot,\nu)\in L^1(\mu)$, monotone convergence yields 
    \[ 
        \int (-f_\delta^{C_W})^{C_V}\,\rmd\mu \mathrel{\Big\downarrow} \int (-f^{C_W})^{C_V}\,\rmd\mu . 
    \] 
    Combined with \eqref{eq:proof:fundamental_inf_conv:dual.sandwich}, this proves dual attainment as claimed in \ref{it:thm.fundamental.3}.
    
    It remains to prove \ref{it:thm.fundamental.4}.
    First, assume that $\rho^\circ \in \cP_p(\mathcal Y)$ attains $\inf \bigl\{V(\mu,\rho)+W(\rho,\nu): \rho \in \cP_p(\mathcal Y)\bigr\}$, and that $f^\circ \in L^1(\nu)$ attains the dual problem. Then
    \begin{equation}
        \label{eq:thm.fundamental.aux2}
        V(\mu,\rho^\circ) + W(\rho^\circ,\nu)
        =
        \int (-(f^\circ)^{C_W})^{C_V} \,\rmd \mu + \int f^\circ \,\rmd \nu.
    \end{equation}
    At the same time, we have by duality that
    \begin{equation}
        \label{eq:thm.fundamental.auxV}
        V(\mu,\rho^\circ) \ge \int -(f^\circ)^{C_W} \,\rmd \rho^\circ + \int (-(f^\circ)^{C_W})^{C_V} \,\rmd \mu,
    \end{equation}
    and
    \begin{equation}
        \label{eq:thm.fundamental.auxW}
        W(\rho^\circ,\nu) \ge \int (f^\circ)^{C_W} \,\rmd \rho^\circ + \int f^\circ \,\rmd \nu.
    \end{equation}
    Adding \eqref{eq:thm.fundamental.auxV} and \eqref{eq:thm.fundamental.auxW} and comparing with \eqref{eq:thm.fundamental.aux2}, we see that both inequalities must in fact be equalities.

    Conversely, assume that $\rho^\circ \in \cP_p(\mathcal Y)$, $f^\circ \in L^1(\nu)$ satisfy \eqref{eq:thm.fundamental.complementary.slackness}.
    Since $C_{V \Box W}$ admits a $p$-growth bound and
    \[
        (-(f^\circ)^{C_W})^{C_V}(x) 
        =
        (f^\circ)^{C_{V \Box W}}(x) 
        \le 
        C_{V \Box W}(x,\nu) - \int f^\circ \, \rmd \nu,
    \]
    the positive part of $( -(f^\circ)^{C_W})^{C_V}$ is $\mu$-integrable.
    In particular, $\int (-(f^\circ)^{C_W})^{C_V} \, d\mu \in[-\infty,\infty)$.
    Moreover, since $W(\rho^\circ,\nu) \in (-\infty,\infty]$, the second identity in \eqref{eq:thm.fundamental.complementary.slackness} implies $\int (f^\circ)^{C_W} \, \rmd \rho^\circ \in (-\infty,\infty]$.
    Similarly, since ${V(\mu,\rho^\circ) \in (-\infty,\infty]}$, the identities in \eqref{eq:thm.fundamental.complementary.slackness} ensure that $\int (-(f^\circ)^{C_W})^{C_V} \, \rmd \mu$ and $\int -(f^\circ)^{C_W} \, \rmd \rho^\circ$ are real-valued. 
    Hence the two identities may be added, and this gives \eqref{eq:thm.fundamental.aux2}.
    Together with \ref{it:thm.fundamental.3}, this proves \ref{it:thm.fundamental.4}.
\end{proof} 

\begin{remark}
    Note that the assumption $W(\rho,\nu)<+\infty$ for all $\rho\in\cP_p(\mathcal Y)$ is convenient, but not necessary. For $f \in L^1(\nu)$, it is used only to guarantee the finiteness of
    \[
        \int_\mathcal Y f^{C_W}\,\rmd\rho
    \]
    on the effective domain of $C_V$.
    For instance, \ref{it:thm.fundamental.3} and \ref{it:thm.fundamental.4} of \cref{thm:fundamental_inf_conv} remain valid if there exists $\mathcal D\subset\cP_p(\mathcal Y)$ such that, for all $x\in\mathcal X$ and $\eta\in\cP_p(\mathcal Z)$,
    \[
        \mathcal D
        =
        \{\rho\in\cP_p(\mathcal Y): C_V(x,\rho)<+\infty\}
        =
        \{\rho\in\cP_p(\mathcal Y): W(\rho,\eta)<+\infty\}.
    \]
    In this case, the infima defining the $C$-conjugates are taken over $\mathcal D$ where the required finiteness holds.
\end{remark}

\subsection{Stability}\label{subsec:stability}
An important question in (weak) optimal transport is whether optimal values and optimizers depend continuously on the prescribed marginals. Note that such continuity does not hold in general; see \cite{BrJu22}.

We address this question for infimal convolutions of weak transport problems under standard assumptions. In particular, we prove convergence of the values $\bigl(V \Box W(\mu_k,\nu_k)\bigr)_{k}$ along convergent sequences $(\mu_k)_{k}$, $(\nu_k)_{k}$.
For the analysis of \cref{alg:SB}, however, this is not sufficient. Indeed, as a key step in the convergence analysis, we need to understand when the map
\[
    (\mu,\nu) \longmapsto \rho^\circ \in \arg \min \{V(\mu,\rho) + W(\rho,\nu) :  \rho \in \cP_p(\mathcal Y)\}
\]
is continuous. For this purpose, we impose the following compactness condition.

\begin{definition}[$\Wcal_r$-stability]\label{def:Wr-stability}
Let $W : \cP_p(\mathcal Y)\times \cP_p(\mathcal Z) \to \R \cup \{+\infty\}$ be a standard weak transport problem. We say that $W$ is $\Wcal_r$-stable if the following holds: 
If $(\rho_k)_k \subset \cP_p(\mathcal Y)$ converges weakly to some $\rho \in \cP(\mathcal Y)$,
$(\nu_k)_k \subset \cP_p(\mathcal Z)$ is precompact in $(\cP_p(\mathcal Z),\Wcal_p)$, and $\sup_{k} W(\rho_k,\nu_k) < \infty$,
then $\rho_k \to \rho$ in $(\cP_p(\mathcal Y),\Wcal_r)$.
\end{definition}

Crucially for the convergence analysis in \cref{sec:convergence_SB_algorithm}, this condition holds for $W = \Wcal_p^p$.
More generally, suppose that $p>1$ and that there exist $a > 0$, $b \in \R$ and $y_0 \in \mathcal Y$, $z_0 \in \mathcal Z$ such that, for all $(\rho,\nu) \in \cP_p(\mathcal Y) \times \cP_p( \mathcal Z)$,
\begin{equation}\label{aspt:lower-p-growth}
    W(\rho, \nu) \geq a \int d_\mathcal{Y}(y_0,y)^p \, \rho(\rmd y) - b \left(1 + \int d_\mathcal{Z}(z_0,z)^p \, \nu(\rmd z)\right).
\end{equation}
Then $W$ is $\Wcal_r$-stable for all $r \in [1,p)$. Indeed, if $\sup_{k} W(\rho_k,\nu_k) < \infty$ and $(\nu_k)_k$ is precompact in
$(\cP_p(\mathcal Z),\Wcal_p)$, we have
\[
    \sup_{k \in \N} \int d_\mathcal{Y}(y_0,y)^p \, \rho_k(\rmd y) < \infty.
\]
If, in addition, $\rho_k \to \rho \in \cP(\mathcal Y)$ weakly, then $\rho_k \to \rho$ in $(\cP_p(\mathcal Y),\Wcal_r)$ for all $r \in [1,p)$ by uniform integrability. In particular, if $\mathcal Y=\mathcal Z$ and $W=\Wcal_p^p$, then \eqref{aspt:lower-p-growth} holds with $a=2^{1-p}$ and $b=1$. Hence $\Wcal_p^p$ is $\Wcal_r$-stable for all $r\in[1,p)$.

Under the assumptions of \Cref{prop:properties_inf_conv}, together with the $p$-growth bound on $C_{V\Box W}$, \Cref{lem:VBoxW.tightness.min.sequ} shows that every minimizing sequence for the infimal convolution is tight. However, a limit point in the weak topology need not have a finite $p$-moment. The $\Wcal_r$-stability condition upgrades this tightness to precompactness in $(\cP_p(\mathcal Y),\Wcal_r)$. This yields the following attainment result.

\begin{corollary}[Attainment in the infimal convolution]\label{cor:attainment.infimal.convolution}
    Let $V,W$ be standard weak transport problems such that $(C_V,W)$ satisfies \Cref{aspt:continuity} and \Cref{aspt:coercivity} for some $r \in [1,p]$. 
    Assume further that $C_{V\Box W}$ admits a $p$-growth bound and that $W$ is $\Wcal_r$-stable. Then, for all $\mu\in\cP_p(\mathcal X)$ and $\nu\in\cP_p(\mathcal Z)$, there exists $\rho\in\cP_p(\mathcal Y)$ such that
    \[
        V\Box W(\mu,\nu) = V(\mu,\rho) + W(\rho,\nu).
    \]
\end{corollary}
\begin{proof}
    Let $(\rho_k)_{k \in \N} \subset \cP_p(\mathcal Y)$ such that 
    \[
        V(\mu,\rho_k) + W(\rho_k,\nu) \longrightarrow V\Box W(\mu,\nu),
        \qquad
        \sup_{k \in \N} W(\rho_k,\nu) < \infty. 
    \]
    By \cref{lem:VBoxW.tightness.min.sequ}, this sequence is tight. Passing to a subsequence, there exists $\rho \in \cP(\mathcal Y)$ such that $\rho_k \to \rho$ weakly. Since $W$ is $\Wcal_r$-stable, it follows that $\rho_k \to \rho$ in $(\cP_p(\mathcal Y),\Wcal_r)$. In particular, $\rho \in \cP_p(\mathcal Y)$. By \Cref{lem:coercivity_gives_lsc},
    \[
        V(\mu,\rho) + W(\rho,\nu) \le \lim_{k \to \infty} V(\mu,\rho_k) + W(\rho_k,\nu) = V\Box W(\mu,\nu). \qedhere
    \]
\end{proof}

Finally, we turn to stability under perturbations of the marginals. The following theorem proves convergence of the values and, under $\Wcal_r$-stability, compactness and stability of minimizers for the infimal convolution.

\begin{theorem}[Stability of the infimal convolution]\label{thm:stability}
    Let $V,W$ be standard weak transport problems such that $(C_V,W)$ satisfies \Cref{aspt:continuity} and \Cref{aspt:coercivity} for some $r \in [1,p]$. 
    Assume further that $C_{V\Box W}$ admits a $p$-growth bound.
    Let $(\mu_k)_k \subset \cP_p(\mathcal X)$ and $(\nu_k)_k \subset \cP_p(\mathcal Z)$ such that $\mu_k \to \mu$ in $\cP_p(\mathcal X)$ and $\nu_k \to \nu$ in $\cP_p(\mathcal Z)$. Then the following hold:
    \begin{enumerate}[label = (\roman*)]
        \item \label{it:stability.1} We have $V \Box W(\mu_k,\nu_k) \to V\Box W(\mu,\nu)$. Moreover, any sequence of minimizers
        \begin{equation}\label{eq:thm.stability.1:VBoxW.attained}
            \rho_k \in \argmin\left\{V(\mu_k,\rho) + W(\rho,\nu_k) : \rho \in \cP_p(\mathcal Y)\right\}
        \end{equation}
        is tight.

        \item \label{it:stability.2} Suppose, in addition, that $W$ is $\Wcal_r$-stable. Then there exists a sequence $(\rho_k)_{k \in \N}$ satisfying \eqref{eq:thm.stability.1:VBoxW.attained}, and every such sequence is precompact in $(\cP_p(\mathcal{Y}),\Wcal_r)$. Moreover, all its limit points belong to 
        \begin{equation}\label{eq:thm.stability.argmin}
            \argmin\{V(\mu,\rho) + W(\rho,\nu) : \rho \in \cP_p(\mathcal Y) \}.
        \end{equation}
        In particular, if the minimizer in \eqref{eq:thm.stability.argmin} is unique for every $\mu \in \cP_p(\mathcal X)$, $\nu \in \cP_p(\mathcal Z)$, then the map
        \[
            (\mu,\nu) \longmapsto \rho^\circ \in \arg \min \{V(\mu,\rho) + W(\rho,\nu) :  \rho \in \cP_p(\mathcal Y)\},
        \]
        is continuous from $\cP_p(\mathcal X)\times\cP_p(\mathcal Z)$ to $(\cP_p(\mathcal{Y}),\Wcal_r)$.
    \end{enumerate}
\end{theorem}

\begin{proof}
    By \Cref{prop:properties_inf_conv}, $V\Box W$ is a standard weak transport problem with cost function $C_{V\Box W}$. Moreover, $C_{V \Box W}$ is continuous and admits a $p$-growth bound. Thus the stability theorem for weak transport problems yields $V \Box W(\mu_k,\nu_k) \to V\Box W(\mu,\nu)$; see \cite[Theorem 2.8]{BeJoMaPa23}.
    
    Let $(\rho_k)_{k\in \N} \subset \cP_p(\mathcal Y)$ be such that, for every $k \in \N$,
    \[
        V\Box W(\mu_k,\nu_k) = V(\mu_k,\rho_k) + W(\rho_k,\nu_k).
    \]
    Then
    \begin{align*}
        V\Box W(\mu_k,\nu_k)
        &\leq
        \int C_{V\Box W}(x,\pi^k_x)\,\mu_k(\rmd x) 
        \leq 
        \int C_V(x,\xi_x^k) + W(\xi_x^k,\pi^k_x)\,\mu_k(\rmd x) 
        \\
        &\leq 
        \int \left(C_V(x,\xi_x^k) + \int C_W(y,\chi^k_y)\,\xi_x^k(\rmd y)\right)\,\mu_k(\rmd x)
        = V(\mu_k,\rho_k) + W(\rho_k,\nu_k),
    \end{align*}
    hence all inequalities are equalities. In particular, $C_{V \Box W}(x,\pi^k_x) = C_V(x,\xi^k_x) + W(\xi_x^k,\pi^k_x)$ for $\mu_k$-a.e.\ $x$.
    Since $(\mu_k)_k$ and $(\nu_k)_k$ are precompact in $\cP_p(\mathcal X)$ resp.\ $\cP_p(\mathcal Z)$, it follows as in the proof of \cref{lem:VBoxW.tightness.min.sequ} that $(\rho_k)_{k\in\N}$ is tight.

    To show \ref{it:stability.2}, assume in addition that $W$ is $\Wcal_r$-stable. Existence of a sequence satisfying \eqref{eq:thm.stability.1:VBoxW.attained} follows from \Cref{cor:attainment.infimal.convolution}. Let $(\rho_k)_{k\in\N}$ be any such sequence.
    Passing to a subsequence, let $(\rho_k)_{k \in \N}$ converge weakly to some $\rho^\circ \in \cP(\mathcal Y)$. 
    Since $\sup_k W(\rho_k,\nu_k) < \infty$ and by $\Wcal_r$-stability, $\rho_k \to \rho^\circ$ in $(\cP_p(\mathcal Y),\Wcal_r)$.
    Then \cref{lem:coercivity_gives_lsc} gives
    \[
        V(\mu,\rho^\circ) + W(\rho^\circ,\nu) 
        \le 
        \liminf_{k \to \infty} V(\mu_k,\rho_k) + W(\rho_k,\nu_k)
        = 
        \liminf_{k \to \infty} V \Box W(\mu_k,\nu_k) = V \Box W(\mu,\nu),
    \]
    and so $\rho^\circ \in \cP_p(\mathcal Y)$ attains $V\Box W(\mu,\nu)$.
\end{proof}

\subsection{Deconvolution of weak transport problems}\label{subsec:deconvolution}

We finally introduce an operation which is closely related to the infimal convolution. 
For weak transport problems $V$ and $W$, set
\[
    V \boxminus W(\mu,\nu)
    :=
    \inf_{\rho \in \Pcal_p(\mathcal Y)} \bigl\{ V(\mu,\rho) - W(\rho,\nu) \bigr\},
    \qquad
    \mu \in \Pcal_p(\mathcal X),\ \nu \in \Pcal_p(\mathcal Z),
\]
whenever the right-hand side is well-defined. We call $V\boxminus W$ the deconvolution of $V$ by $W$. The following proposition gives its dual representation.

\begin{proposition}\label{prop:deconvolution}
    Let $\mu \in \cP_p(\mathcal X)$, $\nu \in \cP_p(\mathcal Z)$ and let $V$, $W$ be weak transport problems such that $W(\rho,\nu) < +\infty$ for all $\rho \in \cP_p(\mathcal{Y})$. Then
    \[
        V\boxminus W(\mu,\nu) = \inf_{f \in C_{b,p}(\mathcal Z)} \left\{ \int (f^{C_W})^{C_V} \,\rmd\mu - \int f \, \rmd\nu \right\}.
    \]
\end{proposition}

\begin{proof}
By strong duality (see \cite[Theorem 3.1]{backhoff2019existence}), $W$ admits the dual representation
\[
    W(\rho,\nu)
    =
    \sup_{f \in C_{b,p}(\mathcal Z)}
    \left\{
        \int_\mathcal{Z} f \,\rmd\nu + \int_\mathcal{Y} f^{C_W} \,\rmd\rho
    \right\}
\]
for every $\rho \in \cP_p(\mathcal Y)$. Hence
\begin{align*}
    V\boxminus W(\mu,\nu)
    &= \inf_{\rho \in \cP_p(\mathcal Y)} \bigl\{V(\mu,\rho) - W(\rho,\nu)\bigr\} \\
    &= \inf_{f \in C_{b,p}(\mathcal Z)}
    \inf_{\rho \in \cP_p(\mathcal Y)}
    \left\{
        V(\mu,\rho) - \int f \,\rmd\nu - \int f^{C_W} \,\rmd\rho
    \right\}.
\end{align*}
Moreover, by weak duality, for every $f \in L^1(\nu)$ and every $\rho \in \cP_p(\mathcal{ Y})$,
\[
    \int f^{C_W}\,\rmd\rho \le W(\rho,\nu)-\int f \,\rmd\nu < \infty.
\]

Now fix $f \in C_{b,p}(\mathcal Z)$. Then
\begin{align*}
    \inf_{\rho \in \cP_p(\mathcal Y)}
    \left\{
        V(\mu,\rho) - \int f \,\rmd\nu - \int f^{C_W} \,\rmd\rho
    \right\}
    &=
    \inf_{\pi \in \cpl_p(\mu,\ast)}
    \left\{
        \int_\mathcal{X}
        \left(
            C_V(x,\pi_x) - \int_\mathcal{Y} f^{C_W}(y)\,\pi_x(\rmd y)
        \right)
        \mu(\rmd x)
        -
        \int_\mathcal{Z} f \,\rmd\nu
    \right\} \\
    &=
    \int_\mathcal{X}
    \inf_{\rho \in \cP_p(\mathcal Y)}
    \left\{
        C_V(x,\rho) - \int_\mathcal{Y} f^{C_W}(y)\,\rho(\rmd y)
    \right\}
    \mu(\rmd x)
    -
    \int_\mathcal{Z} f \,\rmd\nu \\
    &=
    \int_\mathcal{X} (f^{C_W})^{C_V} \,\rmd\mu
    -
    \int_\mathcal{Z} f \,\rmd\nu.
\end{align*}
The first equality is the definition of $V$. The second equality follows by interchanging the infimum and the integral, as in the proofs of \Cref{prop:properties_inf_conv,thm:fundamental_inf_conv}.
Taking the infimum over $f \in C_{b,p}(\mathcal Z)$ proves the claim.
\end{proof} 

\begin{remark} \leavevmode
    \begin{enumerate}
        \item It readily follows from the proof that one may replace $C_{b,p}(\mathcal Z)$ by $L_{b,p}(\mathcal Z)$. 

        \item More generally, the preceding argument does not require $W$ itself to be given by a weak transport problem. It suffices that $W$ admits a dual representation of the form
        \[
            W(\eta,\nu)
            =
            \sup_{f \in C_{b,r}(\mathcal Z)}
            \left\{
                \int f(z)\,\nu(\rmd z) + \int \Tcal[f](y)\,\eta(\rmd y)
            \right\},
        \]
        where $\Tcal$ is an operator on $C_{b,r}(\mathcal Z)$ taking values in the set of measurable functions on $\mathcal Y$ that are bounded from below. In that case, the conclusion of \Cref{prop:deconvolution} remains valid with $f^{C_W}$ replaced by $\Tcal[f]$.
    \end{enumerate}
\end{remark}

\subsection{Applications}

\subsubsection{The Bass problem}\label{subsub:mBB_deconv}
As a first application of \Cref{prop:deconvolution}, we recover the dual formulation of the Bass problem~\eqref{eq:bass_problem}; see, e.g., \cite{BackhoffSchachTschid2025} for further details.

Let $\mathcal{X}=\mathcal{Y}=\mathcal{Z}=\R^d$ for some $d\in\N$, and let $\mu,\nu \in \cP_2(\R^d)$. Define
\[
    V(\mu,\nu) := -\MCov(\mu,\nu) =  -\sup_{\pi \in \cpl(\mu,\nu)} \int x \cdot y \, \pi(\rmd x,\rmd y).
\]
Then the associated $C_V$-transform is
\[
    f^{C_V}(x)
    =
    \inf_{y\in\R^d}\bigl\{-x\cdot y-f(y)\bigr\}
    =
    -(-f)^*(x),
\]
and thus $(f^{C_V})^{C_V} = -(-f)^{**}$.

As a first consequence, \Cref{prop:deconvolution} yields
\begin{align*}
    \inf_{\alpha\in\cP_2(\R^d)}
    \bigl\{
        \MCov(\alpha,\nu)-\MCov(\mu,\alpha)
    \bigr\}
    &=
    \inf_{f \in C_{b,2}(\R^d)}
    \left\{
        -\int (-f)^{**} \,\rmd\mu 
        -
        \int f \,\rmd\nu
    \right\} \\
    &=
    \inf_{\substack{f \in L^1(\nu),\\ \text{concave, usc}}}
    \left\{
        \int f \,\rmd\mu - \int f \, \rmd \nu
    \right\},
\end{align*}
where the last infimum runs over all concave, upper semicontinuous functions ${f\in L^1(\nu)}$. Indeed, for every ${g\in L^1(\nu)}$, its lower semicontinuous hull $g^{**}$ satisfies $g^{**}\ge g$, and so one may restrict to concave upper semicontinuous functions.

Next, define $W(\mu,\nu) := -\MCov(\mu \ast \gamma, \nu)$. Using the dual representation of $W$, we obtain
\[
    W(\mu,\nu) = \sup_{\substack{f\in L^1(\nu),\\ \text{concave, usc}}} \left\{ \int -(-f)^\ast \ast \gamma \,\rmd\mu + \int f \, \rmd\nu \right\}.
\]
Therefore
\[
    (f^{C_W})^{C_V} = (-(-f)^\ast \ast \gamma)^{C_V} = -((-f)^\ast \ast \gamma)^\ast,
\]
and hence \cref{prop:deconvolution} gives
\begin{equation}\label{eq:bass_functional.dual}
    \inf_{\alpha \in \cP_2(\R^d)} \bigl\{\MCov(\alpha \ast \gamma,\nu) - \MCov(\mu,\alpha)\bigr\}
    = 
    \inf_{\substack{f \in L^1(\nu), \\ \text{concave, usc}}} \left\{ - \int \big( (-f)^\ast \ast \gamma \big)^\ast \, \rmd\mu - \int f \, \rmd\nu \right\}.
\end{equation}
If $\mu \le_{\rm cvx} \nu$, i.e., ${\int \psi \,\rmd\mu \le \int \psi \,\rmd\nu}$ for all convex functions $\psi$,
then \cite[Theorem 1.5]{BackhoffSchachTschid2025} shows that
\[
    \inf_{\alpha \in \cP_2(\R^d)} \bigl\{\MCov(\alpha \ast \gamma,\nu) - \MCov(\mu,\alpha)\bigr\} = \frac{d}2 + \int q_1\,\rmd(\nu - \mu) - V_{\rm mBB}(\mu,\nu).
\]
Combining this identity with \eqref{eq:bass_functional.dual}, we arrive at the dual formulation of the Bass problem~\eqref{eq:bass_problem}:
\begin{equation}\label{eq:bass_problem.dual}\tag{mBBd}
    V_{\rm mBB}(\mu,\nu) 
    = 
    \sup_{\substack{\psi \in L^1(\nu),\\ \text{ convex, lsc}}} \left\{\int\!\left(\frac{d}{2} - q_1 + \big(\psi^**\gamma\big)^*\right)\rmd\mu + \int (q_1-\psi)\,\rmd\nu\right\}.
\end{equation}

\subsubsection{The Schrödinger--Bass problem} \label{subsub:SB_deconv}
As a final application, we combine the quadratic Wasserstein transport problem with the entropic transport problem through infimal convolution and deconvolution. In \Cref{sec:the_SB_problem}, the resulting variational problem is shown to coincide with the Schrödinger--Bass problem. Let $\beta>0$, and consider
\[
    W_\beta(\mu,\nu) := \frac{\beta}{2}\Wcal_2^2(\mu,\nu),
    \qquad
    V_{\rm EOT}(\mu,\nu) := \inf_{\pi \in \cpl(\mu,\nu)} \int H(\pi_x \,|\, \gamma_x)\,\mu(\rmd x).
\]
To begin with, we verify that \cref{thm:fundamental_inf_conv} and \cref{thm:stability} apply to this case.

\begin{lemma}\label{lem:SB.cost.assumptions} 
    Let $\beta>0$ and set 
    \[
        C_{V_{\rm EOT}}(x,\rho):=H(\rho\,|\,\gamma_x), 
        \qquad
        C_{W_\beta}(y,\eta):= \int \frac\beta2|y-z|^2\,\eta(\rmd z).
    \]
    Then $C_{V_{\rm EOT}}$ and $C_{W_\beta}$ are standard weak transport costs on $\Rd \times \cP_2(\Rd)$ and $W_\beta$ is $\Wcal_r$-stable for every $r \in [1,2)$.
    Moreover, $(C_{V_{\rm EOT}},W_\beta)$ satisfies \Cref{aspt:continuity} and \Cref{aspt:coercivity} for every $r\in[1,2)$, and $C_{V_{\rm EOT}\Box W_\beta}$ admits a $2$-growth bound.
\end{lemma}

\begin{proof}
    The cost $C_{W_\beta}$ is finite, continuous, bounded from below, and linear in $\eta$. Moreover, since $W_\beta$ satisfies \eqref{aspt:lower-p-growth}, it is $\Wcal_r$-stable for all $r \in [1,2)$.
    For $C_{V_{\rm EOT}}$, non-negativity gives the lower bound. Convexity in $\rho$ follows from convexity of $H$, and using
    \[
        H(\rho\,|\,\gamma_x)
        =
        H(\rho\,|\,\gamma) - x\cdot\bar\rho + q_1(x).
    \]
    lower semicontinuity follows from the weak lower semicontinuity of $\rho \mapsto H(\rho\,|\,\gamma)$.
    
    Moreover, \Cref{aspt:continuity} is satisfied, since $x\mapsto H(\rho\,|\,\gamma_x)$ is continuous whenever $\rho\ll\gamma$, while $\eta\mapsto\Wcal_2^2(\rho,\eta)$ is continuous on $\cP_2(\Rd)$ if $\rho \in \cP_2(\Rd)$.
    To prove \Cref{aspt:coercivity}, let $K\subseteq\Rd\times\cP_2(\Rd)$ be compact and suppose that
    \[
        \inf_{(x,\eta)\in K}\left\{ H(\rho\,|\,\gamma_x) + \frac{\beta}{2}\Wcal_2^2(\rho,\eta) \right\} \le R.
    \]
    Then there exists $(x,\eta)\in K$ such that $\Wcal_2^2(\rho,\eta)\le \frac{2R}\beta+1$, and so the second moments of $\eta \in K$ are uniformly bounded on $K$. Hence
    \[
        \int |y|^2\,\rho(\rmd y)
        \le
        2\Wcal_2^2(\rho,\eta)
        +
        2\int |z|^2\,\eta(\rmd z)
    \]
    is uniformly bounded. Since the sublevel set is closed by weak lower semicontinuity, it is compact in $(\cP_2(\Rd),\Wcal_r)$ for every $r<2$.
    Finally,
    \[ 
        C_{V_{\rm EOT}\Box W_\beta}(x,\eta) 
        \le 
        \frac{\beta}{2}\Wcal_2^2(\gamma_x,\eta)    
        \le 
        c\left(1+|x|^2+\int |z|^2\,\eta(\rmd z)\right), 
    \] 
    for some $c>0$. Thus $C_{V_{\rm EOT}\Box W_\beta}$ admits a $2$-growth bound. 
\end{proof}

In particular, $V_{\rm EOT} \Box W_\beta$ is a standard weak transport problem, and its $C$-transform is explicit: Since
\begin{align*}
    f^{C_{V_{\rm EOT}}}(x)
    &= \inf_{\rho \in \cP_2(\R^d)}
    \left\{
        H(\rho \,|\, \gamma_x) - \int f \,\rmd \rho
    \right\}
    = -\log\bigl(\exp(f)\ast\gamma(x)\bigr), \\
    f^{C_{W_\beta}}(x)
    &= \inf_{y \in \R^d}
    \left\{
        \frac{\beta}{2}|x-y|^2 - f(y)
    \right\}
    = q_\beta \Box (-f)(x),
\end{align*}
we have, by \cref{thm:fundamental_inf_conv},
\[
    f^{C_{V_{\rm EOT}\Box W_\beta}}
    =
    \bigl(-f^{C_{W_\beta}}\bigr)^{C_{V_{\rm EOT}}}
    =
    -\log\bigl(\exp(-q_\beta \Box (-f))\ast\gamma\bigr).
\]

Moreover, set $\Tcal^\beta_1[f]:= -\log\bigl( \exp(-q_\beta \Box (-f)) \ast \gamma \bigr)$.
Applying \Cref{prop:deconvolution}, we obtain
\begin{gather}
    \label{eq:SB.variational}\tag{SB$\beta$v}
    \sup_{\alpha \in \cP_2(\Rd)}
    \inf_{\rho \in \cP_2(\Rd)}
    \bigl\{
        - W_\beta(\mu,\alpha)
        + V_{\rm EOT}(\alpha,\rho)
        + W_\beta(\rho,\nu)
    \bigr\}
    \\
    \label{eq:SB.dual}\tag{SB$\beta$d}
    = \sup_{f \in C_{b,2}(\R^d)}
    \left\{
        \int f \,\rmd\nu
        -
        \int q_\beta \Box \bigl(-\Tcal^\beta_1[f]\bigr)\,\rmd\mu
    \right\}.
\end{gather}
Note that, if we define 
\[
    g := (f^{C_{W_\beta}})^{C_{W_\beta}},
\]
then $g \ge f$ and $g^{C_{W_\beta}} = f^{C_{W_\beta}}$. It follows that the dual problem in \eqref{eq:SB.dual} may be restricted to functions satisfying
\[
    (f^{C_{W_\beta}})^{C_{W_\beta}} = f,
\]
that is, to $\beta$-semiconcave functions $f \in L^1(\nu)$. 
In \Cref{sec:the_SB_problem}, it is further shown that the variational problem~\eqref{eq:SB.variational} is in fact a standard weak transport problem and coincides with the Schrödinger--Bass problem~\eqref{eq:main.SB.dynamic}.

We conclude the present section by proving that both \eqref{eq:SB.variational} and \eqref{eq:SB.dual} are attained.

\begin{proposition}\label{prop:SB_deconv_attained}
    Let $\mu, \nu \in \cP_2(\Rd)$ and $\beta >0$. 
    Then there exist a $\beta$-semiconcave potential $f^\circ$ and ${\alpha^\circ \in \cP_2(\Rd)}$, such that \eqref{eq:SB.variational} and \eqref{eq:SB.dual} are attained, that is,
    \begin{equation}\label{eq:SB.variational_attained}
        - W_\beta(\mu,\alpha^\circ) +
        \inf_{\rho \in \cP_2(\Rd)}
        \bigl\{
            V_{\rm EOT}(\alpha^\circ,\rho)
            + W_\beta(\rho,\nu)
        \bigr\}
        =
        \int f^\circ \,\rmd\nu
        -
        \int q_\beta \Box \bigl(-\Tcal^\beta_1[f^\circ]\bigr)\,\rmd\mu.
    \end{equation}
\end{proposition}

\begin{remark}\label{rem:SB_deconv_attained.optimizers}
    For every $\alpha \in \Pcal_2(\Rd)$, the infimum in \eqref{eq:SB.variational_attained} is attained because it is an infimal convolution of standard weak transport problems and \Cref{thm:fundamental_inf_conv} applies: for the optimizers $f^\circ$ and $\alpha^\circ$ from \Cref{prop:SB_deconv_attained}, there exists $\rho^\circ \in \cP_2(\Rd)$ such that
    \[
        - W_\beta(\mu,\alpha^\circ) + V_{\rm EOT}(\alpha^\circ,\rho^\circ) + W_\beta(\rho^\circ,\nu)
        =
        \int f^\circ \,\rmd\nu
        -
        \int q_\beta \Box \bigl(-\Tcal^\beta_1[f^\circ]\bigr)\,\rmd\mu.
    \]
    Applying the complementary slackness condition from \Cref{thm:fundamental_inf_conv} to $V_{\rm EOT}\Box W_\beta(\alpha^\circ,\nu)$ we obtain
    \begin{align*}
        V_{\rm EOT}(\alpha^\circ,\rho^\circ)
        &=
        \int -q_\beta\Box(-f^\circ) \,\rmd\rho^\circ + \int \Tcal^\beta_1[f^\circ]\,\rmd\alpha^\circ, \\
        W_\beta(\rho^\circ,\nu)
        &=
        \int q_\beta\Box(-f^\circ) \,\rmd\rho^\circ + \int f^\circ\,\rmd\nu.
    \end{align*}
    In particular, the unique primal optimizer $\chi^\circ \in \Cpl(\alpha^\circ,\rho^\circ)$ of $V_{\rm EOT}(\alpha^\circ,\rho^\circ)$ is given by
    \[
        \frac{\rmd \chi_y^\circ}{\rmd\gamma_y} 
        := 
        \frac{\exp\bigl(-q_\beta \Box(-f^\circ)\bigr)}{\exp\bigl(-q_\beta \Box(-f^\circ)\bigr) * \gamma (y)}
        \qquad \text{ for }\alpha^\circ\text{-a.e. }y,
    \]
    and with $f^\circ = q_\beta - \beta v^\circ$, where $v^\circ:\Rd \to \R$ is a convex function, we have 
    \[
        (\nabla (v^\circ)^\ast)_\# \rho^\circ = \nu.
    \]
    Finally, $u:=q_1-\frac{1}{\beta}q_\beta\Box(-\Tcal^\beta_1[f^\circ])$ satisfies
    \[
        (\nabla u)_\#\mu = \alpha^\circ.
    \]
\end{remark}

\begin{proof}[Proof of \cref{prop:SB_deconv_attained}]
    Let $(f_n)_{n \in \N}$ be a maximizing sequence of $\beta$-semiconcave functions for \eqref{eq:SB.dual}, i.e.,
    \[
        s^\circ:= \sup_{f \in C_{b,2}(\R^d)}
        \left\{
            \int f \,\rmd\nu
            -
            \int q_\beta \Box \bigl(-\Tcal^\beta_1[f]\bigr)\,\rmd\mu
        \right\}
        =
        \lim_{n\to\infty}
        \left\{
            \int f_n \,\rmd\nu
            -
            \int q_\beta \Box \bigl(-\Tcal^\beta_1[f_n]\bigr)\,\rmd\mu
        \right\}.
    \]
    Since $\Tcal^\beta_1[f+c]=\Tcal^\beta_1[f]+c$ for every $c\in\R$, the value of the functional is invariant under addition of constants:
    \[
        \int (f+c)\,\rmd\nu -
        \int q_\beta \Box \bigl(-\Tcal^\beta_1[f+c]\bigr)\,\rmd\mu
        =
        \int f \,\rmd\nu
        -
        \int q_\beta \Box \bigl(-\Tcal^\beta_1[f]\bigr)\,\rmd\mu.
    \]
    We may therefore assume that $\int f_n \,\rmd \nu = 0$ for all $n\in\N$. For each $n\in\N$, define
    \[
        u_n := q_1 - \frac{1}{\beta}q_\beta\Box(-\Tcal^\beta_1[f_n]) \in L^1(\mu).
    \]
    By \Cref{cor:smooth_brenier_map}, each $u_n$ is convex and $\frac{1+\beta}{\beta}$-smooth, and by construction
    \[
        \int u_n\,\rmd\mu \ \mathrel{\Big\uparrow} \ \frac{s^\circ}{\beta} + \int q_1\,\rmd\mu.
    \]        
    By \Cref{lem:tightness}, there exist a constant $c(\beta,d)>0$ and a subsequence $(u_{n_k})_{k\in\N}$ which converges locally uniformly on $\Rd$ to a convex, $\frac{1+\beta}{\beta}$-smooth function $u$ such that
    \[
        \sup_{n \in \N}|u_n(x)| \le c(\beta,d)\left(1+|x|^2\right) \ \text{ for all } x\in\Rd, \qquad \alpha_{n_k} := (\nabla u_{n_k})_\#\mu \longrightarrow \alpha^\circ :=(\nabla u)_\#\mu \ \text{ in } \cP_2(\Rd).
    \]
    For every $k \in \N$, complementary slackness from \Cref{thm:fundamental_inf_conv} yields
    \[
        - W_\beta(\mu,\alpha_{n_k}) + V_{\rm EOT}\Box W_\beta(\alpha_{n_k},\nu)
        =
        \int f_{n_k}\,\rmd\nu
        -
        \int q_\beta \Box \bigl(-\Tcal^\beta_1[f_{n_k}]\bigr)\,\rmd\mu = \int(\beta u_{n_k}-q_1)\,\rmd\mu.
    \]
    By \Cref{thm:stability}, and since $W_\beta(\mu,\alpha_{n_k}) \to W_\beta(\mu,\alpha^\circ)$, it follows that
    \[
        - W_\beta(\mu,\alpha_{n_k}) + V_{\rm EOT}\Box W_\beta(\alpha_{n_k},\nu)
        \longrightarrow
        - W_\beta(\mu,\alpha^\circ) + V_{\rm EOT}\Box W_\beta(\alpha^\circ,\nu),
    \]
    while dominated convergence gives $\int(\beta u_{n_k}-q_1)\,\rmd\mu \to \int (\beta u-q_1)\,\rmd\mu = s^\circ$. 
    Consequently,
    \[
        - W_\beta(\mu,\alpha^\circ) + V_{\rm EOT}\Box W_\beta(\alpha^\circ,\nu)
        =
        \int(\beta u-q_1)\,\rmd\mu.
    \]
    By \Cref{thm:fundamental_inf_conv}, there exists $f^\circ \in L^1(\nu)$ such that
    \[
        V_{\rm EOT}\Box W_\beta(\alpha^\circ,\nu)
        =
        \int f^\circ\,\rmd\nu + \int \Tcal^\beta_1[f^\circ]\,\rmd\alpha^\circ.
    \]
    Let
    \[
        g:=\bigl((f^\circ)^{C_{W_\beta}}\bigr)^{C_{W_\beta}}.
    \]
    Then $g$ is $\beta$-semiconcave, $g\ge f^\circ$, and $\Tcal^\beta_1[g]=\Tcal^\beta_1[f^\circ]$. 
    Thus we may choose $f^\circ$ to be $\beta$-semiconcave.
    
    Since $\Tcal^\beta_1[f^\circ]$ is an admissible dual candidate for $W_\beta(\mu,\alpha^\circ)$, we have
    \[
        W_\beta(\mu,\alpha^\circ) \geq \int q_\beta\Box(-\Tcal^\beta_1[f^\circ])\,\rmd \mu + \int \Tcal^\beta_1[f^\circ] \,\rmd\alpha^\circ.
    \]
    If the inequality were strict, then
    \[
        - W_\beta(\mu,\alpha^\circ) + V_{\rm EOT}\Box W_\beta(\alpha^\circ,\nu) < \int f^\circ\,\rmd\nu - \int q_\beta\Box(-\Tcal^\beta_1[f^\circ])\,\rmd \mu,
    \]
    contradicting the definition of $s^\circ$. Hence
    \[
        W_\beta(\mu,\alpha^\circ) = \int q_\beta\Box(-\Tcal^\beta_1[f^\circ])\,\rmd \mu + \int \Tcal^\beta_1[f^\circ] \,\rmd\alpha^\circ,
    \]
    and \eqref{eq:SB.variational_attained} follows.
\end{proof}

\section{The Schrödinger--Bass problem}\label{sec:the_SB_problem}

We now turn to the Schrödinger--Bass problem given by
\begin{equation}\label{eq:SB.dynamic}\tag{SB$\beta$}
    V_{\rm SB}^\beta(\mu,\nu) 
    =
    \inf_{\substack{X_0 \sim\mu, \, X_1\sim \nu, \\ \rmd X_t = a_t \rmd t + b_t \rmd B_t}} \mathbb E\!\left[ \int_0^1 \frac12|a_t|^2 + \frac\beta2|b_t - I_d|_{\rm HS}^2 \, \rmd t\right]
\end{equation}
for $\beta > 0$. The purpose of this section is twofold. First, we prove that \eqref{eq:SB.dynamic} is in fact a standard weak transport problem and coincides with the variational problem~\eqref{eq:SB.variational} studied in \cref{subsub:SB_deconv}. Second, we identify the \emph{Schrödinger--Bass system}, which uniquely determines the law of the optimal semimartingale; see \cref{fig:SB_system} for a schematic illustration.

We begin by defining $C_{\rm SB}^\beta: \Rd \times \cP_2(\Rd) \to \R$ by
\begin{equation}\label{eq:SB.cost.dynamic}
    C_{\rm SB}^\beta(x,\eta)
    :=
    \inf_{\substack{X_0=x, \, X_1\sim\eta,\\\rmd X_t = a_t \rmd t + b_t \rmd B_t}} \mathbb E\!\left[\int_0^1 \frac12 |a_t|^2 + \frac{\beta}{2}|b_t-I_d|_{\mathrm{HS}}^2 \,\rmd t\right].
\end{equation}
It is shown in \cref{lem:SB.cost} that $C_{\rm SB}^\beta$ is a continuous standard weak transport cost and satisfies a quadratic growth bound. Moreover, $\eta \mapsto C_{\rm SB}^\beta(x,\eta)$ is strictly convex for every $x \in \Rd$. It follows then from the general theory of weak optimal transport that the associated WOT problem
\begin{equation}\label{eq:SB.wot}\tag{SB$\beta$w}
    \inf_{\pi \in \Cpl(\mu,\nu)} \int C_{\rm SB}^\beta(x,\pi_x)\,\mu(\rmd x),
\end{equation}
admits a unique optimizer $\pi^\circ \in \Cpl(\mu,\nu)$.
Moreover, conditioning on $X_0$ immediately gives
\begin{equation}\label{inequ:SB.easy}
    V_{\rm SB}^\beta(\mu,\nu)
    \ge
    \inf_{\substack{X_0 \sim\mu, \, X_1\sim \nu, \\ \rmd X_t = a_t \rmd t + b_t \rmd B_t}} \mathbb E\!\left[ C_{\rm SB}^\beta\bigl(X_0,\mathrm{Law}(X_1\,|\,X_0)\bigr)\right] 
    =
    \int C_{\rm SB}^\beta(x,\pi^\circ_x)\,\mu(\rmd x).
\end{equation}
A major goal of this section is therefore to prove the converse inequality.
Equivalently, starting from the unique optimizer $\pi^\circ$ of the static weak transport problem~\eqref{eq:SB.wot}, we construct a semimartingale $X^\circ$ with ${\mathrm{Law}(X^\circ_0,X^\circ_1)=\pi^\circ}$ and which, conditional on $X^\circ = x$, attains $C_{\rm SB}^\beta(x,\pi^\circ_x)$ for $\mu$-a.e.\ $x$.

We achieve this by relating \eqref{eq:SB.wot} to the variational problem~\eqref{eq:SB.variational} and its dual formulation~\eqref{eq:SB.dual}.
This perspective gives rise to the Schrödinger--Bass system and, in particular, allows us to construct from a dual optimizer the optimal semimartingale $X^\circ$.
The following theorem summarizes the main results of this section.

\begin{theorem}[Existence and uniqueness of the Schrödinger--Bass system]\label{thm:SB}
    Let $\beta > 0$ and $\mu,\nu \in \cP_2(\Rd)$. Then, $C_{\rm SB}^\beta: \R^d\times\cP_2(\Rd) \to \R$ is a continuous standard weak transport cost function, and
    \begin{align} 
        V_{\rm SB}^\beta(\mu,\nu) 
        &=
        \inf_{\substack{X_0 \sim\mu, \, X_1\sim \nu, \\ \rmd X_t = a_t \rmd t + b_t \rmd B_t}} \mathbb E\!\left[ \int_0^1 \frac12|a_t|^2 + \frac\beta2|b_t - I_d|_{\rm HS}^2 \, \rmd t\right] 
        \label{eq:SB0}\\
        &= 
        \min_{\pi \in \cpl(\mu,\nu)} \int C_{\rm SB}^\beta(x,\pi_x) \, \mu(\rmd x)
        \label{eq:SB1} \\
        &= \max_{\alpha \in \cP_2(\R^d)} \left\{-W_\beta(\mu,\alpha) + \inf_{\rho \in \cP_2(\R^d)} \bigl\{ V_{\rm EOT}(\alpha,\rho) + W_\beta(\rho,\nu)\bigr\}\right\} 
        \label{eq:SB2} \\
        &= \max_{\substack{f\in L^1(\nu), \\ \text{ $\beta$-semiconcave}}} \left\{ \int f \, \rmd \nu - \int q_\beta \Box(-\Tcal^\beta_1[f]) \, \rmd \mu\right\},
        \label{eq:SB3}
    \end{align}
    where $\Tcal^\beta_1[f] := -\log\bigl(\exp(-q_\beta\Box (-f))\ast\gamma)\bigr)$ for $f \in L^1(\nu)$.
    
    Moreover, the weak transport problem~\eqref{eq:SB1} admits a unique optimizer $\pi^\circ \in \cpl(\mu,\nu)$, and $V_{\rm SB}^\beta(\mu,\nu)$ is attained by a semimartingale $X^\circ = (X_t^\circ)_{t \in [0,1]}$, unique in law. The variational problem~\eqref{eq:SB2} admits unique optimizers ${\alpha^\circ,\rho^\circ \in \cP_2(\Rd)}$. Finally, the dual problem~\eqref{eq:SB3} admits a $\beta$-semiconcave maximizer $f^\circ \in L^1(\nu)$, unique $\nu$-a.e.\ up to an additive constant.

    Let $g_1 := \exp\bigl(-q_\beta\Box(-f^\circ)\bigr)$, $g_t := g_1 * \gamma_{0;1-t}$ for $t \in [0,1]$.
    The Schrödinger--Bass system is characterized by
    \begin{gather*}
        u^\circ = q_1-\frac{1}{\beta} q_\beta \Box \bigl(-\Tcal^\beta_1[f^\circ]), 
        \qquad
        \alpha^\circ = (\nabla u^\circ)_\# \mu, 
        \\
        \chi_y^\circ(\rmd z) = \frac{g_1(z)}{g_0(y)} \gamma_y(\rmd z) \quad \alpha^\circ\text{-a.e.}, 
        \qquad
        \rho^\circ(\rmd z) = \int \chi_y^\circ(\rmd z) \, \alpha^\circ(\rmd y),
        \\
        v^\circ := q_1 -\tfrac{1}{\beta}f^\circ,
        \qquad
        \nu = (\nabla (v^\circ)^*)_\# \rho^\circ,
    \end{gather*}
    Moreover, $X^\circ = (X_t^\circ)_{t \in [0,1]}$ is given by
    \begin{align*}
        \rmd X_t^\circ
        &= \nabla \log g_t(Y_t^\circ) \,\rmd t
        + \bigl(I_d + \tfrac{1}{\beta}\nabla^2 \log  g_t(Y_t^\circ)\bigr)\,\rmd B_t,
        \quad X_0^\circ \sim \mu,
        \\
        \rmd Y_t^\circ
        &= \nabla \log g_t(Y_t^\circ)\,\rmd t + \rmd B_t,
        \quad Y^\circ_0 = \nabla u^\circ(X_0^\circ).
    \end{align*}
\end{theorem}

Before turning to the proof of \cref{thm:SB}, we study the semimartingale transport problem $C_{\rm SB}^\beta$ defined by~\eqref{eq:SB.cost.dynamic}, that is, the Schrödinger--Bass problem with initial law $\delta_x$. The next lemma shows that $C_{\rm SB}^\beta$ is a continuous standard weak transport cost and that, for every $(x,\eta) \in \Rd \times \cP_2(\Rd)$, the infimum defining $C_{\rm SB}^\beta(x,\eta)$ is attained.

\begin{lemma}\label{lem:SB.cost}
    For $x \in \Rd$ and $\eta \in \cP_2(\Rd)$, define
    \begin{equation}\label{eq:SB.wot.cost}
        C_{\rm stat}^\beta(x,\eta)
        := 
        \sup_{y \in \Rd}
        \left\{-q_\beta(x-y) + 
        \inf_{\rho \in \cP_2(\Rd)} 
        \bigl\{H(\rho \,|\, \gamma_y)+ W_\beta(\rho,\eta)\bigr\}
        \right\}.
    \end{equation}
    Then the following hold:
    \begin{enumerate}[label = (\roman*)]
        \item\label{it:SB.cost:1}
        For every $(x,\eta) \in \Rd \times \cP_2(\Rd)$, $C_{\rm SB}^\beta(x,\eta) = C_{\rm stat}^\beta(x,\eta)$.
        In particular, $C_{\rm SB}^\beta$ is a continuous standard weak transport cost and admits a $2$-growth bound.
        Moreover, for every $x \in \Rd$, the map $\eta \mapsto C_{\rm SB}^\beta(x,\eta)$ is strictly convex.

        \item\label{it:SB.cost:2}
        Let $(x,\eta) \in \Rd \times \cP_2(\Rd)$, and let $f^\circ \in L^1(\eta)$ attain \eqref{eq:SB.dual} for the pair ${(\delta_x,\eta)}$. Then,
        \[
            C_{\rm SB}^\beta(x,\eta) = \int f^\circ\,\rmd\eta - q_\beta\Box(-\Tcal^\beta_1[f^\circ])(x).
        \]
        Moreover, set $g_1 := \exp\bigl(-q_\beta \Box(-f^\circ)\bigr)$ and $g_t := g_1 * \gamma_{0;1-t}$.
        Then the infimum in \eqref{eq:SB.cost.dynamic} defining $C_{\rm SB}^\beta(x,\eta)$ is attained by $X^\circ = (X^\circ_t)_{t \in [0,1]}$ satisfying
        \begin{align*}
            \rmd X^\circ_t
            &= \nabla \log g_t(Y^\circ_t) \,\rmd t + \bigl(I_d + \tfrac{1}{\beta}\nabla^2 \log g_t(Y^\circ_t) \bigr) \, \rmd B_t, 
            \quad X^\circ_0 = x,
            \\
            \rmd Y^\circ_t &= \nabla \log g_t(Y^\circ_t) \,\rmd t + \rmd B_t,
            \quad Y^\circ_0 = x - \tfrac{1}{\beta}\bigl(\nabla q_\beta \Box(-\Tcal^\beta_1[f^\circ])\bigr)(x).
        \end{align*}
    \end{enumerate}
\end{lemma}

\begin{proof}
    Fix $x \in \Rd$ and $\eta \in \Pcal_2(\Rd)$.

    \emph{Claim 1: $C_{\rm stat}^\beta(x,\eta) \le C_{\rm SB}^\beta(x,\eta)$.}

    To prove \emph{Claim 1}, fix $y \in \Rd$, and let $X = (X_t)_{t \in [0,1]}$ and $Y = (Y_t)_{t \in [0,1]}$ be semimartingales such that
    \begin{align*}
        \rmd X_t &= a_t \, \rmd t + b_t \, \rmd B_t, \quad X_0 = x, \ X_1 \sim \eta, 
        \\
        \rmd Y_t &= a_t \, \rmd t + \rmd B_t, \quad Y_0 = y,
    \end{align*}
    where $(a_t)_{t \in [0,1]}$ and $(b_t)_{t \in [0,1]}$ are progressively measurable satisfying
    \[
        \mathbb E\!\left[\int_0^1|a_t|^2 + |b_t-I_d|_{\rm HS}^2 \,\rmd t\right] < \infty.
    \]
    Set $\rho := \mathrm{Law}(Y_1)$. Then $\rho \in \cP_2(\Rd)$.
    Moreover, It\^o's isometry yields
    \begin{align*}
        \mathbb E\!\left[ \int_0^1 |a_t|^2 + \beta |b_t - I_d|_{\rm HS}^2 \, \rmd t \right] &=
        \mathbb E\!\left[ \int_0^1 |a_t|^2 \, \rmd t \right] + \beta\,\mathbb E\!\left[ \left|\int_0^1 b_t \, \rmd B_t - B_1\right|^2\right] \\
        &= \mathbb E\!\left[ \int_0^1 |a_t|^2 \, \rmd t \right] + \beta\,\mathbb E\!\left[ |X_1 - Y_1|^2 -|x-y|^2\right].
    \end{align*}
    Since $\mathrm{Law}(Y_1,X_1) \in \Cpl(\rho,\eta)$, we have $\Wcal_2^2(\rho,\eta) \le \mathbb E[|X_1-Y_1|^2]$.
    Moreover, by Föllmer's drift representation (see e.g.~\cite[Proposition~1]{Lehec2013}),
    \[
        H(\rho \,|\, \gamma_y) \le \frac12 \,\mathbb E\!\left[ \int_0^1 |a_t|^2 \, \rmd t \right],
    \]
    so that
    \[
        \inf_{\rho \in \cP_2(\R^d)} \bigl\{  -q_\beta(x-y) + H(\rho\,|\,\gamma_y) + W_\beta(\rho,\eta) \bigr\} \le 
        \frac12 \,\mathbb E\!\left[ \int_0^1 |a_t|^2 + \beta |b_t - I_d|_{\rm HS}^2 \, \rmd t \right].
    \]
    Taking the supremum over $y \in \Rd$ yields \emph{Claim 1}.
    
    By \Cref{prop:SB_deconv_attained}, there exists a $\beta$-semiconcave function $f^\circ \in L^1(\eta)$ such that
    \begin{align*}
        \int f^\circ \,\rmd\eta - q_\beta \Box \bigl(-\Tcal^\beta_1[f^\circ]\bigr)(x)
        &=
        \max_{\substack{f \in L^1(\eta), \\ \beta\text{-semiconcave}}}
        \left\{
            \int f \,\rmd\eta - q_\beta \Box \bigl(-\Tcal^\beta_1[f]\bigr)(x)
        \right\} \\
        &=
        \max_{\alpha \in \cP_2(\Rd)}
        \bigl\{
            -W_\beta(\delta_x,\alpha) + V_{\rm EOT} \Box W_\beta(\alpha,\eta)
        \bigr\}.
    \end{align*}

    \emph{Claim 2: \quad 
    $\displaystyle C_{\rm stat}^\beta(x,\eta)
        =
        \int f^\circ \,\rmd\eta - q_\beta \Box \bigl(-\Tcal^\beta_1[f^\circ]\bigr)(x).$
    }

    To show \emph{Claim 2}, set $u:=q_1 - \frac{1}{\beta}q_\beta \Box(-\Tcal^\beta_1[f^\circ])$.
    By \Cref{cor:smooth_brenier_map}, the function $u$ is convex and $\frac{1+\beta}{\beta}$-smooth. Let $y^\circ = \nabla u(x)$. Then, by the Fenchel--Legendre duality,
    \[
        W_\beta(\delta_x,\delta_{y^\circ}) = q_\beta(x-y^\circ) = q_\beta \Box(-\Tcal^\beta_1[f^\circ])(x) + \Tcal^\beta_1[f^\circ](y^\circ),
    \]
    so that, by duality from \Cref{thm:fundamental_inf_conv},
    \begin{align*}
        \int f^\circ\,\rmd\eta - q_\beta \Box(-\Tcal^\beta_1[f^\circ])(x) 
        &=
        - q_\beta(x-y^\circ) + \int f^\circ\,\rmd\eta + \Tcal^\beta_1[f^\circ](y^\circ) \\
        &\le 
        - q_\beta(x-y^\circ) + \inf_{\rho \in \cP_2(\Rd)}\bigl\{H(\rho\,|\,\gamma_{y^\circ}) + W_\beta(\rho,\eta)\bigr\} \\
        &\le C_{\rm stat}^\beta(x,\eta).
    \end{align*}
    On the other hand, by the optimality of $f^\circ$,
    \begin{align*}
        \int f^\circ\,\rmd\eta - q_\beta \Box(-\Tcal^\beta_1[f^\circ])(x) 
        &\ge \sup_{y \in \Rd} \bigl\{-W_\beta(\delta_x,\delta_{y}) + \inf_{\rho \in \cP_2(\Rd)}\left\{H(\rho\,|\,\gamma_y) + W_\beta(\rho,\eta)\right\}\bigr\} \\
        &= C_{\rm stat}^\beta(x,\eta).
    \end{align*}
    Therefore,
    \[
        C_{\rm stat}^\beta(x,\eta)
        =
        \int f^\circ \,\rmd\eta - q_\beta \Box \bigl(-\Tcal^\beta_1[f^\circ]\bigr)(x).
    \]
    This proves \emph{Claim 2}.

    By \Cref{thm:fundamental_inf_conv}, there exists $\rho^\circ \in \cP_2(\Rd)$ such that
    \[
        \inf_{\rho \in \cP_2(\Rd)}
        \bigl\{
            H(\rho\,|\,\gamma_{y^\circ}) + W_\beta(\rho,\eta)
        \bigr\}
        =
        H(\rho^\circ\,|\,\gamma_{y^\circ}) + W_\beta(\rho^\circ,\eta).
    \]
    Moreover, by \Cref{rem:SB_deconv_attained.optimizers}, $\rho^\circ$ admits density
    \[
        \frac{\rmd \rho^\circ}{\rmd \gamma_{y^\circ}}(\rmd z) = \frac{g_1(z)}{g_0(y^\circ)},
    \]
    where $g_1 := \exp\bigl(-q_\beta \Box(-f^\circ)\bigr)$ and $g_t := g_1 * \gamma_{0;1-t}$.
    Furthermore, $v^* := q_1 + \frac{1}{\beta}\log g_1$ satisfies $(\nabla v^*)_\#\rho^\circ = \eta$. In particular, $v^*$ is a real-valued convex function, and hence differentiable a.e.

    \emph{Claim 3: There exist semimartingales $X = (X_t)_{t \in [0,1]}$ and $Y = (Y_t)_{t \in [0,1]}$ with
    \begin{align*}
        \rmd X_t &= a_t \rmd t + b_t \rmd B_t, \quad X_0 = x, \ X_1 \sim \eta, \\
        \rmd Y_t &= a_t \rmd t + \rmd B_t, \quad Y_0 = y^\circ, \ Y_1 \sim \rho^\circ
    \end{align*}
    such that
    \[
        \frac12 \mathbb E\!\left[\int_0^1 |a_t|^2 + \beta |b_t-I_d|_{\mathrm{HS}}^2 \,\rmd t\right]
        = C_{\rm stat}^\beta(x,\eta).
    \]
    In particular, $C_{\rm SB}^\beta(x,\eta) \le C_{\rm stat}^\beta(x,\eta)$.}

    To prove \emph{Claim 3}, let $\mathbb Q$ be a probability measure on $C([0,1];\Rd)$ under which the canonical process $Y$ is a standard Wiener process started at $y^\circ$. 
    Set $a_t := \nabla \log g_t(Y_t)$, and define $\mathbb P \sim \mathbb Q$ by
    \[
        \frac{\rmd\mathbb P}{\rmd \mathbb Q}
        :=
        \frac{g_1(Y_1)}{g_0(y^\circ)}
        =
        \exp\left(
            \int_0^1 a_t \,\rmd Y_t
            - \frac12 \int_0^1 |a_t|^2 \,\rmd t
        \right).
    \]
    By Girsanov's theorem, $B_t := Y_t - \int_0^t a_s \,\rmd s$ is a Brownian motion under $\mathbb P$, and so $\rmd Y_t = a_t \rmd t + \rmd B_t$.
    Applying It\^o's formula to $\log g_t(Y_t)$, we obtain
    \[
        \rmd \log g_t(Y_t) 
        =
        a_t \,\rmd Y_t - \frac12 |a_t|^2 \,\rmd t
        =
        a_t \,\rmd B_t + \frac12 |a_t|^2 \,\rmd t.
    \]
    Taking expectations under $\mathbb P$, it follows that
    \[
        \frac12 \mathbb E_\mathbb P\!\left[\int_0^1 |a_t|^2 \,\rmd t\right]
        =
        \mathbb E_\mathbb P\!\left[\log g_1(Y_1) - \log g_0(y^\circ)\right]
        =
        \mathbb E_\mathbb Q\!\left[\log\!\left(\frac{g_1(Y_1)}{g_0(y^\circ)}\right) \frac{g_1(Y_1)}{g_0(y^\circ)}\right]
        =
        H(\rho^\circ \,|\, \gamma_{y^\circ}).
    \]
    This is F\"ollmer's construction; see, e.g., \textcite{Foellmer1985,Lehec2013}. 
    
    Since $(g_t)_{t \in[0,1]}$ solves the backward heat equation, $\bigl(\nabla g_t(Y_t)\bigr)_{t \in (0,1]}$ is a $\mathbb Q$-martingale. Hence
    \[
        \nabla \log g_t(Y_t) = \frac{\nabla g_t(Y_t)}{g_t(Y_t)}
    \]
    is a $\mathbb P$-martingale. 
    For $t \in (0,1]$, define $X_t := Y_t + \frac1\beta \nabla \log g_t(Y_t)$ and set $X_0 := \lim_{t \downarrow 0} X_t$. Then It\^o's formula yields
    \[
        \rmd X_t
        =
        \rmd Y_t + \tfrac1\beta \nabla^2 \log g_t(Y_t) \,\rmd B_t
        =
        a_t\,\rmd t + \left(I_d + \tfrac1\beta \nabla^2 \log g_t(Y_t)\right)\rmd B_t.
    \]
    Since $X_1 = \nabla v^*(Y_1)$ and $Y_1 \sim \rho^\circ$ under $\mathbb P$, we have $X_1 \sim \eta$.
    Moreover,
    \[
        X_0 
        = 
        Y_0 + \frac1\beta \mathbb E_\mathbb P\!\left[\nabla \log g_t(Y_t) \right], \qquad t \in (0,1].
    \]
    Since $g_1$ is differentiable a.e.\ and $\nabla g_1$ is its weak derivative, integration by parts gives
    \[
        \mathbb E_\mathbb Q[\nabla g_1(Y_1)] = \mathbb E_\mathbb Q[(Y_1-Y_0)g_1(Y_1)] = 
        \mathbb E_\mathbb P[Y_1] - Y_0.
    \]
    In particular,
    \begin{align*}
        \mathbb E_\mathbb P\!\left[\nabla \log g_t(Y_t) \right] 
        &= 
        \mathbb E_{\mathbb Q}[\nabla g_t(Y_t)] 
        = 
        \mathbb E_\mathbb Q[\nabla g_t(Y_1)]
        = 
        \mathbb E_\mathbb P[Y_1] - Y_0,
    \end{align*}
    and therefore
    \[
        X_0 = y^\circ + \frac{1}{\beta}(\bar\rho^\circ-y^\circ) = x,
    \]
    where the last identity follows from \Cref{rem:stat.cost.optimizers}. 
    Finally, It\^o's isometry yields
    \begin{align*}
        \frac{\beta}{2} \mathbb E_\mathbb P\!\left[\int_0^1 |b_t - I_d|^2 \,\rmd t\right] 
        &= 
        \frac{\beta}{2}\mathbb E_\mathbb P\!\left[\left|\int_0^1 b_t \,\rmd B_t - B_1\right|^2\right] = \frac{\beta}{2}\mathbb E_\mathbb P\!\left[|X_1-Y_1|^2 - |x-y^\circ|^2\right]
        \\
        &= 
        -q_\beta(x-y^\circ) + \frac{\beta}{2}\mathbb E_\mathbb P\!\left[|\nabla v^*(Y_1)-Y_1|^2\right] = -q_\beta(x-y^\circ) + W_\beta(\rho^\circ,\eta),
    \end{align*}
    and hence
    \[
        \frac{1}{2} \mathbb E_\mathbb P\!\left[\int_0^1 |a_t|^2 + \beta |b_t - I_d|^2 \,\rmd t\right] = -q_\beta(x-y^\circ) + H(\rho^\circ \,|\, \gamma_{y^\circ}) + W_\beta(\rho^\circ,\eta) = C_{\rm stat}^\beta(x,\eta).
    \]
    This proves \emph{Claim 3}.
    
    Therefore, $C_{\rm SB}^\beta(x,\eta) = C_{\rm stat}^\beta(x,\eta)$ for all $x \in \Rd$, $\eta \in \cP_2(\Rd)$. The asserted properties of $C_{\rm SB}^\beta$ now follow directly from \Cref{lem:stat.cost}.
\end{proof}

With \cref{lem:SB.cost} at hand, we can now prove \cref{thm:SB}.

\begin{proof}[Proof of \Cref{thm:SB}]
    Equality of \eqref{eq:SB2} and \eqref{eq:SB3}, as well as existence of optimizers $\alpha^\circ,\rho^\circ \in \cP_2(\Rd)$ and a $\beta$-semiconcave maximizer $f^\circ \in L^1(\nu)$ in \eqref{eq:SB3}, follow from \Cref{prop:SB_deconv_attained}. The relation between the optimizers $f^\circ$, $\alpha^\circ$ and $\chi^\circ$ is given in \Cref{rem:SB_deconv_attained.optimizers}.

    By duality for the standard weak transport problem, see \cite[Theorem~3.1]{backhoff2019existence}, we obtain
    \begin{equation}\label{eq:proof:thm:SB:primal_dual_WOT}
        \inf_{\pi \in \Cpl(\mu,\nu)} \int C_{\rm SB}^\beta(x,\pi_x)\,\mu(\rmd x)
        =
        \sup_{f \in C_{b,2}(\Rd)} \left\{
            \int f \,\rmd \nu + \int f^{C_{\rm SB}^\beta} \,\rmd \mu
        \right\}.
    \end{equation}
    Since $C_{\rm SB}^\beta$ is continuous and satisfies a $2$-growth bound, \cite[Theorem~2.9]{backhoff2019existence} yields existence of a primal optimizer. Moreover, for every $x \in \Rd$, the map $\eta \mapsto C_{\rm SB}^\beta(x,\eta)$ is strictly convex by \Cref{lem:SB.cost}. Hence the weak transport problem admits a unique optimizer $\pi^\circ \in \Cpl(\mu,\nu)$.
    
    We claim that $C_{\rm SB}^\beta(x,\pi^\circ_x) = \int f^\circ\,\rmd\pi^\circ_x - q_\beta\Box(-\Tcal^\beta_1[f^\circ])(x)$ for $\mu$-a.e.\ $x$. Set $u^\circ := q_1 - \frac{1}{\beta} q_\beta\Box(-\Tcal^\beta_1[f^\circ])$, so that by the Fenchel--Legendre duality, for all $x \in \Rd$,
    \[
        q_\beta\bigl(x - \nabla u^\circ(x)\bigr) 
        =
        q_\beta\Box(-\Tcal^\beta_1[f^\circ])(x) + \Tcal^\beta_1[f^\circ](\nabla u^\circ(x)).
    \]
    Fix $x \in \Rd$. By duality of the infimal convolution, see \Cref{thm:fundamental_inf_conv}, and since $f^\circ$ is an admissible dual candidate, we obtain
    \[  
        \inf_{\rho \in \cP_2(\Rd)} \left\{H\bigr(\rho \,|\,\gamma_{\nabla u^\circ(x)}\bigr) + W_\beta(\rho,\pi^\circ_x)\right\}
        \geq
        \int f^\circ\,\rmd\pi^\circ_x +\Tcal^\beta_1[f^\circ](\nabla u^\circ(x)).
    \]
    Therefore,
    \begin{align*}
        \int f^\circ\,\rmd\nu + \int \Tcal^\beta_1[f^\circ]\,\rmd\alpha^\circ
        &\leq 
        \int \inf_{\rho \in \cP_2(\Rd)}\left\{H\bigr(\rho \,|\,\gamma_{\nabla u^\circ(x)}\bigr) + W_\beta(\rho,\pi^\circ_x)\right\}\,\mu(\rmd x) \\
        &\leq
        \int H\bigr(\chi^\circ_{\nabla u^\circ(x)} \,|\,\gamma_{\nabla u^\circ(x)}\bigr) + W_\beta(\chi^\circ_{\nabla u^\circ(x)},\pi^\circ_x) \, \mu(\rmd x) \\
        &= \int \left(\int f^\circ\,\rmd\pi^\circ_x +\Tcal^\beta_1[f^\circ](\nabla u^\circ(x))\right)\, \mu(\rmd x),
    \end{align*}
    where the last identity follows from complementary slackness; see \Cref{thm:fundamental_inf_conv}. Hence all inequalities are equalities.
    For $x \in \Rd$, since the map $\rho \mapsto H\bigl(\rho\,|\,\gamma_{\nabla u^\circ(x)}\bigr)$ is strictly convex, it follows that $\chi_{\nabla u^\circ(x)}$ is the unique minimizer of
    \[
        \inf_{\rho \in \cP_2(\Rd)} \left\{H\bigr(\rho \,|\,\gamma_{\nabla u^\circ(x)}\bigr) + W_\beta(\rho,\pi^\circ_x)\right\}.
    \]
    Moreover,
    \begin{align*}
        \int f^\circ\,\rmd\pi^\circ_x - q_\beta\Box(-\Tcal^\beta_1[f^\circ])(x) 
        &= 
        - q_\beta\bigl(x - \nabla u^\circ(x)\bigr) 
        + 
        \inf_{\rho \in \cP_2(\Rd)}\left\{H\bigr(\rho \,|\,\gamma_{\nabla u^\circ(x)}\bigr) + W_\beta(\rho,\pi^\circ_x)\right\} \\
        &=
        \sup_{\alpha \in \cP_2(\Rd)}
        \left\{-W_\beta(\delta_x,\alpha) 
        + \inf_{\rho \in \cP_2(\Rd)}\left\{V_{\rm EOT}(\alpha,\rho) + W_\beta(\rho,\pi^\circ_x)\right\}\right\},
    \end{align*}
    where the last identity follows from \Cref{prop:SB_deconv_attained}. 
    In particular, for $\mu$-a.e.\ $x$ the function $f^\circ$ attains \eqref{eq:SB.variational} for the pair $(\delta_x,\pi^\circ_x)$. We may therefore apply \Cref{lem:SB.cost} to conclude that, for $\mu$-a.e.\ $x$,
    \begin{equation}\label{eq:proof:thm:pw.dual.equal}
        C^\beta_{\rm SB}(x,\pi^\circ_x) = \int f^\circ\,\rmd\pi^\circ_x - q_\beta\Box(-\Tcal^\beta_1[f^\circ])(x).
    \end{equation}
    By \Cref{lem:stat.cost.dual.inequ} and by definition of $f^\circ$, we also have
    \begin{align*}
        \sup_{f\in C_{b,2}(\Rd)}\left\{\int f \,\rmd\nu + \int f^{C_{\rm SB}^\beta} \,\rmd\mu\right\} 
        &\geq 
        \sup_{f\in C_{b,2}(\Rd)}\left\{\int f \,\rmd\nu - \int q_\beta\Box(-\Tcal^\beta_1[f]) \,\rmd\mu\right\}, \\
        &= \int f^\circ \,\rmd\nu - \int q_\beta\Box(-\Tcal^\beta_1[f^\circ]) \,\rmd\mu.
    \end{align*}
    Combined with \eqref{eq:proof:thm:SB:primal_dual_WOT}, this yields
    \[
        \inf_{\pi \in \Cpl(\mu,\nu)} \int C^\beta_{\rm SB}(x,\pi_x)\,\mu(\rmd x) = \int f^\circ \,\rmd\nu - \int q_\beta\Box(-\Tcal^\beta_1[f^\circ]) \,\rmd\mu,
    \]
    and so \eqref{eq:SB1} equals \eqref{eq:SB2} and \eqref{eq:SB3}.

    Furthermore, set $g_1 := \exp\bigl(-q_\beta \Box(-f^\circ)\bigr)$, $g_t := g_1 * \gamma_{0;1-t}$, and define $X^\circ = (X^\circ_t)_{t\in[0,1]}$ by
    \begin{align*}
        \rmd X^\circ_t &= 
        \nabla \log g_t(Y^\circ_t) \, \rmd t  +  \bigl(I_d + \tfrac{1}{\beta}\nabla^2 \log g_t(Y_t^\circ) \bigr) \, \rmd B_t,
        \quad X^\circ_0 \sim \mu, 
        \\
        \rmd Y^\circ_t  &= 
        \nabla \log g_t(Y^\circ_t) \,\rmd t  +  \, \rmd B_t, 
        \quad Y^\circ_0 = \nabla u^\circ(X^\circ_0).
    \end{align*}
    By \Cref{lem:SB.cost}, for $\mu$-a.e.\ $x$, the conditional law of $X^\circ$ given $X_0^\circ=x$ attains $C_{\rm SB}^\beta(x,\pi_x^\circ)$. Hence,
    \begin{align*}
        \int C_{\rm SB}^\beta(x,\pi^\circ_x)\, \mu(\rmd x)
        &=\mathbb E\!\left[
        \mathbb E\!\left[
        \int_0^1 \frac12 \left|\nabla \log g_t(Y^\circ_t) \right|^2 + \frac{1}{2\beta}\left|\nabla^2 \log g_t(Y^\circ_t)\right|_{\mathrm{HS}}^2 \,\rmd t
        \, \bigg|\, X^\circ_0\right]\right] \\
        &\ge V_{\rm SB}^\beta(\mu,\nu).
    \end{align*}
    By \eqref{inequ:SB.easy}, equality holds. In particular, \eqref{eq:SB0} equals \eqref{eq:SB1}, \eqref{eq:SB2}, and \eqref{eq:SB3}.
    
    We now turn to the characterization of the Schrödinger--Bass system. It suffices to show that $f^\circ$ is $\nu$-a.e.\ unique up to additive constants. 
    To this end, let $v^\circ:\Rd\to \R$ be the convex function defined by $f^\circ = q_\beta - \beta v^\circ$, so that $(\nabla (v^\circ)^*)_\# \rho^\circ = \nu$.
    By the above observations, for $\mu$-a.e.\ $x$, the measure $\pi^\circ_x$ satisfies \eqref{eq:proof:thm:pw.dual.equal}.
    Moreover, by \Cref{rem:stat.cost.optimizers}, for $\mu$-a.e.\ $x$, the point $\nabla u^\circ(x) = x - \frac{1}{\beta} \nabla q_\beta \Box(-\Tcal^\beta_1[f^\circ])(x)$ is the unique solution of
    \[
        C_{\rm SB}^\beta(x,\pi^\circ_x) = \sup_{y \in \Rd}\left\{-q_\beta(x-y) + \inf_{\rho \in \cP_2(\Rd)}\bigl\{H(\rho \, |\, \gamma_y) + W_\beta(\rho,\pi^\circ_x)\bigr\}\right\},
    \]
    and $\chi^\circ_{\nabla u^\circ(x)}$ uniquely attains
    \[
        \inf_{\rho \in \cP_2(\Rd)}\left\{H\bigl(\rho \, |\, \gamma_{\nabla u^\circ(x)}\bigr) + W_\beta(\rho,\pi^\circ_x)\right\}.
    \]
    Since $\chi^\circ_{\nabla u^\circ(x)}$ is characterized by its density
    \[
        \frac{\rmd \chi^\circ_{\nabla u^\circ(x)}}{\rmd \gamma_{\nabla u^\circ(x)}} \propto \exp\bigl(-q_\beta \Box(-f^\circ)\bigr),
    \]
    which is unique up to Lebesgue-null sets, the function $-q_\beta\Box(-f^\circ)$ is uniquely determined up to an additive constant, which is in fact independent of $x$. Since $f^\circ$ is $\beta$-semiconcave, it follows that $f^\circ$ is determined $\nu$-a.e.\ up to an additive constant.

    Finally, we prove uniqueness of $\operatorname{Law}(X^\circ)$ in $\cP\bigl(\mathcal C([0,1];\Rd)\bigr)$. Let $(\Omega,\mathbb P)$ support a $d$-dimensional Brownian motion $B$ and an initial draw from $\mu$, independent of $B$, and let $X = (X_t)_{t\in [0,1]}$ with $\rmd X_t = a_t\,\rmd t+b_t\,\rmd B_t$ be any admissible optimizer for \eqref{eq:SB0}. Define $Y = (Y_t)_{t \in [0,1]}$ by
    \[
        \rmd Y_t = a_t\, \rmd t + \rmd B_t, \quad Y_0 = \nabla u^\circ(X_0),
    \]
    and set $\rho := \operatorname{Law}(Y_1)$. By It\^o's isometry,
    \[
        \mathbb E\!\left[\int_0^1 \frac\beta2|b_t-I_d|^2_{\rm HS} \, \rmd t\right]
        =
        \mathbb E\!\left[\frac{\beta}{2}\left(|X_1-Y_1|^2-|X_0-Y_0|^2\right)\right] \ge W_\beta(\rho,\nu) - W_\beta(\mu,\alpha^\circ).
    \]
    Moreover,
    \[
        V_{\rm EOT}(\alpha^\circ,\rho) \leq \mathbb E\!\left[\int_0^1\frac{1}{2}|a_t|^2 \,\rmd t\right].
    \]
    Combining the previous two inequalities yields
    \[
        V_{\rm SB}^\beta(\mu,\nu) \ge - W_\beta(\mu,\alpha^\circ) + W_\beta(\rho,\nu) + V_{\rm EOT}(\alpha^\circ,\rho).
    \]
    Since \eqref{eq:SB0} equals \eqref{eq:SB2}, it follows that $\rho=\rho^\circ$. Hence all inequalities are equalities, and $Y$ is a Schrödinger bridge from $\alpha^\circ$ to $\rho^\circ$. Since $\rho^\circ$ is absolutely continuous and $\operatorname{Law}(Y_1,X_1)$ attains $\Wcal_2^2(\rho^\circ,\nu)$, Brenier's theorem gives $X_1= \nabla (v^\circ)^*(Y_1)$ a.s.

    Now define $\tilde X = (\tilde X_t)_{t \in [0,1]}$ by $\tilde X_t := Y_t + \frac{1}{\beta}\nabla \log g_t(Y_t)$ for $t \in (0,1]$ and set $\tilde X_0 := \lim_{t \downarrow 0} \tilde X_t$. By construction, $\operatorname{Law}(\tilde X) = \operatorname{Law}(X^\circ)$. Writing $\tilde b_t := I_d + \frac1\beta \nabla^2\log g_t(Y_t)$, we have
    \[
        \rmd \tilde X_t = a_t \,\rmd t + \tilde b_t \,\rmd B_t,
        \qquad 
        \tilde X_0 = \nabla (u^\circ)^*(Y_0) = X_0 \ \text{ a.s.},
        \qquad
        \tilde X_1 = \nabla (v^\circ)^* (Y_1) = X_1 \ \text{ a.s.}
    \]
    Set $\hat X_t := \frac{X_t + \tilde X_t}{2}$ and $\hat b_t := \frac{b_t+\tilde b_t}{2}$. Then
    \[
        \hat X_t = a_t \, \rmd t + \hat b_t \, \rmd B_t, \qquad 
        \hat X_0 = X_0 \ \text{ a.s.}, 
        \qquad \hat X_1 = X_1 \ \text{ a.s.},
    \]
    and hence $\hat X$ is admissible for \eqref{eq:SB0}. By convexity of the Hilbert--Schmidt norm,
    \[
        \frac\beta2|\hat b_t-I_d|^2_{\rm HS} \le \frac\beta4| b_t-I_d|^2_{\rm HS} + \frac\beta4|\tilde b_t-I_d|^2_{\rm HS}
    \]
    and thus
    \[
        \mathbb E\!\left[\int_0^1 \frac{1}{2}|a_t|^2+ \frac\beta2|\hat b_t-I_d|^2_{\rm HS} \, \rmd t\right] \leq V_{\rm SB}^\beta(\mu,\nu).
    \]
    If $b \neq \tilde b$ on a set of positive $\mathbb P$-measure, the inequality is strict. Thus $b = \tilde b \,$ $\mathbb P$-a.e., and so $X=\tilde X$ up to indistinguishability. In particular, $\operatorname{Law}(X) = \operatorname{Law}(X^\circ)$.
\end{proof}

\begin{remark}\label{rem:SB.c_transform}
    For general $\beta$-semiconcave $f$, \cref{lem:stat.cost.dual.inequ} yields only
    \[
        f^{C_{\rm SB}^\beta}(x) \ge -q_\beta \Box (-\Tcal^\beta_1[f])(x).
    \]
    If $f^\circ$ is an optimizer for \eqref{eq:SB.dual}, however, the preceding proof shows that, for $\mu$-a.e.\ $x$,
    \[
        (f^\circ)^{C_{\rm SB}^\beta}(x)
        =
        \inf_{\eta \in \cP_2(\Rd)}
        \left\{
            C_{\rm SB}^\beta(x,\eta) - \int f^\circ\,\rmd\eta
        \right\}
        =
        -q_\beta \Box \bigl(-\Tcal^\beta_1[f^\circ]\bigr)(x).
    \]
    If $\beta > 1$, the identity can also be obtained by a min--max argument, since for every $x \in \Rd$ and $\eta \in \cP_2(\Rd)$,
    \[
        y \longmapsto \Bigl(-q_\beta(x-y) + \inf_{\rho \in \cP_2(\Rd)}
        \bigl\{
            H(\rho \mid \gamma_y) + W_\beta(\rho,\eta)
        \bigr\}\Bigr)
    \]
    is strongly concave. Remarkably, despite the loss of concavity, the identity continues to hold for $\beta \in (0,1]$.
\end{remark}

We conclude this section with an application of the characterization provided by the Schrödinger--Bass system. In the Gaussian-to-Gaussian case, it allows us to determine the Schrödinger-Bass bridge explicitly.

\begin{example}[The Schrödinger--Bass bridge between Gaussians]\label{ex:SB.bridge.Gaussians}
    Let $\beta>0$, $m_0,m_1\in\Rd$. Let $\Sigma_0, \Sigma_1 \in \R^{d\times d}$ be symmetric and positive semidefinite. Set $\mu=\mathcal N(m_0,\Sigma_0)$, $\nu=\mathcal N(m_1,\Sigma_1)$.
    We write $\Sigma\succ0$ and $\Sigma\succeq0$ if $\Sigma$ is positive definite and positive semidefinite, respectively.
    
    Let $v^\circ:=q_1-\frac{1}{\beta}f^\circ$ and consider the ansatz $(v^\circ)^*(y) = \frac12 y^TCy+d^Ty$ where $C \in \R^{d\times d}$ is symmetric and positive semidefinite and $d\in\Rd$. Then
    \[
        q_\beta\Box(-f^\circ)(y) = q_\beta(y)-\beta (v^\circ)^*(y) = \frac{\beta}{2}y^T(I_d-C)y-\beta d^Ty.
    \]
    Suppose that $P:=(\beta+1)I_d-\beta C\succ0$.
    Convolution with $\gamma$ yields
    \[
        -\Tcal^\beta_1[f^\circ](y) 
        = 
        \frac12 y^T(P^{-1}-I_d)y + \beta d^TP^{-1}y + c, \qquad c\in\R.
    \]
    Since $C\succeq0$ and $P\succ0$, the matrix $Q:=\beta I_d-(\beta-1)C$ is positive definite. A further infimal convolution gives
    \[
        q_\beta\Box\bigl(-\Tcal^\beta_1[f^\circ]\bigr)(x)
        =
        -\frac{\beta}{2}x^T(I_d-C)Q^{-1}x + \beta d^TQ^{-1}x + \tilde c, \qquad \tilde c\in\R.
    \]
    Consequently, 
    \[ 
        \nabla u^\circ(x) = \bigl(I_d+(I_d-C)Q^{-1}\bigr)x-Q^{-1}d = PQ^{-1}x-Q^{-1}d. 
    \] 
    Since $\alpha^\circ=(\nabla u^\circ)_\#\mu$ and $\chi^\circ_y = \mathcal N\bigl(P^{-1}(y+\beta d),P^{-1}\bigr)$, the measure $\rho^\circ= \int \chi^\circ_y \, \alpha^\circ(\rmd y)$ is Gaussian with mean and covariance 
    \[ 
        m_{\rho^\circ} = Q^{-1}\bigl(m_0+(\beta-1)d\bigr), 
        \qquad 
        \Sigma_{\rho^\circ} = Q^{-1}\Sigma_0Q^{-1}+P^{-1}. 
    \]
    Since $\nu = (\nabla (v^\circ)^*)_\#\rho^\circ = \mathcal{N}(C m_{\rho^\circ}+d, C\Sigma_{\rho^\circ}C^T)$, comparison of means and covariances yields
    \begin{equation}\label{eq:ex:SB.bridge.Gaussians:1st.order.condition}
        d = \frac1\beta\left(Qm_1-Cm_0\right),
        \qquad
        \Sigma_1 = CQ^{-1}\Sigma_0Q^{-1}C+CP^{-1}C.
    \end{equation}

    It remains to solve the covariance equation in \eqref{eq:ex:SB.bridge.Gaussians:1st.order.condition}. 
    Suppose first that $\Sigma_1\succ0$, and set $R:=\beta C^{-1}-(\beta-1)I_d$ so that $C=\beta\bigl(R+(\beta-1)I_d\bigr)^{-1}$.
    Then $P\succ0$ is equivalent to $(\beta+1)R-I_d\succ0$, and, since $QC^{-1}=R$, the covariance equation in \eqref{eq:ex:SB.bridge.Gaussians:1st.order.condition} becomes
    \begin{equation}\label{eq:ex:SB.bridge.Gaussians:1st.order.condition.R}
        \Sigma_1 = R^{-1}\Sigma_0R^{-1} + \beta^2 \bigl(R+(\beta-1)I_d\bigr)^{-1}\bigl((\beta+1)R-I_d\bigr)^{-1}.
    \end{equation}
    On the open convex set $\mathcal D := \left\{ R\in\mathbb R^{d\times d}: R=R^T,\, (\beta+1)R-I_d\succ0 \right\}$, consider
    \[
        \mathcal J(R) := \tr(\Sigma_1R) + \tr(\Sigma_0R^{-1}) + \log\det\bigl(R+(\beta-1)I_d\bigr) - \log\det\bigl((\beta+1)R-I_d\bigr),
    \]
    The first-order condition $\nabla\mathcal J(R)=0$ is precisely \eqref{eq:ex:SB.bridge.Gaussians:1st.order.condition.R}. 
    Moreover, $\mathcal J$ is strictly convex on $\mathcal D$ and tends to $+\infty$ both at the boundary of $\mathcal D$ and as $|R|_{\mathrm{HS}}\to \infty$. Hence $\mathcal J$ admits a unique minimizer $R^\circ\in\mathcal D$, and
    \[
        C^\circ = \beta\bigl(R^\circ+(\beta-1)I_d\bigr)^{-1}
    \]
    is the unique positive-definite solution of the covariance equation in \eqref{eq:ex:SB.bridge.Gaussians:1st.order.condition} with $(\beta+1)I_d-\beta C^\circ\succ0$.
    
    If $\Sigma_1$ is singular, let $E:=\operatorname{im}(\Sigma_1)$, so that $E^\perp=\ker(\Sigma_1)$.
    Since $\ker(\Sigma_1)=\ker(C)$, any solution is of the form $C=C_E\oplus0_{E^\perp}$ with $C_E\succ0$.
    Compressing the covariance equation in \eqref{eq:ex:SB.bridge.Gaussians:1st.order.condition} to $E$ yields an equation of the same form as in the positive-definite case.
    Applying the preceding argument on $E$ and extending the resulting solution by zero on $E^\perp$ yields the required positive-semidefinite solution on $\Rd$.

    Set 
    \[ 
        P^\circ := (\beta+1)I_d-\beta C^\circ, 
        \qquad 
        Q^\circ := \beta I_d-(\beta-1)C^\circ,
        \qquad 
        d^\circ := \frac1\beta\bigl(Q^\circ m_1-C^\circ m_0\bigr).
    \]
    We can now state the dynamics of the Schrödinger--Bass bridge $X^\circ$ explicitly. For $t\in[0,1]$, let
    \[
        P^\circ_t := tI_d+(1-t)P^\circ,
        \qquad
        Q^\circ_t := tC^\circ+(1-t)Q^\circ.
    \]
    Convolution with $\gamma_{0;1-t}$ gives
    \[
        \nabla\log g_t(y) = \beta(P_t^\circ)^{-1}\bigl((C^\circ-I_d)y+d^\circ\bigr),
        \qquad
        \nabla^2\log g_t(y) = \beta(P_t^\circ)^{-1}(C^\circ-I_d).
    \]
    Hence, by \cref{thm:SB},
    \begin{align*}
        \rmd X_t^\circ &= a_t^\circ\,\rmd t + b_t^\circ\,\rmd B_t, \quad X_0^\circ \sim \mathcal{N}(m_0,\Sigma_0).
    \end{align*}
    Since $X_t^\circ = Y_t^\circ + \frac1\beta\nabla \log g_t(Y_t^\circ)$, we have
    \[
        a_t^\circ(X_t^\circ) = \beta (Q_t^\circ)^{-1}\bigl((C^\circ-I_d)X_t^\circ+d^\circ\bigr),
        \qquad
        b_t^\circ = I_d+\frac1\beta\nabla^2\log g_t = Q_t^\circ(P_t^\circ)^{-1}.
    \]
    With $\frac{\rmd}{\rmd t} Q_t^\circ = \beta (C^\circ - I_d)$, variation of constants yields
    \[
        X_t^\circ = Q_t^\circ (Q^\circ)^{-1}X_0^\circ + \beta Q^\circ_t \int_0^t (Q_s^\circ)^{-2}d^\circ\,\rmd s + Q_t^\circ \int_0^t (P_s^\circ)^{-1} \,\rmd B_s.
    \]
    In particular, $X_t^\circ$ is Gaussian with mean and covariance
    \[
        m_t^\circ = (1-t)\,m_0 + t\,m_1,
        \qquad
        \Sigma_t^\circ = Q_t^\circ(Q^\circ)^{-1}\Sigma_0 (Q^\circ)^{-1}Q_t^\circ + Q_t^\circ\left(\int_0^t (P_s^\circ)^{-2}\,\rmd s\right)Q_t^\circ.
    \]
\end{example}

\section{Convergence of the Schrödinger--Bass algorithm}\label{sec:convergence_SB_algorithm}

Throughout this section, we denote by $(f_i)_{i \in \N} \subset L^1(\nu)$ the iterates of \Cref{alg:SB}. By arguments analogous to those in \Cref{subsub:SB_deconv}, these functions may be chosen $\beta$-semiconcave. Moreover, for each $i \in \N$, we set
\begin{equation}\label{def:obj4conv}
    \Tcal^\beta_1[f_i] := -\log\bigl(\exp(-q_\beta\Box(-f_i))*\gamma\bigr), \qquad u_i := q_1-\frac1\beta q_\beta\Box(-\Tcal^\beta_1[f_i]).
    \qquad
    \alpha_i:= (\nabla u_i)_\#\mu,
\end{equation}
By \Cref{rem:smooth_brenier_map}, we have $\alpha_i \in \cP_2(\Rd)$ for all $i \in \N$.

The Schrödinger--Bass system, see \Cref{fig:SB_system}, together with its uniqueness established in \Cref{thm:SB}, naturally leads to an alternating optimization scheme, namely \Cref{alg:SB}, which is studied in detail in the present section. In particular, the main result of this section, \Cref{thm:convergence_SB}, shows that $(f_i)_{i \in \N}$ converges, up to normalization and under suitable assumptions on the target measure $\nu$, to the dual optimizer of
\begin{equation}\label{eq:convergence:SB.dual}\tag{SB$\beta$d}
    V_{\rm SB}^\beta(\mu,\nu)=\max_{\substack{f\in L^1(\nu), \\ \text{ $\beta$-semiconcave}}} \left\{ \int f \, \rmd \nu - \int q_\beta \Box(-\Tcal^\beta_1[f]) \, \rmd \mu\right\},
\end{equation}
which is $\nu$-a.e.\ uniquely determined up to additive constants by \Cref{thm:SB}. To this end, we study the dual objective of the Schrödinger--Bass problem
\[
    \mathcal{D}_\beta[f] := \int f \,\rmd\nu - \int q_\beta \Box (-\Tcal^\beta_1[f]) \,\rmd\mu.
\]

The following result, \Cref{lem:strict_ascent}, shows that \Cref{alg:SB} increases the value of $\mathcal{D}_\beta$ as long as the Schrödinger--Bass system has not yet been attained.

\begin{lemma}[Strict ascent]\label{lem:strict_ascent}
    \Cref{alg:SB} strictly increases $\mathcal{D}_\beta$ at every step $i \in \N$ as long as $f_i$ does not solve the Schrödinger--Bass system, that is,
    \[
        \mathcal{D}_\beta[f_{i-1}] < \mathcal{D}_\beta[f_i]
    \]
    if and only if $\alpha_i \neq \alpha_{i+1}$.
\end{lemma}

\begin{proof}
    Let $f_{i-1} \in L^1(\nu)$ be $\beta$-semiconcave.
    By Brenier's theorem we have that
    \[
        \int q_\beta \Box (-\mathcal T^\beta_1[f_{i-1}]) \, \rmd \mu + \int \mathcal T^\beta_1 [f_{i-1}] \, \rmd \alpha_i = \frac\beta2\mathcal W_2^2(\mu,\alpha_i) < \infty,
    \]
    hence, $\mathcal T^\beta_1 [f_{i-1}] \in L^1(\alpha_i)$.
    Therefore, we can write
    \[
         \int f_{i-1} \,\rmd\nu - \int q_\beta \Box (-\Tcal^\beta_1[f_{i-1}]) \,\rmd\mu = \underbrace{\left(\int f_{i-1} \,\rmd\nu + \int \Tcal^\beta_1[f_{i-1}] \,\rmd\alpha_{i}\right)}_{\displaystyle \rm (I)} - \underbrace{\left(\int q_\beta \Box (-\Tcal^\beta_1[f_{i-1}]) \,\rmd\mu + \int \Tcal^\beta_1[f_{i-1}] \,\rmd\alpha_{i}\right)}_{\displaystyle \rm (II)}.
    \]
    Observe that by \Cref{thm:fundamental_inf_conv}
    \[
        V_{\rm EOT}\Box W_\beta(\alpha_{i},\nu) = \max_{\varphi \in L^1(\nu)} \left\{ \int \varphi \, \rmd\nu + \int \mathcal \Tcal^\beta_1[\varphi] \, \rmd\alpha_{i}\right\},
    \]
    and thus by construction of $f_{i}$, which is the maximizer to the right-hand side, satisfies
    \[
        {\rm (I)} \leq \int f_{i} \,\rmd\nu + \int \Tcal^\beta_1[f_{i}] \,\rmd\alpha_{i}.
    \]
    Likewise, $\Tcal^\beta_1[f_{i-1}]$ achieves $W_\beta(\mu,\alpha_{i})$, so that
    \[
        {\rm (II)} = W_\beta(\mu,\alpha_{i}) \geq \int q_\beta \Box (-\Tcal^\beta_1[f_{i}]) \,\rmd\mu + \int \Tcal^\beta_1[f_{i}] \,\rmd\alpha_i.
    \]
    Combining these two inequalities, we obtain
    \[
        {\rm (I) - (II)} \leq \int f_{i} \,\rmd\nu - \int q_\beta \Box (-\Tcal^\beta_1[f_{i}]) \,\rmd\mu.
    \]
    In case of equality, we have 
    \[
        W_\beta(\mu,\alpha_{i}) = {\rm (II)} = \int q_\beta \Box (-\mathcal T^\beta_1[f_i]) \, \rmd \mu + \int \mathcal T^\beta_1[f_i] \, \rmd \alpha_i,
    \]
    hence, $\mathcal T^\beta_1[f_i]$ is also a dual optimizer of $\frac\beta2 \Wcal_2^2(\mu,\alpha_i)$.
    Since $q_1 - \frac1\beta q_\beta \Box (-\mathcal T^\beta_1[f_i])$ is differentiable, we conclude that $\alpha_i = \alpha_{i + 1}$.
    In this case, $f_i$ already solves the Schrödinger--Bass system by \Cref{thm:SB}. 
\end{proof}

Since $\mathcal D_\beta[f+c] = \mathcal D_\beta[f]$ for every $c \in \R$, we must fix a normalization in order to obtain convergence of the sequence $(f_i)_{i \in \N}$. For the convergence analysis in \Cref{thm:convergence_SB}, we normalize the functions $f_i$ by imposing
\[
    \int f_i \,\rmd \nu = 0
\]
for every $i \in \N$. It is therefore convenient to suppress the dependence of the dual objective on $f$ in the notation and to consider instead the functional
\[
    \mathcal{E}_\beta[u] := \int u \,\rmd\mu,
\]
which we evaluate at $u_i = q_1 - \frac{1}{\beta} q_\beta \Box\bigl(-\Tcal^\beta_1[f_i]\bigr)$.
It follows that, for all $i \in \N$,
\[
    \mathcal{E}_\beta[u_i] = \mathcal{D}_\beta[f_i] + \int q_1 \,\rmd\mu.
\]
In particular, \Cref{lem:strict_ascent} may be restated as follows.

\begin{corollary}[Strict ascent]\label{cor:strict_ascent}
    \Cref{alg:SB} strictly increases $\Ecal_\beta$ at every step $i \in \N$ as long as $u_i$ does not solve the Schrödinger--Bass system, that is,
    \[
        \Ecal_\beta[u_{i-1}] < \Ecal_\beta[u_i]
    \]
    if and only if $\alpha_i \neq \alpha_{i+1}$.
\end{corollary}

To show convergence of \Cref{alg:SB}, we recall the notion of epi-convergence.
\begin{definition}[Epi-convergence]
    Let ${\psi_n,\psi:\Rd \to \mathbb{R}\cup\{+\infty\}}$, and let $E \subset \Rd$. We say that $\psi_n\to \psi$ in \emph{epi-convergence} on $E$ if, for every $x\in E$,
    \begin{align*}
        \psi(x)
        &\le
        \liminf_{n\to\infty} \psi_n(x_n)
        \quad\text{for every sequence }x_n\to x\text{ in }E,\\
        \psi(x)
        &\ge
        \limsup_{n\to\infty} \psi_n(x_n)
        \quad\text{for some sequence }x_n\to x\text{ in }E.
    \end{align*}
    Moreover, for ${\varphi_n,\varphi:\Rd \to \mathbb{R}\cup\{-\infty\}}$ we say that $\varphi_n\to \varphi$ in \emph{hypo-convergence} on $E$ if $-\varphi_n\to -\varphi$ in epi-convergence on $E$.
\end{definition}
We will repeatedly use the following standard consequence of convex analysis; see, for instance, \cite{Rock97,Rock98}. Let $\psi_n$ and $\psi$ be proper lower semicontinuous convex functions on $\R^d$. Let $A\subset \R^d$ be an affine subspace, and let $E\subset A$ be non-empty, convex, and open relative to $A$. Assume that $E\subset \operatorname{ri}(\operatorname{dom}(\psi_n))$ for all $n \in \N$ and $E\subset \operatorname{ri}(\operatorname{dom}(\psi))$.
Then the following are equivalent:
\begin{enumerate}
    \item $\psi_n\to \psi$ in epi-convergence on $E$;
    \item $\psi_n\to \psi$ uniformly on every compact set $K\subset E$.
\end{enumerate}
In the full-dimensional case, let $C\subset \R^d$ be bounded, open, and convex. If $C\subset \operatorname{int}(\operatorname{dom}(\psi_n))$ for all $n \in \N$ and $C\subset \operatorname{int}(\operatorname{dom}(\psi))$
then $\psi_n$ and $\psi$ are locally Lipschitz on $C$, hence differentiable almost everywhere on $C$ by Rademacher's theorem. In particular, epi-convergence on $C$ implies $\nabla \psi_n(x)\to \nabla \psi(x)$ for a.e.\ $x \in C$.

As a first step towards convergence, we show that the iteration of \cref{alg:SB} is continuous with respect to epi-convergence.

\begin{lemma}[Continuity of the iteration]\label{lem:continuity_iteration_map}
    Suppose that $\nu$ has all exponential moments, that is, $\int e^{t|y|}\,\nu(\rmd y) < \infty$ for all $t > 0$.
    Let $L > 0$, and let $(u_n)_{n\in\N}$ be $L$-smooth convex functions such that $u_n$ epi-converges on $\Rd$ to a proper convex function $u$. 
    Denote $(u_n^+)_{n \in \N}$ and $u^+$ the successors of $(u_n)_{n \in \N}$ resp.\ $u$ after one step of \Cref{alg:SB}.
    Then $u_n^{+}\to u^{+}$ in epi-convergence on $\Rd$.
\end{lemma}

\begin{proof}
    Let $\alpha := (\nabla u)_\#\mu$ and $\alpha_n := (\nabla u_n)_\#\mu$ for $n \in \N$. Then $\alpha, \alpha_n \in \cP_2(\Rd)$ for $n \in \N$ and $\alpha_n \to \alpha$ in $\cP_2(\Rd)$.
    By \Cref{thm:stability}, for every $n \in \N$, there exists a unique optimizer $\rho_n$ of $V_{\rm EOT} \Box W_\beta(\alpha_n,\nu)$, and $\rho_n\to\rho$ in $(\cP_2(\Rd),\Wcal_r)$ for all $r\in[1,2)$, where $\rho$ denotes the unique minimizer of $V_{\rm EOT}\Box W_\beta(\alpha,\nu)$.
    In particular,
    \[
        \lim_{n \to \infty }V_{\rm EOT}(\alpha_n,\rho_n) 
        =
        V_{\rm EOT}(\alpha,\rho) 
        \quad\text{and}\quad
        \lim_{n \to \infty }W_\beta(\rho_n,\nu)
        =
        W_\beta(\rho,\nu).
    \]
    For each $n \in \N$ let $v_n^*$ be the Brenier potential satisfying ${(\nabla v_n^*)_\#\rho_n = \nu}$ and ${\int \beta v_n -q_\beta\,\rmd\nu = 0}$.
    Since $\rho$ has a.e.\ positive density with respect to the Lebesgue measure on $\Rd$, $v_n^\ast \to v^\ast$ in epi-convergence on $\Rd$ where $v^*$ is the Brenier potential with ${(\nabla v^*)_\#\rho = \nu}$ and ${\int \beta v - q_\beta \,\rmd\nu = 0}$; see, e.g., \cite[Theorem 2.8]{BaLo19}. Moreover,
    \[
        u_n^+ = q_1 - \frac1\beta q_\beta\Box \Bigl(\log\bigl(\exp(\beta v_n^*-q_\beta)*\gamma\bigr)\Bigr)\quad (n\in \N), 
        \qquad
        u^+ = q_1 - \frac1\beta q_\beta\Box \Bigl(\log\bigl(\exp(\beta v^*-q_\beta)*\gamma\bigr)\Bigr).
    \]
    
    To prove $u_n^+ \to u^+$ in epi-convergence on $\Rd$, recall that
    \[
        V_{\rm EOT}(\alpha_n,\rho_n)
        = 
        \inf_{\pi \in \mathrm{Cpl}(\alpha_n,\rho_n)} \int H(\pi_x \,|\, \gamma_x)\,\rmd \alpha_n(x).
    \]
    Since $\eta \mapsto H(\eta\,|\,\gamma_x)$ is strictly convex, $V_{\rm EOT}(\alpha_n,\rho_n)$ admits a unique minimizer ${\pi^n\in \Cpl(\alpha_n,\rho_n)}$ for all $n \in \N$. As ${(\alpha_n,\rho_n) \to (\alpha,\rho)}$ weakly and ${V_{\rm EOT}(\alpha_n,\rho_n) \to V_{\rm EOT}(\alpha,\rho)}$, we have $\pi^n \to \pi$ where $\pi$ is the minimizer of $V_{\rm EOT}(\alpha,\rho)$.
    By strict convexity, we even have that ${(\operatorname{id},\pi_\cdot^n)_\# \alpha_n \to (\operatorname{id},\pi_\cdot)_\# \alpha}$ weakly.
    
    Hence there exists a probability space with random variables ${X_n \sim \alpha_n}$, ${X \sim \alpha}$ such that ${X_n \to X}$ and ${\pi_{X_n}^n \to \pi_X}$ almost surely. We claim that 
    \[
        \lim_{n \to \infty}\int \beta v_n^* -q_\beta\, \rmd\pi_{X_n}^n = \int \beta v^*-q_\beta \, \rmd \pi_{X}  \quad \text{a.s.}
    \]
    Convexity of $(v_n^*)_{n\in\N}$ yields $\beta v_n(\bar\nu) - \int q_\beta\,\rmd\nu \leq \int \beta v_n - q_\beta\,\rmd\nu = 0$, and therefore
    \[
        \sup_{n \in \N}v_n(\bar\nu) \leq \int q_1\,\rmd\nu.
    \]
    By the Fenchel--Young inequality, $v_n^*(y) \ge y\cdot \bar\nu - \int q_1\,\rmd\nu$ for all $y \in \Rd$ and all $n \in \N$. Set
    \[
        g_n(y):=v_n^*(y)-y\cdot\bar\nu+\int q_1\,\rmd\nu,
        \qquad
        g(y):=v^*(y)-y\cdot\bar\nu+\int q_1\,\rmd\nu.
    \]
    Then $g_n\ge0$ for all $n$, and $g_n\to g$ in epi-convergence on $\Rd$, hence
    \begin{align*}
        \liminf_{n \to \infty } \int g_n \, \rmd\pi_{X_n}^n 
        \ge 
        \int g \, \rmd \pi_X.
    \end{align*}
    Therefore
    \[
        \liminf_{n \to \infty} \int \beta v^*_n - q_\beta\,\rmd \pi_{X_n}^n \ge \int \beta v^*-q_\beta \, \rmd \pi_{X}  \quad \text{a.s.}
    \]
    Since $W_\beta(\rho_n,\nu) \to W_\beta(\rho,\nu)$, or equivalently,
    \[
         \lim_{n\to\infty}\int \beta v_n^* - q_\beta \,\rmd\rho_n = \int \beta v^* - q_\beta \,\rmd\rho,
    \]
    it follows that
    \[
        \lim_{n \to \infty}\int \beta v_n^* -q_\beta\, \rmd\pi_{X_n}^n = \int \beta v^*-q_\beta \, \rmd \pi_{X}  \quad \text{a.s.}
    \]
    
    Since $v_n^* \to v^*$ in epi-convergence, for every $z \in \Rd$,
    \[
        \liminf_{n \to \infty} \bigl(\beta v_n^*(X_n-z)-q_\beta(X_n-z)\bigr)
        \ge
        \beta v^*(X-z) - q_\beta(X-z) \quad \text{a.s.}
    \]
    Fatou's lemma therefore gives
    \[
        \liminf_{n \to \infty} \log\bigl(\exp(\beta v_n^*-q_\beta)\ast \gamma(X_n)\bigr) \ge \log\bigl(\exp(\beta v^*-q_\beta)\ast \gamma(X)\bigr) \quad \text{a.s.}
    \]
    On the other hand, since $V_{\rm EOT}(\alpha_n,\rho_n) \to V_{\rm EOT}(\alpha,\rho)$,
    \[
        \lim_{n \to \infty} \left(\int  \beta v_n^*-q_\beta \,\rmd\pi_{X_n}^n-\log\bigl(\exp(\beta v_n^*-q_\beta)\ast \gamma(X_n)\bigr)\right)
        =
        \int \beta v^* -q_\beta\,\rmd\pi_{X}-\log\bigl(\exp(\beta v^*-q_\beta)\ast \gamma(X)\bigr)
    \]
    almost surely. In particular,
    \[
        \lim_{n \to \infty} \log\bigl(\exp(\beta v_n^*-q_\beta)\ast \gamma(X_n)\bigr) = \log\bigl(\exp(\beta v^*-q_\beta)\ast \gamma(X)\bigr) \quad \text{a.s.}
    \]
    Let $\varepsilon\in(0,1)$. By Egorov's theorem, there exists a set $\tilde\Omega$ with $\mathbb P[\tilde\Omega]\ge 1-\varepsilon$ such that the convergence is uniform on $\tilde\Omega$. After possibly shrinking $\tilde\Omega$, we moreover assume that, on $\tilde\Omega$, ${X\in K}$ for some compact set ${K\subset\Rd}$ and
    \[
        m \le \exp(\beta v^*-q_\beta)\ast\gamma(X) \le M
    \]
    for some $0 < m < M$.
    Set
    \[
        \tilde \alpha_n := {\rm law}\bigl(X_n \,|\, \tilde \Omega\bigr),
        \quad
        \tilde \rho_n := \int \pi_x^n \, \rmd \tilde \alpha_n(x), 
        \qquad
        \tilde \alpha := {\rm law}\bigl(X \,|\, \tilde \Omega\bigr),
        \quad
        \tilde \rho:= \int\pi_x \, \rmd \tilde \alpha(x).
    \]
    Then $\tilde\alpha$ is supported on $K$, and $(\nabla v_n^*)_\# \tilde\rho_n \le \frac{\nu}{1-\varepsilon}$, $(\nabla v^*)_\# \tilde\rho \le \frac{\nu}{1-\varepsilon}$.
    Moreover,
    \[
        e^I
        :=
        \tilde\alpha\text{-}\mathrm{ess\,inf}\bigl(\exp(\beta v^*-q_\beta)\ast\gamma\bigr)
        \ge m,
        \quad \text{ and } \quad
        \liminf_{n\to\infty}
        \tilde\alpha_n\text{-}\mathrm{ess\,inf}\bigl(\exp(\beta v_n^*-q_\beta)\ast\gamma\bigr)
        = e^I.
    \]
    With $c := \frac1\beta |I|+ \sup_{n \in \N} |v_n^*(0)|$, \Cref{lem:exponential_moments} yields, for every $t \ge 0$ and every Borel set $B \subseteq \Rd$,
    \begin{equation}\label{ineq:proof:continuity_iteration}
        \limsup_{n \to \infty} \int_B e^{t|y|} \, \tilde \rho_n(\rmd y) \\
        \le 
        \frac1{1-\varepsilon} \left( \int_B e^{t |y|} \, \tilde \alpha \ast \gamma(\rmd y) + e^{2t \sqrt{c}} \int_B e^{2t|y|} \, \nu(\rmd y) \right) < \infty.
    \end{equation}
    
    Let $(x_n)_{n \in \N}$ be a bounded sequence in $\Rd$. We claim that $(f_n)_{n \in \N}$ defined by
    \[
        f_n(z) := \exp\bigl(\beta v_n^*(z) -q_{\beta}(z) - q_1(x_n) +x_n\cdot z\bigr)
    \]
    is uniformly integrable in $L^1(\gamma)$.
    To see this, set $t := \sup_{n \in \N} |x_n|$ and $s := \sup_{x \in K} |x|$. Then, for every Borel set $B \subseteq \Rd$ and every $n \in \N$,
    \begin{align*}
        \int_B |f_n(z)|\, \gamma(\rmd z) 
        &\leq 
         M e^{\frac12 s^2} \int \!\int_B (2\pi)^{-\frac{d}{2}}e^{(t+s) |z|} \frac{e^{\beta v_n^*(z)-\frac{1+\beta}{2}|z|^2 - \frac12 |x|^2 + x\cdot z}}{e^{\beta v_n^*-q_\beta}*\gamma(x)}\, \rmd z\, \tilde\alpha(\rmd x), \\
        &=M e^{\frac12 s^2} \int_B e^{(t+s) |y|} \, \tilde \rho_n(\rmd y),
    \end{align*}
    so that uniform integrability follows from \eqref{ineq:proof:continuity_iteration}. 
    Finally, let $x\in\Rd$ and choose any sequence $x_n\to x$. By uniform integrability and pointwise convergence $v_n^*\to v^*$ on $\Rd$,
    \[
        \lim_{n \to \infty} \exp(\beta v_n^* - q_\beta)*\gamma(x_n) = \exp(\beta v^* - q_\beta)*\gamma(x),
    \]
    and hence
    \[
        \lim_{n \to \infty }-\log\bigl( \exp(\beta v_n^* - q_\beta)*\gamma\bigr) = -\log\bigl( \exp(\beta v^* - q_\beta) * \gamma \bigr)
    \]
    in epi-convergence on $\Rd$.
    Since the infimal convolution is stable under epi-convergence, $u_n^+ \to u^+$ in epi-convergence on $\Rd$.
\end{proof}

We are now in a position to prove convergence of \Cref{alg:SB}.

\begin{theorem}[Convergence of the Schrödinger--Bass algorithm]\label{thm:convergence_SB}
    Let $\beta > 0$ and let $\mu, \nu \in \Pcal_2(\Rd)$ be such that $\nu$ has all exponential moments. Let $(f_i)_{i\in\N} \subset L^1(\nu)$ be the $\beta$-semiconcave functions generated by \Cref{alg:SB}, normalized so that, for every $i \in \N$,
    \[
        \int f_i \,\rmd\nu = 0.
    \]
    Then $(f_i)_{i\in\N}$ hypo-converges on $I_\nu:=\operatorname{ri}\bigl(\operatorname{co}(\operatorname{supp}(\nu))\bigr)$ to the dual optimizer of the Schrödinger--Bass problem~\eqref{eq:convergence:SB.dual}.
\end{theorem}

\begin{proof}
    Let $(u_i)_{i \in \N}$ and $(\alpha_i)_{i \in \N}$ be as in \eqref{def:obj4conv}. By \Cref{lem:tightness}, there exist a constant $c(\beta,d)>0$ and a subsequence $(u_{i_j})_{j\in\N}$ converging locally uniformly on $\Rd$ to a convex, $\frac{1+\beta}{\beta}$-smooth function $u$ such that
    \[
        \sup_{i \in \N}|u_i(x)| \le c(\beta,d)\left(1+|x|^2\right) \ \text{ for all } x\in\Rd,
        \qquad
        \alpha_{i_j} := (\nabla u_{i_j})_\#\mu \longrightarrow \alpha^\circ :=(\nabla u)_\#\mu
        \ \text{ in } \cP_2(\Rd).
    \]
    In particular, $(u_i)_{i \in \N}$ admits at least one accumulation point with respect to local uniform convergence on $\Rd$, or equivalently, with respect to epi-convergence.

    Let $u$ be an epi-accumulation point of $(u_i)_{i \in \N}$, and let $(u_{i_j})_{j\in\N}$ be a subsequence which attains $u$.
    Then dominated convergence yields
    \[
        \lim_{j \to \infty} \Ecal_\beta[u_{i_j}] = \Ecal_\beta[u].
    \]
    By \Cref{lem:strict_ascent}, the sequence $(\Ecal_\beta[u_i])_{i\in\N}$ is monotonically increasing, so that
    \[
        \lim_{i\to\infty} \Ecal_\beta[u_i] = \lim_{j\to\infty}\Ecal_\beta[u_{i_j}] = \Ecal_\beta[u].
    \]
    By \Cref{lem:continuity_iteration_map}, we have $u_{i_j+1} \to u^+$ in epi-convergence on $\Rd$, where $u^+$ denotes the next iterate of $u$ in \Cref{alg:SB}. As above, dominated convergence gives
    \[
        \lim_{j\to\infty} \Ecal_\beta[u_{i_j+1}] = \Ecal_\beta[u^+].
    \]
    On the other hand, by monotonicity of the iterations, for all $j \in \N$,
    \[
        \Ecal_\beta[u_{i_j}] \leq \Ecal_\beta[u_{i_j+1}] \leq \Ecal_\beta[u_{i_{j+1}}]
    \]
    Taking $j\to\infty$ in this chain and invoking the convergence of the three terms, we obtain $\Ecal_\beta[u] = \Ecal_\beta[u^+]$. 
    Finally, by \Cref{lem:strict_ascent}, equality $\Ecal_\beta[u]=\Ecal_\beta[u^+]$ can only occur if $u$ solves the Schrödinger--Bass system. 
    
    By \Cref{thm:SB}, the Schrödinger--Bass system is uniquely attained, and therefore all accumulation points of $(u_i)_{i \in \N}$ coincide. 
    In particular, $u_i \to u$ in epi-convergence on $\Rd$. Let $f \in L^1(\nu)$ be the $\beta$-semiconcave optimizer of the dual formulation of \eqref{eq:main.SB.dynamic}, and set $v := q_1 - \frac1\beta f$. For $i \in \N$, set $v_i := q_1 - \frac1\beta f_i$.
    
    We conclude by showing that $(v_i)_{i \in \N}$ epi-converges to $v$ on $I_\nu$. 
    By arguments analogous to those in the proof of \Cref{lem:continuity_iteration_map}, $\nabla v_i^* \to \nabla v^*$ up to Lebesgue-null sets. Since $(v_i)_{i \in \N} \subset L^1(\nu)$ and each $v_i$ is convex, we have $I_\nu \subset \operatorname{dom}(v_i)$ for all $i \in \N$. In particular, since the additive constant is fixed by the normalization $\int f_i \,\rmd\nu = 0$, we obtain by epi-convergence that $v_i \to v$ on $I_\nu$.
\end{proof}

\section{From Schrödinger to Bass and Brenier--Strassen} \label{sec:SB_limits}

We conclude by identifying the limiting regimes of the Schrödinger--Bass problem
\begin{equation}\label{eq:SB_limits.SB}\tag{SB$\beta$w}
    V_{\rm SB}^\beta(\mu,\nu) 
    =
    \inf_{\pi\in\Cpl(\mu,\nu)} \int C_{\rm SB}^\beta(x,\pi_x)\,\mu(\rmd x),
\end{equation}
as motivated in \cref{sec:introduction}. 
For $\mu, \nu \in \cP_2(\Rd)$, the Schrödinger problem takes the form
\begin{equation}\label{eq:SB_limits.EOT}\tag{SB$\infty$}
    V_{\rm EOT}(\mu,\nu)
    =
    \inf_{\pi\in\Cpl(\mu,\nu)} \int H(\pi_x\,|\,\gamma_x)\,\mu(\rmd x),
\end{equation}
and one expects that $V_{\rm SB}^\beta(\mu,\nu) \to V_{\rm EOT}(\mu,\nu)$ as ${\beta \uparrow \infty}$; see \cref{sec:introduction}.
On the other hand, after rescaling by $\beta^{-1}$ and letting $\beta\downarrow0$, the dynamic formulation of the Schrödinger--Bass problem~\eqref{eq:SB.dynamic} suggests convergence to the Bass problem
\begin{equation}\label{eq:SB_limits.mBB}\tag{mBB}
    V_{\rm mBB}(\mu,\nu)
    =
    \inf_{\pi\in\Cpl_M(\mu,\nu)} \int \frac12 \Wcal_2^2(\pi_x,\gamma_x)\,\mu(\rmd x).
\end{equation}
To show this, we first identify the unrescaled limit $\beta\downarrow0$ which leads to the Brenier--Strassen problem
\begin{equation}\label{eq:BStr}\tag{SB$0$}
    V_{\rm BStr}(\mu,\nu) := \inf_{\pi \in \Cpl(\mu,\nu)} \int \frac12 \bigl|x- \bar \pi_x \bigr|^2 \,\mu(\rmd x).
\end{equation}
This is itself a notable instance of weak optimal transport; see, e.g., \cite{gozlan2017kantorovich,gozlan2020mixture,AlCoJo20,GuNiWi25}.
We remark that, in general, the minimizers of \eqref{eq:BStr} are not unique. However, we show in the following that the Schrödinger--Bass problem selects in the limit $\beta\downarrow0$ a distinguished optimizer among them.
Moreover, the Brenier--Strassen problem admits the equivalent metric projection formulation
\begin{equation}\label{eq:SB_limits.BStr2}
    V_{\rm BStr}(\mu,\nu) = \inf_{\substack{\eta \in \Pcal_2(\Rd),\\ \mu \le_{\rm cvx} \eta}} \Wcal_2^2(\eta,\nu),
\end{equation}
see, for instance, \cite{alfonsi2020sampling,gozlan2020mixture,KiRu24}. Here, $\mu \le_{\rm cvx} \eta$ denotes convex order, that is,
\[
    \mu \le_{\rm cvx} \eta \quad\Longleftrightarrow\quad \int \psi \,\rmd\mu \le \int \psi \,\rmd\eta \quad\text{for all convex } \psi:\R^d\to\R.
\]
In particular, \eqref{eq:SB_limits.BStr2} recovers a famous result by \textcite{Strassen1965}: $\mu \leq_{\rm cvx} \nu$ if and only if $V_{\rm BStr}(\mu,\nu) = 0$. Equivalently, $\Cpl_M(\mu,\nu)$ is non-empty if and only if $\mu \leq_{\rm cvx} \nu$.

\begin{theorem}[Limits of the Schrödinger--Bass problem]\label{thm:SB.limits}
    For $\beta > 0$, let $\pi^\beta \in \Cpl(\mu,\nu)$ be the primal optimizer of the Schrödinger--Bass problem~\eqref{eq:SB_limits.SB}.
    \begin{enumerate}
        \item \label{it:SB.limits:schrödinger}
        Suppose there exists $\pi \in \Cpl(\mu,\nu)$ such that $H(\pi\,|\,\mu\otimes\gamma_{\bullet}) < \infty$.
        Then
        \[
            \lim_{\beta \uparrow \infty} V_{\rm SB}^\beta(\mu,\nu) = V_{\rm EOT}(\mu,\nu),
        \]
        and $\pi^\beta \to \pi^\infty$ weakly as $\beta \uparrow \infty$, where $\pi^\infty$ is the unique minimizer of $V_{\rm EOT}(\mu,\nu)$.

        \item \label{it:SB.limits:BStr}
        Let $\Cpl_{\rm BStr}(\mu,\nu)$ denote the set of minimizers of the Brenier--Strassen problem~\eqref{eq:BStr}. Then
        \[
            \lim_{\beta \downarrow 0} V_{\rm SB}^\beta(\mu,\nu) = V_{\rm BStr}(\mu,\nu)
        \]
        and $\pi^\beta \to \pi^0$ weakly as $\beta \downarrow 0$, where $\pi^0$ is the unique minimizer of 
        \[
            \inf_{\pi \in \cpl_{\text{\rm BStr}}(\mu,\nu)} \int \frac12 \Wcal_2^2(\pi_x,\gamma_{\bar\pi_x}) \, \mu(\rmd x).
        \]

        \item \label{it:SB.limits:mBB} Suppose that $\mu \le_{\rm cvx} \nu$. Then
        \[
            \lim_{\beta \downarrow 0} \frac1\beta V_{\rm SB}^\beta(\mu,\nu) = V_{\rm mBB}(\mu,\nu),
        \]
        and $\pi^\beta \to \pi^0$ weakly as $\beta \downarrow 0$, where $\pi^0$ is the unique minimizer of $V_{\rm mBB}(\mu,\nu)$.
    \end{enumerate}
\end{theorem}

\begin{remark}\label{rem:bass_limit}
    We call a pair $(\mu,\nu)\in\cP(\Rd)\times\cP(\Rd)$ \emph{irreducible} if, for all measurable sets $A,B\subset\Rd$ with $\mu(A)>0$ and $\nu(B)>0$, there exists $\pi\in\Cpl_M(\mu,\nu)$ such that $\pi(A\times B)>0$. 
    Note that every irreducible pair $(\mu,\nu)$ is necessarily in convex order, but the converse fails in general; see, for instance, \cite{beiglbock2016problem,ObSi17,DeTo19}.
    
    In the setting of \cref{thm:SB.limits}, if $(\mu,\nu)$ is moreover irreducible, then it follows from the theory of the Bass problem~\eqref{eq:SB_limits.mBB} that the unique optimizer $\pi^0$ is induced by a \emph{Bass martingale}; see, for instance, \cite{backhoff2023existence}. 
    More precisely, there exist ${\alpha\in\cP(\Rd)}$ and a convex, lower semicontinuous potential $\psi$ (not necessarily $\nu$-integrable) such that the process ${M=(M_t)_{t\in[0,1]}}$ defined by ${M_t := \mathbb E[\nabla \psi^*(B_1^\alpha)\,|\, B_t^\alpha]}$ satisfies ${\pi^0 = \mathrm{Law}(M_0,M_1)}$.
    Here ${B^\alpha=(B_t^\alpha)_{t\in[0,1]}}$ denotes standard Brownian motion with initial law $\alpha$. 
    Conversely, if such a potential $\psi$ exists, then $(\mu,\nu)$ is necessarily irreducible; see, for instance, \cite[Proposition 3.3]{hasenbichler2025martingalesinkhornalgorithm}.
    We refer to $\alpha$ and $\psi$ as a \emph{Bass measure} and \emph{Bass potential}, respectively. 
    In particular, \cref{thm:SB.limits} implies that, in this case, $\pi^\beta$ converges weakly to the Bass martingale coupling $\pi^0$ as ${\beta \downarrow 0}$.
\end{remark}

Before proving \cref{thm:SB.limits}, we first identify the limiting regimes at the level of the weak transport costs.

\begin{lemma}\label{lem:SB.cost.limit}    
    Let $(x,\eta) \in \R^d \times \cP_2(\R^d)$. Then 
    \[
        \beta \longmapsto C_{\rm SB}^\beta(x,\eta),
    \]
    is non-decreasing whereas 
    \[
        \beta \longmapsto \frac1\beta C_{\rm SB}^\beta(x,\eta), 
        \qquad 
        \beta \longmapsto \frac1\beta \left( C_{\rm SB}^\beta(x,\eta) - \frac12 |x - \bar\eta|^2 \right)
    \]
    are non-increasing. Moreover,
    \begin{gather*}
        \lim_{\beta \uparrow \infty} C_{\rm SB}^\beta(x,\eta) = H(\eta \,|\, \gamma_x), 
        \qquad 
        \lim_{\beta \downarrow 0} C_{\rm SB}^\beta(x,\eta) = \frac12 |x - \bar\eta|^2, \\
        \qquad \lim_{\beta \downarrow 0} \frac1\beta \Big( C_{\rm SB}^\beta(x,\eta) - \frac12 |x - \bar\eta|^2 \Big) = \frac12 \Wcal_2^2(\eta,\gamma_{\bar\eta}).
    \end{gather*}
\end{lemma}

\begin{proof}
    The monotonicity assertions follow immediately from \cref{lem:SB.cost,lem:stat.cost}, specifically \eqref{eq:SB.cost.rep2} and \eqref{lem:stat.cost.repr3}.
    
    For $\beta>0$, let $\alpha_\beta \in \cP_2(\R^d)$ with $\bar\alpha_\beta=\bar\eta$ satisfy
    \[
        C^\beta_{\rm SB}(x,\eta) 
        = 
        \inf_{\substack{\kappa \in \cP_2(\Rd), \\ \bar\kappa = \bar\eta}} \bigl\{H(\kappa \,|\, \gamma_x) + W_\beta(\kappa,\eta)\bigr\}
        =
        H(\alpha_\beta \,|\, \gamma_x) + W_\beta(\alpha_\beta,\eta).
    \]
    
    We first treat the case $\beta\uparrow\infty$. 
    If ${\sup_{\beta > 0}C^\beta_{\rm SB}(x,\eta) < \infty}$, then ${\sup_{\beta > 0} H(\alpha_\beta \,|\, \gamma_x) < \infty}$ and hence ${(\alpha_\beta)_{\beta > 0}}$ is tight.
    Moreover, since $W_\beta = \frac\beta2\Wcal_2^2$, this gives $\alpha_\beta\to\eta$ in $\cP_2(\Rd)$ and in particular weakly.
    Consequently,
    \[
        \sup_{\beta>0}\bigl\{H(\alpha_\beta \,|\, \gamma_x) + W_\beta(\alpha_\beta,\eta)\bigr\}
        \leq 
        H(\eta\,|\,\gamma_x)
        \leq 
        \liminf_{\beta\to\infty} H(\alpha_\beta \,|\, \gamma_x).
    \]
    Hence all inequalities are in fact equalities, and therefore
    \[
        \lim_{\beta\uparrow\infty} C_{\rm SB}^\beta(x,\eta)
        =
        H(\eta\,|\,\gamma_x).
    \]
    If, on the other hand, ${\sup_{\beta>0} C_{\rm SB}^\beta(x,\eta)=\infty}$, then necessarily ${H(\eta\,|\,\gamma_x)=\infty}$, so the same conclusion still holds.
    
    In the case $\beta \downarrow 0$, note that ${H(\kappa \,|\,\gamma_x) \ge H(\gamma_{\bar\eta} \,|\,\gamma_x)}$ for all $\kappa \in \cP(\Rd)$ with $\bar \kappa = \bar\eta$. Therefore
    \[
        C_{\rm SB}^\beta(x,\eta) \ge H(\gamma_{\bar\eta}\,|\,\gamma_x) = \frac12 |x - \bar\eta|^2.
    \]
    At the same time, ${C_{\rm SB}^\beta(x,\eta) \le H(\gamma_{\bar\eta}\,|\,\gamma_x) + W_\beta(\gamma_{\bar\eta},\eta)}$ so that
    \[
        \lim_{\beta \downarrow 0} C_{\rm SB}^\beta(x,\eta) = \frac12 |x - \bar\eta|^2.
    \]
    
    Finally, by \eqref{lem:stat.cost.repr3},
    \[
        C_{\rm SB}^\beta(x,\eta) - \frac12|x - \bar\eta|^2 = \inf_{\substack{\zeta \in \Pcal_2(\Rd),\\ \bar\zeta = 0}}\left\{ H(\zeta\,|\,\gamma) + \frac\beta2\left(\Wcal_2^2(\zeta,\eta) - |\bar\eta|^2\right)\right\},
    \]
    and hence, choosing $\zeta=\gamma$,
    \[
        \sup_{\beta > 0} \frac{1}{\beta}\left(C_{\rm SB}^\beta(x,\eta) - \frac12|\bar\eta-x|^2\right) 
        \le 
        \frac12 \left(\Wcal_2^2(\gamma,\eta) - |\bar\eta|^2\right) 
        = 
        \frac12 \Wcal_2^2(\eta,\gamma_{\bar\eta}).
    \]
    In particular, since 
    \[
        \frac1\beta \left(C_{\rm SB}^\beta(x,\eta) - \frac12|\bar\eta-x|^2\right) 
        = 
        \frac1\beta H(\alpha_\beta \,|\, \gamma_{\bar\eta}) + \frac12\Wcal_2^2(\alpha_\beta,\eta),
    \]
    we must have $\alpha_\beta \to \gamma_{\bar{\eta}}$ weakly as $\beta \downarrow 0$. Since $\Wcal_2^2$ is lower semicontinuous, we obtain
    \[
        \liminf_{\beta \downarrow 0} \frac1\beta \left(C_{\rm SB}^\beta(x,\eta) - \frac12 |\bar\eta -x|^2\right)\ge \liminf_{\beta \downarrow 0} \frac12 \Wcal_2^2(\alpha_\beta,\eta) \ge \frac12 \Wcal_2^2(\eta,\gamma_{\bar\eta}),
    \]
    and thus equality.
\end{proof}

With the cost asymptotics in hand, we now pass to the corresponding weak transport problems and their primal optimizers.

\begin{proof}[Proof of \cref{thm:SB.limits}]
    To prove \eqref{it:SB.limits:schrödinger}, recall from \cref{lem:SB.cost.limit} that ${C_{\rm SB}^\beta(x,\eta) \uparrow H(\eta \,|\, \gamma_x)}$ as $\beta \uparrow \infty$ for all ${x \in \Rd}$, ${\eta \in \cP_2(\Rd)}$. Then monotone convergence yields
    \begin{equation}\label{eq:proof:SB.limits.inequ}
        \lim_{\beta \uparrow \infty} V_{\rm SB}^\beta(\mu,\nu) 
        \leq  
        V_{\rm EOT}(\mu,\nu).
    \end{equation}
    Since $\pi^\beta \in \Cpl(\mu,\nu)$ for all $\beta > 0$, the family $(\pi^\beta)_{\beta > 0}$ is tight. Let therefore $(\pi^{\beta_n})_{n\in\N}$ be a subsequence with $\beta_n\uparrow\infty$ and ${\pi^{\beta_n} \to \pi^\infty}$ in $\cP_2(\Rd\times\Rd)$ for some ${\pi^\infty \in \Cpl(\mu,\nu)}$.
    Fix $\beta' > 0$. Since $\beta_n \uparrow \infty$, we have ${C_{\rm SB}^{\beta'}\le C_{\rm SB}^{\beta_n}}$ for all sufficiently large $n$. By \cref{lem:SB.cost}, $C_{\rm SB}^{\beta'}$ is a continuous standard weak transport cost. Therefore, by \cite[Proposition 2.8]{backhoff2019existence},
    \[
        \int C_{\rm SB}^{\beta'}(x,\pi^\infty_x)\,\mu(\rmd x)
        \le
        \liminf_{n\to\infty} \int C_{\rm SB}^{\beta'}(x,\pi^{\beta_n}_x)\,\mu(\rmd x)
        \le
        \lim_{n\to\infty} V_{\rm SB}^{\beta_n}(\mu,\nu)
        \le
        V_{\rm EOT}(\mu,\nu).
    \]
    By monotone convergence, this gives
    \[  
        \int H(\pi^\infty_x \,|\, \gamma_x) \, \mu(\rmd x)
        =
        \lim_{\beta' \uparrow \infty}\int C_{\rm SB}^{\beta'}(x,\pi^\infty_x)\,\mu(\rmd x)
        \le
        V_{\rm EOT}(\mu,\nu),
    \]
    and so $\pi^\infty$ is an optimizer of the Schrödinger problem~\eqref{eq:SB_limits.EOT}. In particular, we have equality in \eqref{eq:proof:SB.limits.inequ}.
    Since $V_{\rm EOT}(\mu,\nu) < \infty$, $V_{\rm EOT}(\mu,\nu)$ admits a unique optimizer, and hence $(\pi^\beta)_{\beta > 0}$ converges weakly to $\pi^\infty$ as $\beta \uparrow \infty$.

    We now turn to \eqref{it:SB.limits:BStr}. By \cref{lem:SB.cost.limit}, ${C_{\rm SB}^\beta(x,\eta) \downarrow \frac12|x-\bar\eta|^2}$ as $\beta \downarrow 0$ for all ${x \in \Rd}$, ${\eta \in \cP_2(\Rd)}$. Monotonicity of the infimum therefore yields
    \[
        \lim_{\beta\downarrow 0} V_{\rm SB}^\beta(\mu,\nu) = V_{\rm BStr}(\mu,\nu).
    \]
    Since $(\pi^\beta)_{\beta > 0}$ is tight, we choose a subsequence $(\pi^{\beta_n})_{n \in \N}$ with $\beta_n \downarrow 0$ such that $\pi^{\beta_n} \to \pi^0$ in $\cP_2(\Rd\times\Rd)$ for some ${\pi^0 \in \Cpl(\mu,\nu)}$.
    Since $(x,\eta) \mapsto \frac12 |x - \bar\eta|^2$ is a continuous standard weak transport cost, it follows from \cite[Proposition 2.8]{backhoff2019existence} that
    \[
        V_{\rm BStr}(\mu,\nu) 
        = 
        \lim_{n \to \infty} V_{\rm SB}^{\beta_n}(\mu,\nu) 
        \ge 
        \liminf_{n \to \infty} \int \frac12 |x - \bar\pi^{\beta_n}_x|^2 \, \mu(\rmd x) 
        \ge 
        \int \frac12 |x - \bar\pi^0_x|^2 \, \mu(\rmd x).
    \]
    Hence $\pi^0$ is an optimizer of the Brenier--Strassen problem~\eqref{eq:BStr}. 
    Fix $\beta > 0$. 
    If $\pi \in \Cpl_{\rm BStr}(\mu,\nu)$ is another optimizer of the Brenier--Strassen problem~\eqref{eq:BStr}, then \cite[Theorem 1.2]{gozlan2020mixture} yields $\bar\pi_x = \bar\pi^0_x$. 
    Set ${m(x) := \bar\pi^0_x}$.
    For $x \in \Rd$ and $\eta \in \cP_2(\Rd)$, define
    \begin{align*}
        \tilde C_{\rm SB}^\beta(\eta)
        :=  
        \frac1\beta\left(C_{\rm SB}^\beta(x,\eta) - \frac12|x - \bar\eta|^2\right).
    \end{align*}
    By \cref{lem:SB.cost,lem:stat.cost}, $\tilde C_{\rm SB}^\beta$ is a continuous standard weak transport cost.
    Since $\pi$ attains $V_{\rm BStr}(\mu,\nu)$, we obtain
    \begin{equation}\label{eq:proof:thm:SB.limits.mBB.inequ1}
        \int \tilde C_{\rm SB}^\beta(\pi_x) \, \mu(\rmd x)
        \ge 
        \frac1\beta \left(V_{\rm SB}^\beta(\mu,\nu) - V_{\rm BStr}(\mu,\nu)\right)
        \ge 
        \int \tilde C_{\rm SB}^\beta(\pi_x^\beta) \, \mu(\rmd x).
    \end{equation}
    Moreover, \cref{lem:SB.cost.limit} shows that $\tilde C_{\rm SB}^\beta(\eta) \uparrow \frac12\Wcal_2^2(\eta,\gamma_{\bar\eta})$ as $\beta \downarrow 0$ for all ${\eta \in \cP_2(\Rd)}$. Hence, by monotone convergence,
    \begin{align*}
        \int \frac12\Wcal_2^2(\pi_x, \gamma_{m(x)}) \, \mu(\rmd x) 
        &= 
        \lim_{n \to \infty} \int \tilde C_{\rm SB}^{\beta_n}(\pi_x) \, \mu(\rmd x)
        \ge
        \liminf_{n \to \infty} \int \tilde C_{\rm SB}^{\beta_n}(\pi_x^{\beta_n}) \, \mu(\rmd x).
    \end{align*}
    Fix $\beta' > 0$. Then, $\tilde C_{\rm SB}^{\beta'} \le \tilde C_{\rm SB}^{\beta_n}$ for all sufficiently large $n \in \N$, so that
    \[
        \liminf_{n\to\infty} \int \tilde C_{\rm SB}^{\beta_n}(\pi_x^{\beta_n})\,\mu(\rmd x)
        \ge
        \liminf_{n\to\infty} \int \tilde C_{\rm SB}^{\beta'}(\pi_x^{\beta_n})\,\mu(\rmd x)
        \ge 
        \int \tilde C_{\rm SB}^{\beta'}(\pi^0_x)\,\mu(\rmd x),
    \]
    where the last inequality follows from \cite[Proposition 2.8]{backhoff2019existence}.
    By monotone convergence, this gives
    \begin{equation}\label{eq:proof:thm:SB.limits.mBB.inequ2}
        \int \frac12\Wcal_2^2(\pi_x, \gamma_{m(x)}) \, \mu(\rmd x) 
        \ge 
        \lim_{\beta' \downarrow 0}\int \tilde C_{\rm SB}^{\beta'}(\pi^0_x)\,\mu(\rmd x)
        =
        \int \frac12\Wcal_2^2(\pi^0_x, \gamma_{m(x)}) \, \mu(\rmd x).
    \end{equation}
    Thus $\pi^0$ attains
    \[
        \inf_{\pi \in \Cpl_{\rm BStr}(\mu,\nu)} \int \frac12 \Wcal_2^2(\pi_x,\gamma_{m(x)}) \, \mu(\rmd x).
    \]
    Since, for each $x \in \R^d$, the map $\eta \mapsto \Wcal_2^2(\eta,\gamma_{m(x)})$ is strictly convex, $\pi^0$ is the unique such minimizer.
    In particular, $\pi^\beta \to \pi^0$ weakly as $\beta \downarrow 0$.

    If $\mu \leq_{\rm cvx}\nu$, then $V_{\rm BStr}(\mu,\nu) = 0$ and its minimizers are precisely $\Cpl_M(\mu,\nu)$. In addition, ${m(x) = x}$ for $\mu$-a.e.\ $x$. Therefore $\pi^0$ is, in this case, the unique minimizer of \eqref{eq:SB_limits.mBB}. Moreover, it follows from \eqref{eq:proof:thm:SB.limits.mBB.inequ1} and \eqref{eq:proof:thm:SB.limits.mBB.inequ2} that, for all $\pi \in \Cpl_M(\mu,\nu)$,
    \[
        \int \frac{1}{2} \Wcal_2^2(\pi_x,\gamma_x)\,\mu(\rmd x) \ge \lim_{\beta \downarrow 0} \frac1\beta V_{\rm SB}^\beta(\mu,\nu) \geq \int \frac12\Wcal_2^2(\pi^0_x, \gamma_x) \, \mu(\rmd x).
    \]
    Hence $\frac1\beta V_{\rm SB}^\beta(\mu,\nu) \uparrow V_{\rm mBB}(\mu,\nu)$ as $\beta \downarrow 0$. This proves \eqref{it:SB.limits:mBB}.
\end{proof}

\Cref{thm:SB.limits} shows that the optimal values and optimizers of the Schrödinger--Bass problem~\eqref{eq:SB.wot} converge to those of the Schrödinger problem~\eqref{eq:SB_limits.EOT} as $\beta\uparrow\infty$ and, if $\mu\le_{\rm cvx}\nu$ and after rescaling by $\beta^{-1}$, to those of the Bass problem~\eqref{eq:SB_limits.mBB} as $\beta\downarrow0$. In the latter case, this yields the first-order expansion
\[
    V_{\rm SB}^\beta(\mu,\nu) 
    = 
    V_{\rm BStr}(\mu,\nu) + \beta\,V_{\rm mBB}(\mu,\nu) + o(\beta).
\]
We conclude by showing that the dual potentials, rescaled by $\beta^{-1}$, converge to a Bass potential as $\beta \downarrow 0 $ whenever such a potential exists; see also \cref{rem:bass_limit}.

\begin{proposition}\label{prop:sb_tobass_dual}
    Let $\mu, \nu \in \cP_2(\Rd)$ be such that $(\mu,\nu)$ is irreducible. For $\beta > 0$, let $f_\beta$ denote the dual optimizer of the Schrödinger--Bass problem~\eqref{eq:SB.dual}. Then, as $\beta \downarrow 0$, $\bigl(q_1-\frac{1}{\beta}f_\beta\bigr)_{\beta > 0}$ epi-converges, up to affine normalization, to a Bass potential on $\operatorname{ri}\bigl(\operatorname{co}(\operatorname{supp}(\nu))\bigr)$.
\end{proposition}

\begin{proof}
    Set $\bar{C}^\beta_{\rm SB} := \frac{1}{\beta} C^\beta_{\rm SB}.$ 
    By \Cref{lem:SB.cost.limit}, we have ${\bar{C}^\beta_{\rm SB} \uparrow C_{\rm mBB}}$ as $\beta \downarrow 0$, where for $\eta \in \Pcal_2(\R^d)$ and $x \in \R^d$,
    \[
        C_{\rm mBB}(x,\eta) = 
        \begin{cases}
            \displaystyle \frac{1}{2} \Wcal_2^2(\eta,\gamma_x) & \bar{\eta}=x,\\
            +\infty & \text{otherwise}.
        \end{cases}
    \]
    Fix $\beta>0$. Let $\varphi_\beta:=\frac{1}{\beta}f_\beta$ a dual optimizer for $\frac1\beta V^\beta_{\rm SB}(\mu, \nu)$, and denote the corresponding $C$-transform by
    \[
        \varphi_\beta^{\bar{C}^\beta_{\rm SB}}(x) 
        := 
        \inf_{\rho \in \Pcal_2(\R^d)} \left\{ \bar{C}^\beta_{\rm SB}(x,\rho) - \rho(\varphi_\beta) \right\}.
    \]
    By optimality,
    \[
        \frac1\beta V^\beta_{\rm SB} 
        = 
        \int \varphi_\beta^{\bar{C}^\beta_{\rm SB}}\,\rmd\mu + \int \varphi_\beta\,\rmd\nu.
    \]
    Let $\pi^0 \in \Cpl_M(\mu,\nu)$ be a Bass--martingale coupling; see \cref{rem:bass_limit}. 
    Since ${\bar{C}^\beta_{\rm SB} \leq C_{\rm mBB}}$, we have for $\mu$-a.e.\ $x$,
    \[
        \varphi_\beta^{\bar{C}^\beta_{\rm SB}}(x) \leq \varphi_\beta^{C_{\rm mBB}}(x),
    \]
    where
    \[
        \varphi_\beta^{C_{\rm mBB}}(x) = \frac{d}{2}-q_1+\big((q_1-\varphi_\beta)^**\gamma\big)^*.
    \]
    Disintegrating $\pi^0(\rmd x,\rmd y)=\mu(\rmd x)\pi^0_x(\rmd y)$, we obtain
    \begin{equation} \label{eq:SB_to_B_maximizing_sequence}
        \frac1\beta V^\beta_{\rm SB} 
        = \int \!\left(\int\varphi_\beta\,\rmd\pi^0_x + \varphi_\beta^{\bar{C}^\beta_{\rm SB}}(x)\right)\mu(\rmd x) 
        \le 
        \int \!\left(\int\varphi_\beta\,\rmd\pi^0_x + \varphi_\beta^{C_{\rm mBB}}(x)\right)\mu(\rmd x)
        \le 
        V_{\rm mBB}(\mu,\nu).
    \end{equation}
    The last inequality follows from the dual formulation of the Bass problem~\eqref{eq:bass_problem.dual}.
    
    For $\beta > 0$ set $\hat\varphi_\beta :=\varphi_\beta+\ell_\beta$, where $\ell_\beta$ is an affine function such that $\hat\varphi_\beta\leq 0$ with equality at $\bar{\mu}$; see \cite[Lemma 3.2]{hasenbichler2025martingalesinkhornalgorithm}. Recall also that for $\ell(x)=a\cdot x+b$ one has ${(\varphi+\ell)^{C_{\rm mBB}}=\varphi^{C_{\rm mBB}}-\ell}$, and hence, for $\mu$-a.e.\ $x$,
    \[
        \int\hat\varphi_\beta\,\rmd\pi^0_x + \hat\varphi_\beta^{C_{\rm mBB}}(x) = \int\varphi_\beta\,\rmd\pi^0_x + \varphi_\beta^{C_{\rm mBB}}(x).
    \]
    Let $(\beta_n)_{n \in \N}$ be an arbitrary null-sequence. By \Cref{thm:SB.limits}, 
    \[
        \lim_{n \to \infty }\frac1{\beta_n} V^{\beta_n}_{\rm SB}(\mu, \nu) 
        =
        V_{\rm mBB}(\mu,\nu). 
    \]
    In view of \eqref{eq:SB_to_B_maximizing_sequence}, $(q_1-\hat\varphi_{\beta_n})_{n \in \N}$ is a maximizing sequence for \eqref{eq:bass_problem.dual}. By \cite[Proposition 3.12]{hasenbichler2025martingalesinkhornalgorithm}, the sequence ${(q_1-\hat\varphi_{\beta_n})_{n \in \N}}$ epi-converges to a Bass potential on ${\operatorname{ri}\bigl(\operatorname{co}(\operatorname{supp}(\nu))\bigr)}$.
\end{proof}

\appendix

\section{Auxiliary results and postponed proofs} \label{sec:aux_results}

\begin{lemma} \label{lem:convexity_WOT}
    Let ${W:\Pcal_p(\mathcal{X})\times\Pcal_p(\mathcal{Y})\to\R\cup\{+\infty\}}$ be a standard weak transport problem. Then $W$ is jointly convex.
\end{lemma}

\begin{proof}
    For $\mu\in \Pcal_p(\mathcal{X})$ and $\nu \in \Pcal_p(\mathcal{Y})$, we set
    \[
        \Lambda(\mu,\nu) := \left\{P \in \Pcal(\mathcal{X} \times \Pcal(\mathcal{Y}))\colon {\rm pr}^\mathcal{X}_\# P=\mu, \, I({\rm pr}^{\Pcal (\mathcal{Y})}_\# P) = \nu\right\}.
    \]
    In \cite[Lemma~2.1]{backhoff2019existence} it is shown that
    \[
        W(\mu,\nu) = \inf_{P \in \Lambda(\mu,\nu)} \int_{\mathcal{X} \times \Pcal(\mathcal{Y})} C_W(x,\rho) \, P(\rmd x, \rmd \rho).
    \]
    Let $\lambda \in [0,1]$, choose $\mu_1,\mu_2\,\in \Pcal_p(\mathcal{X})$ and $\nu_1,\nu_2 \in \Pcal_p(\mathcal{Y})$, and set $\mu:=\lambda\mu_1+(1-\lambda)\mu_2$, $\nu:=\lambda\nu_1+(1-\lambda)\nu_2$. Then, for any $P_1 \in \Lambda(\mu_1,\nu_1)$, $P_2 \in \Lambda(\mu_2,\nu_2)$, we have $\lambda P_1+(1-\lambda)P_2 \in \Lambda(\mu,\nu)$. Hence,
    \[
        \inf_{P \in \Lambda(\mu,\nu)} \int_{\mathcal{X} \times \Pcal(\mathcal{Y})} C_W(x,\rho) \, P(\rmd x, \rmd \rho) \leq \lambda\int_{\mathcal{X} \times \Pcal(\mathcal{Y})} C_W(x,\rho) \, P_1(\rmd x, \rmd \rho) + (1-\lambda)\int_{\mathcal{X} \times \Pcal(\mathcal{Y})} C_W(x,\rho) \, P_2(\rmd x, \rmd \rho).
    \]  
    Since this inequality holds for all $P_1 \in \Lambda(\mu_1,\nu_1)$ and $P_2 \in \Lambda(\mu_2,\nu_2)$, taking the infimum over $P_1$ and $P_2$ yields the claim.
\end{proof}

\begin{lemma}\label{lem:appendix:coercivity_gives_lsc} Let ${V:\Pcal_p(\mathcal{X})\times\Pcal_p(\mathcal{Y})\to\R\cup\{+\infty\}}$ and ${W:\Pcal_p(\mathcal{Y})\times\Pcal_p(\mathcal{Z})\to\R\cup\{+\infty\}}$ be standard weak transport problems
    such that $(C_V,W)$ satisfies the coercivity assumption~\Cref{aspt:coercivity} for some $r \in [1,p]$. Then
    \[
        (x,\rho,\eta) \longmapsto C_V(x,\rho) + W(\rho,\eta)
    \]
    is lower semicontinuous on $\mathcal X \times (\mathcal{P}_p(\mathcal Y), \Wcal_r) \times \cP_p(\mathcal Z)$. Moreover,
    \[
        (\mu,\rho,\nu) \longmapsto V(\mu,\rho) + W(\rho,\nu)
    \]
    is lower semicontinuous on $(\cP_p(\mathcal X), \tau_w) \times (\mathcal{P}_p(\mathcal Y), \Wcal_r) \times \cP_p(\mathcal Z)$.
\end{lemma}
\begin{proof}
    Let $(x_k)_k \subset \mathcal X$, $(\rho_k)_k \in \cP_p(\mathcal Y)$ and $(\eta_k)_k \subset \cP_p(\mathcal Z)$ be such that $x_k \to x^\circ$, $\rho_k \to \rho^\circ$ in $(\cP_p(\mathcal Y),\Wcal_r)$ and $\eta_k \to \eta^\circ$ in $\cP_p(\mathcal Z)$. After passing to a subsequence, we assume that 
    \[
        L = \lim_{k \to \infty} C_V(x_k,\rho_k) + W(\rho_k,\eta_k) < \infty,
    \]
    as otherwise there is nothing to show. 
    For $n \in \N$, set $K_n := \{(x^\circ,\eta^\circ)\} \cup \{(x_k,\eta_k): k \geq n\}$.
    By~\Cref{aspt:coercivity}, the map 
    \[
        h_n(\rho) := \inf_{(x,\eta) \in K_n} \bigl\{C_V(x,\rho) + W(\rho,\eta)\bigr\}
    \]
    has compact, hence closed, sublevel sets in $(\cP_p(\mathcal Y),\Wcal_r)$. In particular, $h_n$ is lower semicontinuous, and so
    \[
        h_n(\rho^\circ) \le \liminf_{k \to \infty} h_n(\rho_k) \le L.
    \]
    Moreover,
    \[
        h_n(\rho^\circ) 
        = 
        \min \Bigl\{ C_V(x^\circ,\rho^\circ) + W(\rho^\circ,\eta^\circ), \, \inf_{k \geq n} \bigl\{C_V(x_k,\rho^\circ) + W(\rho^\circ,\eta_k)\bigr\}\Bigr\}.
    \]
    Since $(x,\eta) \mapsto C_V(x,\rho^\circ)+W(\rho^\circ,\eta)$ is lower semicontinuous on $\mathcal{X}\times\cP_p(\mathcal Z)$, we conclude
    \[
        C_V(x^\circ,\rho^\circ) + W(\rho^\circ,\eta^\circ) = \lim_{n \to \infty} h_n(\rho^\circ) \leq L.
    \]
    
    We now turn to the second claim. Fix $\mu \in \cP_p(\mathcal X)$, $\rho \in \cP_p(\mathcal Y)$ and $\nu \in \cP_p(\mathcal{Z})$. We first show that
    \begin{equation}\label{eq:proof:coercivity_gives_lsc:lifted_V+W}
        V(\mu,\rho) + W(\rho,\nu) = \inf_{P \in \Lambda(\mu,\rho,\nu)} \int C_V(x,p) + W(p,q) \,P(\rmd x, \rmd p, \rmd q),
    \end{equation}
    where
    \[
        \Lambda(\mu,\rho,\nu) = \left\{P \in \cP(\mathcal{X} \times \cP_p(\mathcal Y)\times \cP_p(\mathcal Z)): \ \operatorname{pr}^\mathcal{X}_\#P = \mu, \ I\!\left(\operatorname{pr}^{\cP_p(\mathcal Y)}_\#P\right) = \rho, \ I\!\left(\operatorname{pr}^{\cP_p(\mathcal Z)}_\#P\right) = \nu\right\}.
    \]
    Let $P \in \Lambda(\mu,\rho,\nu)$, denote $(P_x)_x$ its disintegration with respect to $\mu$, and set 
    \[
        \bar p_x := \int p \, P_x(\rmd p, \rmd q), \qquad \bar q_x := \int q \, P_x(\rmd p, \rmd q).
    \]
    Since $\rho \mapsto C_V(x,\rho)$ is convex for each $x \in \mathcal X$ and $W$ is jointly convex by \cref{lem:convexity_WOT} as well as lower semicontinuous on $\cP_p(\mathcal Y)\times \cP_p(\mathcal Z)$ by \cite[Theorem 2.9]{backhoff2019existence}, Jensen's inequality gives
    \begin{align*}
        \int C_V(x,p) + W(p,q)\, P(\rmd x, \rmd p,\rmd q) 
        &= 
        \iint C_V(x,p) + W(p,q)\, P_x(\rmd p,\rmd q) \,\mu(\rmd x)
        \\
        &\ge 
        \int C_V(x,\bar p_x) + W(\bar p_x, \bar q_x) \, \mu(\rmd x)
        \\
        &\ge 
        \int C_V(x,\bar p_x) \, \mu(\rmd x) + W(\rho, \nu)
        \ge
        V(\mu,\rho) + W(\rho,\nu).
    \end{align*}
    The last inequality holds because $\mu \otimes \bar p_\bullet \in \Cpl(\mu,\rho)$. 
    Conversely, let $\pi \in \Cpl(\mu, \rho)$ and $\pi' \in \Cpl(\rho,\nu)$ be optimal for $V(\mu,\rho)$ resp.\ $W(\rho,\nu)$. Set $\hat\pi_x := \int_\mathcal{Y} \pi_y' \pi_x(\rmd y)$ and define
    \[
        P^\circ := \bigl(x \mapsto (x, \pi_x, \hat\pi_x)\bigr)_\# \mu.
    \]
    Then $P^\circ \in \Lambda(\mu,\rho,\nu)$. Moreover, Jensen's inequality gives
    \begin{align*}
        \int W(\pi_x,\hat\pi_x)\,\mu(\rmd x) 
        \ge 
        W(\rho,\nu) 
        = 
        \int C_W(y,\pi'_y) \, \rho(\rmd y) 
        = \iint C_W(y, \pi'_y) \, \pi_x(\rmd y) \, \mu(\rmd x) \ge \int W(\pi_x,\hat\pi_x)\,\mu(\rmd x),
    \end{align*}    
    hence $W(\rho,\nu) = \int W(\pi_x,\hat\pi_x)\,\mu(\rmd x)$. Consequently,
    \begin{align*}
        \int C_V(x,p) + W(p,q)\, P^\circ(\rmd x, \rmd p,\rmd q) 
        &= 
        \int C_V(x,\pi_x)\,\mu(\rmd x) + \int W(\pi_x,\hat\pi_x)\,\mu(\rmd x) = V(\mu, \rho) + W(\rho, \nu).
    \end{align*}
    This proves \eqref{eq:proof:coercivity_gives_lsc:lifted_V+W}. 
    
    Finally, let $(\mu_k)_k \subset \cP_p(\mathcal X)$, $(\rho_k)_k \subset \cP_p(\mathcal Y)$ and $(\nu_k)_k \subset \cP_p(\mathcal Z)$ such that 
    \[
        \mu_k \to \mu \in \cP_p(\mathcal X) \ \text{ weakly}, 
        \qquad 
        \rho_k \to \rho \in \cP_p(\mathcal Y) \ \text{ in } \Wcal_r, 
        \qquad 
        \nu_k \to \nu \in \cP_p(\mathcal Z) \ \text{ in } \Wcal_p. 
    \]
    For each $k$, choose $P^k \in \Lambda(\mu_k,\rho_k,\nu_k)$ attaining the infimum in \eqref{eq:proof:coercivity_gives_lsc:lifted_V+W}. 
    The first marginals $\bigl(\operatorname{pr}^\mathcal{X}_\# P^k\bigr)_k$ are tight. 
    Since $(\rho_k)_k$ and $(\nu_k)_k$ are precompact in $(\cP_p(\mathcal Y), \Wcal_r)$ resp.\ $\cP_p(\mathcal Z)$, we have that $\bigl(\operatorname{pr}^{\cP_p(\mathcal Y)}_\# P^k\bigr)_k$ and $\bigl(\operatorname{pr}^{\cP_p(\mathcal Z)}_\# P^k\bigr)_k$ are precompact in $\cP((\cP_p,\Wcal_r))$ resp.\ $\cP(\cP_p(\mathcal Z))$; see \cite[Lemma 2.4]{backhoff2019existence}. 
    Because all marginals of $(P^k)_k$ are tight, $(P^k)_k$ is tight. Passing to a subsequence, there exists $P \in \Lambda(\mu,\rho,\nu)$ such that $P^k \to P$ weakly.
    Since the map $(x,p,q)\mapsto C_V(x,p)+W(p,q)$ is lower semicontinuous on $\mathcal X \times (\cP_p(\mathcal Y),\Wcal_r) \times \cP_p(\mathcal Z)$ and bounded from below, \cite[Theorem 2.9]{backhoff2019existence} yields
    \[
        \liminf_{k \to \infty} \int C_V(x,p) + W(p,q)\, P^k(\rmd x, \rmd p, \rmd q) 
        \ge
        \int C_V(x,p) + W(p,q)\, P(\rmd x, \rmd p, \rmd q) 
        \ge 
        V(\mu,\rho) + W(\rho,\nu).
    \]
    Thus $(\mu,\rho,\nu) \mapsto V(\mu,\rho) + W(\rho,\nu)$ is lower semicontinuous on $(\cP_p(\mathcal X),\tau_w) \times (\cP_p(\mathcal Y),\Wcal_r) \times \cP_p(\mathcal Z)$.
\end{proof}

\begin{lemma} \label{lem:concavity_regularisation}
    Let $\beta, \sigma^2> 0$, and let $f: \Rd \to \R$ be $\beta$-semiconcave. Then $-\log\bigl(\exp(-f)\ast\gamma_{0;\sigma^2}\bigr)$ is $\frac{\beta}{1+\beta\sigma^2}$-semiconcave and upper semicontinuous (but not necessarily proper).
\end{lemma}

\begin{proof}
    Let $g := q_\beta - f$, which is convex, and set $\phi(y):= \exp\bigl(-\tfrac{1+\beta\sigma^2}{2\sigma^2}|y|^2\bigr)$.
    We have
    \begin{align*}
        \left(2\pi\sigma^2\right)^\frac{d}{2} \bigl(\exp(-f)\ast\gamma_{0;\sigma^2}\bigr)(x) &= \int \exp\!\left(g(y)-\tfrac{\beta}{2}|y|^2 - \tfrac1{2\sigma^2} |x-y|^2 \right) \, \rmd y
        \\
        &= \exp\!\left( -\tfrac{\beta}{2(1+\beta\sigma^2)}|x|^2\right) \int \exp\!\left(g(y) - \tfrac{1 + \beta\sigma^2}{2\sigma^2}\left| y -\tfrac{1}{1+\beta\sigma^2}x \right|^2 \right) \, \rmd y
        \\
        &= \exp\!\left( -\tfrac{\beta}{2(1+\beta\sigma^2)}|x|^2\right) \int \exp\bigl(g\bigl( y + \tfrac{1}{1+\beta\sigma^2}x \bigr)\bigr) \phi(y) \, \rmd y.
    \end{align*}
    Now, fix $x_0,x_1 \in \R^d$ and write $z_t := \frac{(1-t)x_0 + tx_1}{1 + \beta\sigma^2}$.
    Therefore, using convexity of $g$ and Hölder's inequality,
    \begin{align*}
        \int \exp\big(g(y + z_t)\big) \phi(y) \, \rmd y &\le
        \int \exp\big( (1-t) g( y + z_0)  + t g( y + z_1) \big) \phi(y) \, \rmd y 
        \\
        &\le \left(\int \exp\big(g(y + z_0)) \phi(y) \, \rmd y\right)^{1-t} \left( \int \exp\big( g(y + z_1) \big) \phi(y) \, \rmd y \right)^t,
    \end{align*}
    and so $q_{\frac{\beta}{1+\beta\sigma^2}} \! + \log\bigl( \exp(-f) \ast\gamma_{0;\sigma^2} \bigr)$ is convex and lower semicontinuous by Fatou, but not necessarily proper. Equivalently, $-\log\bigl(\exp(-f)\ast\gamma_{0;\sigma^2}\bigr)$ is $\frac{\beta}{1+\beta\sigma^2}$-semiconcave and upper semicontinuous.
\end{proof}

\begin{corollary} \label{cor:smooth_brenier_map}
    Let $\beta,\sigma^2 > 0$, and let $f: \Rd \to \R$ be $\beta$-semiconcave. Set 
    \[
        \Tcal^\beta_{\sigma^2}[f] := - \log\bigl(\exp(-q_\beta\Box (-f))*\gamma_{0;\sigma^2}\bigr).
    \]
    If $\Tcal^\beta_{\sigma^2}[f]$ is proper, the map $q_1 - \frac{1}{\beta}q_\beta \Box(-\Tcal^\beta_{\sigma^2}[f])$ is proper convex and $\frac{1+\beta\sigma^2}{\beta\sigma^2}$-smooth.
\end{corollary}

\begin{proof}
    By \Cref{lem:concavity_regularisation}, $\Tcal^\beta_{\sigma^2}[f]$ is $\frac{\beta}{1+\beta\sigma^2}$-semiconcave. 
    Thus, $g := q_{\frac{\beta}{1+\beta\sigma^2}}\!-\Tcal^\beta_{\sigma^2}[f]$ is proper convex.
    We have
    \begin{align*}
        q_\beta \Box(-\Tcal^\beta_{\sigma^2}[f])(x) 
        &= 
        \inf_{y \in \R^d} \left\{ \tfrac{\beta}{2}|x-y|^2 - \Tcal^\beta_{\sigma^2}[f](y) \right\}
        = \tfrac{\beta}{2}|x|^2 - \sup_{y \in \R^d} \left\{\beta \, x \cdot y - \bigl( \tfrac{\beta}{2}|y|^2 + g(y) - \tfrac{\beta}{2(1+\beta\sigma^2)}|y|^2  \bigr)\right\}\\
        &= \tfrac{\beta}{2}|x|^2 - \sup_{y \in \R^d} \left\{ \beta \, x \cdot y - \left( \tfrac{\beta^2\sigma^2}{2(1+\beta\sigma^2)} |y|^2 + g(y) \right)\right\}.
    \end{align*}
    Hence, $q_1 - \frac{1}{\beta}q_\beta \Box(-\Tcal^\beta_{\sigma^2}[f])$ is $\frac{1 + \beta\sigma^2}{\beta\sigma^2}$-smooth as the convex conjugate of a $\frac{\beta\sigma^2}{1+\beta\sigma^2}$-strongly convex function.
    In particular, this means that the induced Brenier map is $\frac{1 + \beta\sigma^2}{\beta\sigma^2}$-Lipschitz.
\end{proof}

\begin{remark} \label{rem:smooth_brenier_map}
    Let $\beta > 0$, and let $f: \Rd \to \R$ be $\beta$-semiconcave such that $\Tcal^\beta_1[f]$ is proper.
    Let $\mu \in \cP_2(\Rd)$ and set $\alpha := \bigl(\operatorname{id} - \frac{1}{\beta}\nabla q_\beta \Box(-\Tcal^\beta_1[f])\bigr)_\#\mu$. By \Cref{{cor:smooth_brenier_map}}, $v := q_1 - \frac{1}{\beta}q_\beta \Box(-\Tcal^\beta_1[f])$ is a Brenier potential satisfying $(\nabla v)_\# \mu = \alpha$. Moreover,
    \begin{align*}
        \int |z|^2 \, \alpha(\rmd z) = \int |\nabla v|^2 \, \rmd\mu &\leq \int \left(|\nabla v(\bar{\mu})| + \tfrac{(1+\beta)}{\beta}|x-\bar{\mu}|\right)^2 \, \mu(\rmd x)\\
        &\leq 2 |\nabla v(\bar{\mu})|^2 + 2\tfrac{(1+\beta)^2}{\beta^2} \int |x-\bar{\mu}|^2 \, \mu(\rmd x) < \infty,
    \end{align*}
    and so $\alpha \in \cP_2(\Rd)$.
\end{remark}

\begin{lemma}\label{lem:stat.cost}
    The function $C_{\rm stat}^\beta\colon \R^d \times \cP_2(\R^d) \to \R$, defined by
    \begin{align}
        \label{eq:SB.cost.rep}
        C_{\rm stat}^\beta(x,\eta) &:= \sup_{y \in \R^d} \left\{-q_\beta(x-y) + \inf_{\rho  \in \cP_2(\R^d)}\left\{ H(\rho \,|\, \gamma_y) + W_\beta(\rho,\eta)\right\}\right\} \\
        \label{eq:SB.cost.rep2}
        &= \inf_{\substack{\kappa \in \cP_2(\R^d),\\ \bar \kappa = \bar \eta}} \bigl\{H(\kappa\,|\,\gamma_x) + W_\beta(\kappa,\eta)\bigr\}
    \end{align}
    is a continuous standard weak transport cost and admits a $2$-growth bound. 
    Moreover, $\eta \mapsto C_{\rm stat}^\beta(x,\eta)$ is strictly convex for every $x \in\Rd$.
\end{lemma}

\begin{remark}[Characterization of optimizers]\label{rem:stat.cost.optimizers}
    Denote the unique optimizers of \eqref{eq:SB.cost.rep} by $(y^\circ,\rho^\circ) \in \R^d \times \cP_2(\R^d)$ and the unique optimizer of \eqref{eq:SB.cost.rep2} by ${\kappa^\circ \in \cP_2(\R^d)}$.
    They are related as follows:
    \begin{align*}
        \kappa^\circ = (\operatorname{id} + \bar \eta - \bar \rho^\circ)_\# \rho^\circ, \quad
        \bar \rho^\circ= \frac{y^\circ + \beta \bar \eta}{1 + \beta},\quad
        y^\circ = x + \frac{x - \bar\eta}{\beta},
    \end{align*}
    see  \eqref{eq:lem.stat.cost.mstar} and \eqref{eq:lem.stat.cost.ystar}.
    In particular, we have by combining the last two equalities
    \[
        \beta(y^\circ - x) = x - \bar \eta = x - \frac{(1 + \beta) \bar \rho^\circ - y^\circ}{\beta} = \frac{1 + \beta}{\beta}(x - \bar \rho^\circ) + \frac1\beta(y^\circ - x),
    \]
    from where it follows that $ (\beta - 1) (y^\circ - x) = x - \bar \rho^\circ$ and consequently
    \begin{equation}
        \label{eq:rem.optimizers}
        \beta(y^\circ - x) = y^\circ - \bar \rho^\circ.
    \end{equation} 
\end{remark}

\begin{proof}[Proof of \Cref{lem:stat.cost}]
    Let $x\in \R^d$ and $\rho \in \cP_2(\R^d)$.
    First, we pertain to the alternative representation \eqref{eq:SB.cost.rep2}. Let $\tau^m(x) := x - m$. Then,
    \begin{align*}
        H(\rho \,|\, \gamma_y) &= H(\tau^{m}_\#\rho \,|\,\gamma) + \frac12 |m - y|^2, \qquad \Wcal_2^2(\rho,\eta) = \Wcal_2^2(\tau^m_\#\rho,\eta) - 2 \, m \cdot \bar\eta +|m|^2,
    \end{align*}
    when $m = \bar \rho$.
    It follows that
    \begin{multline*}
        \inf_{\rho \in \cP_2(\R^d)} \left\{H(\rho \,|\, \gamma_y) + \frac\beta2 \Wcal_2^2(\rho,\eta)\right\} \\
        = \inf_{\substack{\zeta \in \cP_2(\R^d),\\ \bar\zeta = 0}} \left\{H(\zeta\,|\,\gamma) + \frac\beta2 \Wcal_2^2(\zeta,\eta) + \inf_{m \in \R^d} \left\{\frac1{2} |m - y|^2 - \beta \,m\cdot\bar\eta + \frac\beta2 |m|^2\right\}\right\}, \\
        = \frac{\beta|y - \bar\eta|^2}{2(1 + \beta)} - \frac\beta2 |\bar\eta|^2+\inf_{\substack{\zeta \in \cP_2(\R^d),\\ \bar\zeta = 0}} \left\{H(\zeta\,|\,\gamma) + \frac\beta2 \Wcal_2^2(\zeta,\eta)\right\},
    \end{multline*}
    where the last equality follows from
    \begin{equation}
        \label{eq:lem.stat.cost.mstar}
        \inf_{m \in \R^d} \left\{\frac12 |m - y|^2 - \beta \, m \cdot \bar\eta + \frac\beta2 |m|^2\right\} = \frac{\beta|y - \bar\eta|^2}{2(1 + \beta)} - \frac\beta2 |\bar\eta|^2,
    \end{equation}
    which is uniquely attained at $m^\circ = \frac{y+\beta \bar \eta}{1 + \beta}$.
    Furthermore, we have
    \begin{equation}
        \label{eq:lem.stat.cost.ystar}
        \sup_{y \in \R^d} \left\{-\frac\beta2|x-y|^2 + \frac{\beta|y - \bar\eta|^2}{2(1 + \beta)} \right\} = \frac12 |x - \bar\eta|^2,
    \end{equation}
    that is uniquely achieved at $y^\circ = x + \frac{x - \bar\eta}{\beta}$.
    This allows us to separate $\inf$ and $\sup$ in \eqref{eq:SB.cost.rep} and we get
    \begin{align}
        \tag*{}
        C_{\rm stat}^\beta(x,\eta) &= \sup_{y \in \R^d}\left\{-q_\beta(x-y) + \frac{\beta|y-\bar\eta|^2}{2(1+\beta)}\right\} - \frac\beta2 |\bar\eta|^2 + \inf_{\substack{\zeta \in \cP_2(\R^d),\\ \bar\zeta = 0}} \left\{H(\zeta|\gamma) + W_\beta(\zeta,\eta)\right\}
        \\
        \label{lem:stat.cost.repr3}
        &= \frac12 |x-\bar\eta|^2  + \inf_{\substack{\zeta \in \cP_2(\R^d),\\ \bar\zeta = 0}} \left\{H(\zeta\,|\,\gamma) + \frac{\beta}{2}\Big( \Wcal_2^2(\zeta,\eta) - |\bar\eta|^2\Big)\right\}
        \\
        \tag*{}
        &= \inf_{\substack{\zeta \in \cP_2(\R^d),\\ \bar\zeta = 0}} \left\{\Big( H(\tau^{-\bar\eta}_\# \zeta\,|\,\gamma_{\bar\eta}) + \frac12 |x-\bar\eta|^2 \Big) + \frac\beta2 \Wcal_2^2(\tau^{-\bar\eta}_\# \zeta,\eta)\right\}
        \\
        \label{lem:stat.cost.repr4}
        &= \inf_{\substack{\kappa \in \cP_2(\R^d),\\ \bar{\kappa} = \bar{\eta}}} \left\{H(\kappa\,|\,\gamma_x) + \frac\beta2 \Wcal_2^2(\kappa,\eta)\right\}.
    \end{align}
    In particular, $C_{\rm stat}^\beta \geq 0$. By the same arguments as in the proof of \cref{lem:SB.cost.assumptions}, \eqref{lem:stat.cost.repr4} is attained. Moreover, testing the infimum in \eqref{lem:stat.cost.repr3} with $\zeta = \gamma$ yields, for all $(x,\eta)\in \Rd\times \cP_2(\Rd)$,
    \[
        C_{\rm stat}^\beta(x,\eta) 
        \leq 
        \frac12 |x-\bar\eta|^2 + \frac{\beta}{2}\Big( \Wcal_2^2(\gamma,\eta) - |\bar\eta|^2\Big),
    \]
    hence $C_{\rm stat}^\beta$ satisfies a $2$-growth bound.
    
    Fix $x \in \Rd$. To see that $\eta \mapsto C_{\rm stat}^\beta(x,\eta)$ is strictly convex, let $\eta_0,\eta_1\in\cP_2(\Rd)$ with $\eta_0\neq\eta_1$ and $t \in (0,1)$. Denote by $\kappa_0,\kappa_1\in\cP_2(\Rd)$ minimizers in \eqref{lem:stat.cost.repr4} for $(x,\eta_0)$ resp.\ $(x,\eta_1)$.
    Set 
    \[
        \eta_t := t \eta_0+(1-t)\eta_1
        \qquad
        \kappa_t := t \kappa_0+(1-t)\kappa_1.
    \]
    Then $\bar \kappa_t = \bar \eta_t$. 
    Since $\rho \mapsto H(\rho \,|\, \gamma_x)$ is strictly convex and $\Wcal_2^2$ is jointly convex (see \cref{lem:convexity_WOT}), we obtain
    \begin{align*}
        C_{\rm stat}^\beta(x,\eta_t) 
        &\le
        H(\kappa_t\,|\,\gamma_x) + \frac\beta2 \Wcal_2^2(\kappa_t,\eta_t) 
        \le
        t \,C_{\rm stat}^\beta(x,\eta_0) + (1-t) \,C_{\rm stat}^\beta(x,\eta_1).
    \end{align*}
    If $\kappa_0\neq\kappa_1$, then the entropy inequality is strict, and hence the last inequality is strict. If $\kappa_0=\kappa_1$, then $\kappa_t=\kappa_0$. Since $H(\kappa_0\,|\,\gamma_x)<\infty$, we have $\kappa_0\ll\gamma$. Hence $\eta\mapsto\Wcal_2^2(\kappa_0,\eta)$ is strictly convex, and the last inequality is again strict. Thus $\eta\mapsto C_{\rm stat}^\beta(x,\eta)$ is strictly convex.

    Finally, we show that $C_{\rm stat}^\beta$ is continuous on $\Rd\times\cP_2(\Rd)$. By \eqref{lem:stat.cost.repr3} and since
    \[
        (x,\eta)\longmapsto \frac12|x-\bar\eta|^2-\frac\beta2|\bar\eta|^2
    \]
    is continuous on $\Rd\times(\cP_2(\Rd),\Wcal_2)$, it remains to prove continuity of
    \[
        \eta \longmapsto \inf_{\substack{\zeta \in \cP_2(\R^d),\\ \bar\zeta = 0}} \left\{ H(\zeta\,|\,\gamma) + \frac{\beta}{2}\Wcal_2^2(\zeta,\eta) \right\}.
    \]
    Lower semicontinuity follows exactly as in the first part of the proof of \Cref{prop:properties_inf_conv}, using the weak lower semicontinuity of $\zeta \mapsto H(\zeta \,|\, \gamma)$ and of $\Wcal_2^2$. 
    Upper semicontinuity follows by fixing a minimizer and using the continuity of $\eta\mapsto\Wcal_2^2(\zeta,\eta)$ on $\cP_2(\Rd)$. Thus $C_{\rm stat}^\beta$ is a continuous standard weak transport cost.
\end{proof}

\begin{lemma} \label{lem:stat.cost.dual.inequ}
    Let $\beta > 0$ and denote $C_{\rm stat}^\beta$ as in \Cref{lem:stat.cost}. Let $f$ be a $\beta$-semiconcave function.
    Then,
    \[
        - q_\beta \Box(-\Tcal^\beta_1[f])(x) \leq \inf_{\eta \in \cP_2(\Rd)}\left\{C_{\rm stat}^\beta(x,\eta) - \int f \,\rmd\eta\right\} = f^{C_{\rm stat}^\beta}(x),
    \]
    where $\Tcal^\beta_1[f]:=-\log\bigl(\exp(-q_\beta \Box (-f))\ast\gamma\bigr)$.
\end{lemma}

\begin{proof}
    For $x \in \Rd$ and $\eta \in \cP_2(\Rd)$ consider 
    \[
        C_{\rm stat}^\beta(x,\eta) = \sup_{y \in \R^d} \left\{-q_\beta(x-y) + \inf_{\rho  \in \cP_2(\R^d)}\bigl\{ H(\rho \,|\, \gamma_y) + W_\beta(\rho,\eta)\bigr\}\right\},
    \]
    where $W_\beta(\rho,\eta) =W_\beta(\rho,\eta)$.
    By \Cref{lem:stat.cost}, this is a standard weak transport cost, and the corresponding $C$-transform fulfills, for all $x \in \Rd$,
    \begin{align*}
        f^{C_{\rm stat}^\beta}(x) 
        &= \inf_{\eta \in \cP_2(\Rd)} \sup_{y \in \R^d} \left\{-q_\beta(x-y) + \inf_{\rho  \in \cP_2(\R^d)}\left\{ H(\rho \,|\, \gamma_y) + W_\beta(\rho,\eta) - \int f \, \rmd\eta\right\}\right\}\\
        &\ge \sup_{y \in \R^d} \inf_{\eta \in \cP_2(\Rd)} \left\{-q_\beta(x-y) + \inf_{\rho  \in \cP_2(\R^d)}\left\{ H(\rho \,|\, \gamma_y) + W_\beta(\rho,\eta) - \int f \, \rmd\eta\right\}\right\}.
    \end{align*}
    For $g \in C_{b,2}(\Rd)$ and with
    \[
        W_\beta(\mu,\nu) := \frac{\beta}{2}\Wcal_2^2(\mu,\nu),
        \qquad
        V_{\rm EOT}(\mu,\nu) := \inf_{\pi \in \cpl(\mu,\nu)} \int H(\pi_x \,|\, \gamma_x)\,\mu(\rmd x),
    \]
    which admit the $C$-transforms
    \begin{align*}
        g^{C_{V_{\rm EOT}}}(x)
        &= \inf_{\rho \in \cP_2(\R^d)}
        \left\{
            H(\rho \,|\, \gamma_x) - \int g \,\rmd \rho
        \right\}
        = -\log\bigl( \exp(g)\ast\gamma(x)\bigr), \\
        g^{C_{W_\beta}}(x)
        &= \inf_{y \in \R^d}
        \left\{
            \frac{\beta}{2}|x-y|^2 - g(y)
        \right\}
        = q_\beta \Box (-g)(x),
    \end{align*}
    the $C$-transform of the infimal convolution $V_{\rm EOT}\Box W_\beta$ is given by
    \[
        g^{C_{V_{\rm EOT}\Box W_\beta}}(y)=\inf_{\eta \in \cP_2(\Rd)}\inf_{\rho  \in \cP_2(\R^d)}\left\{H(\rho \,|\, \gamma_y) + W_\beta(\rho,\eta) - \int g \, \rmd\eta\right\}.
    \]
    Hence, by \Cref{prop:properties_inf_conv},
    \[
        \Tcal^\beta_1[f] :=
        f^{C_{V_{\rm EOT}\Box W_\beta}}
        =
        \bigl(-f^{C_{W_\beta}}\bigr)^{C_{V_{\rm EOT}}}
        =
        -\log\bigl(\exp(-q_\beta \Box (-f))\ast\gamma\bigr),
    \]
    so that
    \[
        f^{C_{\rm stat}^\beta}(x) \ge \sup_{y \in \R^d} \bigl\{-q_\beta(x-y) + \Tcal^\beta_1[f](x) \bigr\} = -q_\beta\Box(-\Tcal^\beta_1[f]).\qedhere
    \]
\end{proof}

\begin{lemma}[Tightness] \label{lem:tightness}
    Let $\beta>0$ and $\mu, \nu \in \cP_2(\Rd)$. Let $(f_n)_{n \in \N}$ be a sequence of $\beta$-semiconcave functions such that $\int f_n \,\rmd\nu = 0$ for all $n \in \N$, and define
    \[
        u_n := q_1 - \frac{1}{\beta}q_\beta \Box(-\Tcal^\beta_1[f_n]), \qquad \alpha_n := (\nabla u_n)_\#\mu,
    \]
    where $\Tcal^\beta_1[f_n]:=-\log\bigl(\exp(-q_\beta \Box (-f_n))\ast\gamma\bigr)$.
    Assume that 
    \[
        \int u_n\,\rmd\mu \leq \int u_{n+1}\,\rmd\mu \ \text{ for all } n \in \N.
    \]
    Then the following hold:
    \begin{enumerate}[label=(\roman*)]
        \item\label{it:tightness:uniform_bnd} There exists $c(\beta,d) \in \R$ such that, for all $x \in \Rd$,
        \[
            \sup_{n\in\N} |u_n(x)| \le c(\beta,d)\left(1 + |x|^2\right)
        \]
        
        \item\label{it:tightness:convergence} The sequence $(\nabla u_n)_{n\in\N}$ is uniformly bounded on compact subsets of $\Rd$ and equi-Lipschitz, and $(\alpha_n)_{n \in \N}$ is precompact in $\Pcal_2(\Rd)$. In particular, there exists a subsequence $(u_{n_k})_{k\in\N}$ and a convex, $\frac{1+\beta}{\beta}$-smooth function $u$ such that 
        \[
            u_{n_k} \longrightarrow u \ \text{ locally uniformly on } \Rd, \qquad \alpha_{n_k} \longrightarrow (\nabla u)\#\mu \ \text{ in } \Pcal_2(\Rd).
        \]
    \end{enumerate}
\end{lemma}

\begin{proof}
    For $x \in \Rd$ and $\eta \in \cP_2(\Rd)$ consider 
    \[
        C_{\rm stat}^\beta(x,\eta) = \sup_{y \in \R^d} \left\{-q_\beta(x-y) + \inf_{\rho  \in \cP_2(\R^d)}\left\{ H(\rho \,|\, \gamma_y) + W_\beta(\rho,\eta)\right\}\right\}.
    \]
    By \Cref{lem:stat.cost}, this is a standard weak transport cost, and there exists $c(\beta,d) \in \R$ such that
    \begin{equation*}
        C_{\rm stat}^\beta(x,\eta) \le c(\beta,d)\left(1 + |x|^2 + \int |y|^2\,\eta(\rmd y)\right).
    \end{equation*}
    This combined with \Cref{lem:stat.cost.dual.inequ}, yields by definition of the $C$-transform, for all $n \in \N$ and $x\in\Rd$,
    \[
        -q_\beta\Box(-\Tcal^\beta_1[f_n])(x) \leq f_n^{C_{\rm stat}^\beta}(x) + \int f_n\,\rmd\nu \leq c(\beta,d)\left(1 + |x|^2 + \int |y|^2\,\nu(\rmd y)\right).
    \]
    
    Next, consider the functions $(u_n)_{n\in\N}$. By the previous display, there exist constants $a,b>0$ such that, for all $x \in \Rd$,
    \begin{equation} \label{eq:brenier_pot_q_ub}
        \sup_{n \in \N} u_n(x) \leq a|x|^2+ b.
    \end{equation}
    Moreover, it follows from \Cref{cor:smooth_brenier_map} that $(u_n)_{n\in\N}$ are convex and $L:=\frac{1+\beta}{\beta}$-smooth, hence Brenier maps, so that the Descent Lemma (from classical optimization theory) gives, for all $n \in \N$ and all $x\in\Rd$,
    \[
        u_n(x) \leq u_n(\bar\mu) + (x-\bar\mu)\cdot\nabla u_n(\bar\mu) + \frac{L}{2}|x-\bar\mu|^2.
    \]
    Consequently,
    \[
        \int u_n\,\rmd\mu \leq u_n(\bar\mu) + \frac{L}{2}\int |x-\bar\mu|^2\,\mu(\rmd x).
    \]
    By assumption, we have $\int u_n \, \rmd\mu \leq \int u_{n+1}\,\rmd\mu$ for all $n\in\N$, which together with the previous display implies a uniform lower bound on $(u_n(\bar\mu))_{n\in\N}$. Moreover, evaluating \eqref{eq:brenier_pot_q_ub} at $x=\bar\mu$ gives a uniform upper bound on $(u_n(\bar\mu))_{n \in \N}$, so $(u_n(\bar\mu))_{n \in \N}$ is uniformly bounded. For each $n \in \N$, convexity also yields
    \[
        u_n(x) \geq u_n(\bar\mu) + (x-\bar\mu)\cdot\nabla u_n(\bar\mu)
    \]
    for all $x \in \R^d$. Together with \eqref{eq:brenier_pot_q_ub}, this yields \ref{it:tightness:uniform_bnd}.    
    
    Fix $n\in \N$. If $|\nabla u_n(\bar\mu)| > 0$, let $u:= \nabla u_n(\bar\mu)/|\nabla u_n(\bar\mu)|$. Then, for all $t \in \R$,
    \[
        u_n(tu) \geq u_n(\bar \mu) + |\nabla u_n(\bar\mu)|t - |\nabla u_n(\bar\mu)||\bar \mu|,
    \]
    which also holds in the case $|\nabla u_n(\bar\mu)|=0$. Combining this with \eqref{eq:brenier_pot_q_ub} and the uniform bound on $\left(u_n(\bar\mu)\right)_{n \in \N}$, we obtain
    \[
        |\nabla u_n(\bar\mu)|\,(t- |\bar \mu|) \leq a \,t^2 + b',
    \]
    for some $b'>0$ independent of $n$. Choosing $t = |\bar \mu|+1$ yields
    \[
        \sup_{n \in \N} |\nabla u_n(\bar\mu)| \leq a\,(|\bar \mu|+1)^2 + b' < \infty.
    \]
    The maps $(\nabla u_n)_{n \in \N}$ are $L$-Lipschitz, so, for all $x \in \Rd$ and all $n \in \N$,
    \[
        |\nabla u_n(x)| \le |\nabla u_n(\bar\mu)| + L|x-\bar\mu|.
    \]
    Hence, for all $n \in \N$,
    \[
        \int |z|^2\,\alpha_n(\rmd z) = \int |\nabla u_n(x)|^2\,\mu(\rmd x) \leq 2 |\nabla u_n(\bar \mu)|^2+2L^2\int |x-\bar\mu|^2\,\mu(\rmd x),
    \]
    and therefore $\sup_{n \in \N}\int |z|^2\,\alpha_n(\rmd z) < \infty$. By Markov’s inequality, this implies that $(\alpha_n)_{n \in \N}$ is tight. 

    Moreover, the bound on $\sup_{n \in \N} |\nabla u_n(\bar\mu)|$, together with the uniform Lipschitz constant $L$, implies that $(\nabla u_n)_{n \in \N}$ is uniformly bounded on compacts and equi-Lipschitz-continuous. Let $(K_n)_{n \in \N}$ be a sequence of compacts with $K_n \uparrow \Rd$. For each $n \in \N$, Arzelà--Ascoli yields a sub-sequence converging uniformly on $K_n$ to some $L$-Lipschitz map. By a standard diagonal argument, we extract a subsequence, denoted $\big(\nabla u_{n_k}\big)_{k \in \N}$, that converges locally uniformly on $\Rd$ to an $L$-Lipschitz continuous map $T$.
    
    Set $\alpha := T_\#\mu$. Then $\big(\nabla u_{n_k}, T\big)_\# \mu \in \Cpl\big(\alpha_{n_k},\alpha\big)$, and thus, for all $n \in \N$,
    \begin{align*}
        \Wcal_2^2\big(\alpha_{n_k}, \alpha\big) \leq \int \big|\nabla u_{n_k}(x) - T(x)\big|^2\,\rmd\mu(x) \leq \int_{K_n} \big|\nabla u_{n_k} - T\big|^2\,\rmd\mu + c \int_{K^\mathsf{c}_n} \big(1+|x-\bar\mu|^2\big) \, \mu(\rmd x),
    \end{align*}
    for some constant $c>0$ independent of $k,n$. This proves that $\Wcal_2(\alpha_{n_k},\alpha)\to0$. 
    By \ref{it:tightness:uniform_bnd} and potentially passing to another subsequence, $(u_{n_k})_{k \in \N}$ converge locally uniformly to a convex, $L$-smooth function $u$. In particular, ${\nabla u = T}$ and hence \ref{it:tightness:convergence} holds.
\end{proof}

\begin{lemma} \label{lem:exponential_moments}
    Let $\beta > 0$ and let $\alpha, \rho, \nu \in \cP_2(\R^d)$. Denote $v^*$ the Brenier potential from $\rho$ to $\nu$. Assume there exists $\pi \in \Cpl(\alpha,\rho)$ such that, for $x \in \Rd$,
    \[
        \frac{\rmd \pi_x}{\rmd \gamma_x} = \frac{e^{\beta v^* - q_\beta}}{e^{\beta v^* - q_\beta}*\gamma(x)},
    \]
    and assume that $\nu$ has all exponential moments, that is, $\int e^{t |y|} \, d\nu(y) < \infty$ for all $t > 0$.
    Suppose that 
    \[
        e^I :=\alpha\text{-\rm ess}\inf \left\{ e^{\beta v^*-q_\beta}*\gamma \right\} > 0
    \]
    and set $c := \frac1\beta |I| + |v^*(0)|$.
    Then, for all $t > 0$ and for all $B \subseteq \Rd$,
    \[
        \int_B e^{t|y|} \, \rho(\rmd y) \le \int_B e^{t |y|} \, \alpha\ast\gamma(\rmd y) + e^{2 t \sqrt{c}}\int_B e^{2t|y|} \, \nu(\rmd y).
    \]
\end{lemma}

\begin{proof}
    Let $I := \log\bigl(\alpha\text{-\rm ess}\inf \left\{ e^{\beta v^*-q_\beta}*\gamma \right\}\bigr)$.
    We split $\R^d$ into 
    \[
        A:= \left\{y \in \R^d : \beta v^*(y) - q_\beta \le I\right\},
        \qquad
        A^c=\R^d \setminus A.
    \]
    Note that, on $A$,
    \[
        \frac{d\pi_x}{d\gamma_x} = \frac{e^{\beta v^*-q_\beta}}{e^{\beta v^*-q_\beta} * \gamma(x)} \le e^{\beta v^*-q_\beta - I} \le 1.
    \]
    Consequently,
    \begin{equation}\label{ineq:proof:exp_mom:1}
        \int_{B \cap A} e^{t|y|} \, \rho(\rmd y) 
        \le 
        \int_{B \cap A} e^{t|y|} \, \alpha\ast\gamma(\rmd y) 
        \le
        \int_B e^{t |y|} \, \alpha\ast\gamma(\rmd y).
    \end{equation}
    To bound the remaining, we let $y \in A^c$, write $z = \nabla v^*(y)$.
    For $z:=\nabla v^*(y)$ we have ${v^*(0) + v(z) \ge 0}$ and ${v^*(y) + v(z) = y \cdot z}$ so that ${v^*(0) \ge v^*(y) - y \cdot z}$.
    Young's inequality then gives
    \[
        v^*(y)
        \le 
        v^*(0) + y \cdot z
        \le 
        v^*(0)+\frac14|y|^2+|z|^2.
    \]
    Since $\beta v^*(y)-q_\beta(y) > I$, this yields
    \[
        |y|^2 < 4 v^*(0) + 4|z|^2 - \frac{4}{\beta}I.
    \]
    Hence, with $c = \frac1\beta |I| + |v^*(0)|$,
    \begin{align*}
        \int_{B\cap A^c}  e^{t|y|} \, \rho(\rmd y) 
        \le 
        e^{2t\sqrt{c}} \int_{B\cap A^c} e^{2t|\nabla v^*(y)|} \, \rho(\rmd y) 
        \le
        e^{2t\sqrt{c}} \int_B e^{2t|y|} \, \nu(\rmd y).
    \end{align*}
    This combined with \eqref{ineq:proof:exp_mom:1} yields the claim.
\end{proof}

\printbibliography

@article{GuNiWi25,
  title={Dynamic characterization of barycentric optimal transport problems and their martingale relaxation},
  author={Guo, Ivan and Nilsson, Severin and Wiesel, Johannes},
  journal={arXiv preprint arXiv:2511.21287},
  year={2025}
}

@article{AlCoJo20,
  title={Sampling of probability measures in the convex order by Wasserstein projection},
  author={Alfonsi, Aur{\'e}lien and Corbetta, Jacopo and Jourdain, Benjamin},
  year={2020}
}

@misc{alouadi2026lightsbbmbridgingschrodingerbass,
      title={LightSBB-M: Bridging Schr\"odinger and Bass for Generative Diffusion Modeling}, 
      author={Alexandre Alouadi and Pierre Henry-Labord\`ere and Grégoire Loeper and Othmane Mazhar and Huyên Pham and Nizar Touzi},
      year={2026},
      eprint={2601.19312},
      archivePrefix={arXiv},
      primaryClass={cs.LG},
      url={https://arxiv.org/abs/2601.19312}, 
}

@article{gozlan2020mixture,
  title={On a mixture of Brenier and Strassen theorems},
  author={Gozlan, Nathael and Juillet, Nicolas},
  journal={Proceedings of the London Mathematical Society},
  volume={120},
  number={3},
  pages={434--463},
  year={2020},
  publisher={Wiley Online Library}
}

@article{backhoff2019existence,
  title={Existence, duality, and cyclical monotonicity for weak transport costs},
  author={Backhoff-Veraguas, Julio and Beiglb{\"o}ck, Mathias and Pammer, Gudmund},
  journal={Calculus of Variations and Partial Differential Equations},
  volume={58},
  number={6},
  pages={203},
  year={2019},
  publisher={Springer}
}

@article{hasenbichler2025martingalesinkhornalgorithm,
  title   = {The Martingale Sinkhorn Algorithm}, 
  author  = {Hasenbichler, Manuel and Joseph, Benjamin and Loeper, Gregoire and Obloj, Jan and Pammer, Gudmund},
  journal = {arXiv preprint arXiv:2310.13797}, 
  year    = {2025},
}

@book {BertsekasShreve1978,
    AUTHOR = {Bertsekas, Dimitri P. and Shreve, Steven E.},
     TITLE = {Stochastic optimal control},
    SERIES = {Mathematics in Science and Engineering},
    VOLUME = {139},
      NOTE = {The discrete time case},
 PUBLISHER = {Academic Press, Inc. [Harcourt Brace Jovanovich, Publishers],
              New York-London},
      YEAR = {1978},
     PAGES = {xiii+323},
   %   ISBN = {0-12-093260-1},
  % MRCLASS = {93E20},
 % MRNUMBER = {511544},
%MRREVIEWER = {Vladimir\ Katkovnik},
}

@article {BeJoMaPa23,
    AUTHOR = {Beiglb\"ock, Mathias and Jourdain, Benjamin and Margheriti,
              William and Pammer, Gudmund},
     TITLE = {Stability of the weak martingale optimal transport problem},
   JOURNAL = {Ann. Appl. Probab.},
  FJOURNAL = {The Annals of Applied Probability},
    VOLUME = {33},
      YEAR = {2023},
    NUMBER = {6B},
     PAGES = {5382--5412},
      ISSN = {1050-5164,2168-8737},
   MRCLASS = {49Q22 (60G42 91G80)},
  MRNUMBER = {4677736},
MRREVIEWER = {Marc\ Henry},
       DOI = {10.1214/23-aap1950},
       URL = {https://doi.org/10.1214/23-aap1950},
}

@misc{beiglböck2025fundamentaltheoremweakoptimal,
      title={The Fundamental Theorem of Weak Optimal Transport}, 
      author={Mathias Beiglböck and Gudmund Pammer and Lorenz Riess and Stefan Schrott},
      year={2025},
      eprint={2501.16316},
      archivePrefix={arXiv},
      primaryClass={math.PR},
      url={https://arxiv.org/abs/2501.16316}, 
}

@article {MR4941918,
    AUTHOR = {Acciaio, Beatrice and Marini, Antonio and Pammer, Gudmund},
     TITLE = {Calibration of the {B}ass local volatility model},
   JOURNAL = {SIAM J. Financial Math.},
  FJOURNAL = {SIAM Journal on Financial Mathematics},
    VOLUME = {16},
      YEAR = {2025},
    NUMBER = {3},
     PAGES = {803--833},
      ISSN = {1945-497X},
   MRCLASS = {91G30 (60G44 65J15)},
  MRNUMBER = {4941918},
       DOI = {10.1137/23M1622660},
       URL = {https://doi.org/10.1137/23M1622660},
}

@article {TanTouzi2013,
    AUTHOR = {Tan, Xiaolu and Touzi, Nizar},
     TITLE = {Optimal transportation under controlled stochastic dynamics},
   JOURNAL = {Ann. Probab.},
  FJOURNAL = {The Annals of Probability},
    VOLUME = {41},
      YEAR = {2013},
    NUMBER = {5},
     PAGES = {3201--3240},
      ISSN = {0091-1798,2168-894X},
   MRCLASS = {60H30 (49Q20 60H10 65K99 93E20)},
  MRNUMBER = {3127880},
MRREVIEWER = {Vivek\ S.\ Borkar},
       DOI = {10.1214/12-AOP797},
       URL = {https://doi.org/10.1214/12-AOP797},
}

@article {Strassen1965,
    AUTHOR = {Strassen, V.},
     TITLE = {The existence of probability measures with given marginals},
   JOURNAL = {Ann. Math. Statist.},
  FJOURNAL = {Annals of Mathematical Statistics},
    VOLUME = {36},
      YEAR = {1965},
     PAGES = {423--439},
      ISSN = {0003-4851},
   MRCLASS = {60.05 (60.20)},
  MRNUMBER = {177430},
MRREVIEWER = {J.\ Wolfowitz},
       DOI = {10.1214/aoms/1177700153},
       URL = {https://doi.org/10.1214/aoms/1177700153},
}

@article {BaBeiHueKae2020,
    AUTHOR = {Backhoff-Veraguas, Julio and Beiglb\"ock, Mathias and
              Huesmann, Martin and K\"allblad, Sigrid},
     TITLE = {Martingale {B}enamou-{B}renier: a probabilistic perspective},
   JOURNAL = {Ann. Probab.},
  FJOURNAL = {The Annals of Probability},
    VOLUME = {48},
      YEAR = {2020},
    NUMBER = {5},
     PAGES = {2258--2289},
      ISSN = {0091-1798,2168-894X},
   MRCLASS = {60G42 (49Q22 60G44 91G20)},
  MRNUMBER = {4152642},
       DOI = {10.1214/20-AOP1422},
       URL = {https://doi.org/10.1214/20-AOP1422},
}

@article {Leonard2014,
    AUTHOR = {L\'eonard, Christian},
     TITLE = {A survey of the {S}chr\"odinger problem and some of its
              connections with optimal transport},
   JOURNAL = {Discrete Contin. Dyn. Syst.},
  FJOURNAL = {Discrete and Continuous Dynamical Systems},
    VOLUME = {34},
      YEAR = {2014},
    NUMBER = {4},
     PAGES = {1533--1574},
      ISSN = {1078-0947,1553-5231},
   MRCLASS = {60J25 (46N10 60F10)},
  MRNUMBER = {3121631},
MRREVIEWER = {Nicolas\ Juillet},
       DOI = {10.3934/dcds.2014.34.1533},
       URL = {https://doi.org/10.3934/dcds.2014.34.1533},
}

@article {Lehec2013,
    AUTHOR = {Lehec, Joseph},
     TITLE = {Representation formula for the entropy and functional
              inequalities},
   JOURNAL = {Ann. Inst. Henri Poincar\'e{} Probab. Stat.},
  FJOURNAL = {Annales de l'Institut Henri Poincar\'e{} Probabilit\'es et
              Statistiques},
    VOLUME = {49},
      YEAR = {2013},
    NUMBER = {3},
     PAGES = {885--899},
      ISSN = {0246-0203,1778-7017},
   MRCLASS = {60J65 (39B62 60J70)},
  MRNUMBER = {3112438},
       DOI = {10.1214/11-aihp464},
       URL = {https://doi.org/10.1214/11-aihp464},
}

@incollection {Foellmer1985,
    AUTHOR = {F\"ollmer, H.},
     TITLE = {An entropy approach to the time reversal of diffusion
              processes},
 BOOKTITLE = {Stochastic differential systems ({M}arseille-{L}uminy, 1984)},
    SERIES = {Lect. Notes Control Inf. Sci.},
    VOLUME = {69},
     PAGES = {156--163},
 PUBLISHER = {Springer, Berlin},
      YEAR = {1985},
      ISBN = {3-540-15176-1},
   MRCLASS = {60J60 (60K35)},
  MRNUMBER = {798318},
MRREVIEWER = {Mich\`ele\ Mastrangelo-Dehen},
       DOI = {10.1007/BFb0005070},
       URL = {https://doi.org/10.1007/BFb0005070},
}

@article{GuoLoeper2021,
 ISSN = {10505164, 21688737},
 URL = {https://www.jstor.org/stable/27174891},
 author = {Ivan Guo and Grégoire Loeper},
 journal = {The Annals of Applied Probability},
 number = {3},
 pages = {pp. 1232--1263},
 publisher = {Institute of Mathematical Statistics},
 title = {Path dependent optimal transport and model calibration on exotic derivatives},
 urldate = {2026-02-06},
 volume = {31},
 year = {2021}
}

@article {BeiglHenryPenkner2013,
    AUTHOR = {Beiglb\"ock, Mathias and Henry-Labord\`ere, Pierre and
              Penkner, Friedrich},
     TITLE = {Model-independent bounds for option prices---a mass transport
              approach},
   JOURNAL = {Finance Stoch.},
  FJOURNAL = {Finance and Stochastics},
    VOLUME = {17},
      YEAR = {2013},
    NUMBER = {3},
     PAGES = {477--501},
      ISSN = {0949-2984,1432-1122},
   MRCLASS = {91G20 (49Q20 91G80)},
  MRNUMBER = {3066985},
MRREVIEWER = {R\u azvan\ R\u aducanu},
       DOI = {10.1007/s00780-013-0205-8},
       URL = {https://doi.org/10.1007/s00780-013-0205-8},
}

@article {BackhoffSchachTschid2025,
    AUTHOR = {Backhoff-Veraguas, Julio and Schachermayer, Walter and
              Tschiderer, Bertram},
     TITLE = {The {B}ass functional of martingale transport},
   JOURNAL = {Ann. Appl. Probab.},
  FJOURNAL = {The Annals of Applied Probability},
    VOLUME = {35},
      YEAR = {2025},
    NUMBER = {6},
     PAGES = {4282--4301},
      ISSN = {1050-5164,2168-8737},
   MRCLASS = {60G42 (49Q22 60G44 91G20)},
  MRNUMBER = {4993803},
       DOI = {10.1214/25-AAP2221},
       URL = {https://doi.org/10.1214/25-AAP2221},
}

@article{gozlan2017kantorovich,
  title={Kantorovich duality for general transport costs and applications},
  author={Gozlan, Nathael and Roberto, Cyril and Samson, Paul-Marie and Tetali, Prasad},
  journal={Journal of Functional Analysis},
  volume={273},
  number={11},
  pages={3327--3405},
  year={2017},
  publisher={Elsevier}
}

@article{backhoff2023existence,
  title={Existence of Bass martingales and the martingale Benamou-Brenier problem in $\mathbb{R}^d$},
  author={Backhoff-Veraguas, Julio and Beiglb{\"o}ck, Mathias and Schachermayer, Walter and Tschiderer, Bertram},
  journal={Preprint, available at https://arxiv.org/abs/2306.11019 v3},
  year={2023}
}

@article{beiglbock2016problem,
  title={On a problem of optimal transport under marginal martingale constraints},
  author={Beiglb{\"o}ck, Mathias and Juillet, Nicolas},
  year={2016}
}

@article{galichon2014stochastic,
  title={A stochastic control approach to no-arbitrage bounds given marginals, with an application to lookback options},
  author={Galichon, Alfred and Henry-Labord\`ere, Pierre and Touzi, Nizar},
  year={2014}
}

@article{conze2021bass,
  title={Bass construction with multi-marginals: Lightspeed computation in a new local volatility model},
  author={Conze, Antoine and Henry-Labord\`ere, Pierre},
  journal={Available at SSRN 3853085},
  year={2021}
}

@article{henrylabordere2026bridgingschrodingerbasssemimartingale,
      title={Bridging Schr\"odinger and Bass: A Semimartingale Optimal Transport Problem with Diffusion Control}, 
      author={Pierre Henry-Labord\`ere and Grégoire Loeper and Othmane Mazhar and Huyên Pham and Nizar Touzi},
      year={2026},
      eprint={2603.27712},
      archivePrefix={arXiv},
      primaryClass={math.PR},
      url={https://arxiv.org/abs/2603.27712}, 
}

@article{MiTh6,
  title={Duality theorem for the stochastic optimal control problem},
  author={Mikami, Toshio and Thieullen, Mich{\`e}le},
  journal={Stochastic processes and their applications},
  volume={116},
  number={12},
  pages={1815--1835},
  year={2006},
  publisher={Elsevier}
}

@article{MiTh8,
  title={Optimal transportation problem by stochastic optimal control},
  author={Mikami, Toshio and Thieullen, Michele},
  journal={SIAM Journal on Control and Optimization},
  volume={47},
  number={3},
  pages={1127--1139},
  year={2008},
  publisher={SIAM}
}

@article{MaGe20,
  title={An optimal transport approach for the Schr{\"o}dinger bridge problem and convergence of Sinkhorn algorithm},
  author={Marino, Simone Di and Gerolin, Augusto},
  journal={Journal of Scientific Computing},
  volume={85},
  number={2},
  pages={27},
  year={2020},
  publisher={Springer}
}

@article{JoLoOb26,
  title={Calibration of local volatility models with stochastic interest rates using optimal transport},
  author={Joseph, Benjamin and Loeper, Gr{\'e}goire and Ob{\l}{\'o}j, Jan},
  journal={Finance and Stochastics},
  pages={1--43},
  year={2026},
  publisher={Springer}
}

@inproceedings{ElLeSh20,
  title={Stability of the logarithmic Sobolev inequality via the Follmer process},
  author={Eldan, Ronen and Lehec, Joseph and Shenfeld, Yair},
  booktitle={Annales De L Institut Henri Poincare-Probabilites Et Statistiques},
  volume={56},
  number={3},
  pages={2253--2269},
  year={2020}
}

@article{MiSh24,
  title={The Brownian transport map},
  author={Mikulincer, Dan and Shenfeld, Yair},
  journal={Probability Theory and Related Fields},
  volume={190},
  number={1},
  pages={379--444},
  year={2024},
  publisher={Springer}
}

@article{LeVe24,
  title={Eldan's stochastic localization and the KLS conjecture: Isoperimetry, concentration and mixing},
  author={Lee, Yin Tat and Vempala, Santosh S},
  journal={Annals of Mathematics},
  volume={199},
  number={3},
  pages={1043--1092},
  year={2024},
  publisher={Department of Mathematics, Princeton University Princeton, New Jersey, USA}
}

@article{KiRu24,
  title={Backward and forward Wasserstein projections in stochastic order},
  author={Kim, Young-Heon and Ruan, Yuanlong},
  journal={Journal of Functional Analysis},
  volume={286},
  number={2},
  pages={110201},
  year={2024},
  publisher={Elsevier}
}

@article{DeTo19,
  title={Irreducible convex paving for decomposition of multidimensional martingale transport plans},
  author={De March, Hadrien and Touzi, Nizar},
  journal={The Annals of Probability},
  volume={47},
  number={3},
  pages={1726--1774},
  year={2019},
  publisher={JSTOR}
}

@article{ObSi17,
  title={Structure of martingale transports in finite dimensions},
  author={Ob{\l}{\'o}j, Jan and Siorpaes, Pietro},
  journal={arXiv preprint arXiv:1702.08433},
  year={2017}
}

@inproceedings{alfonsi2020sampling,
  title={Sampling of probability measures in the convex order by Wasserstein projection},
  author={Alfonsi, Aur{\'e}lien and Corbetta, Jacopo and Jourdain, Benjamin},
  booktitle={Annales de l'Institut Henri Poincar{\'e} (B) Probabilit{\'e}s et Statistiques},
  volume={56},
  number={3},
  pages={1706--1729},
  year={2020}
}

@article{BaLo19,
  title={Central Limit Theorem for empirical transportation cost in general dimension},
  author={del Barrio, Eustasio and Loubes, Jean-Michel},
  journal={The Annals of Probability},
  volume={47},
  number={2},
  pages={926--951},
  year={2019}
}

@book{Rock97,
  title={Convex analysis},
  author={Rockafellar, R Tyrrell},
  volume={28},
  year={1997},
  publisher={Princeton university press}
}

@book{Rock98,
  title={Variational analysis},
  author={Rockafellar, R Tyrrell and Wets, Roger JB},
  year={1998},
  publisher={Springer}
}

@article{BrJu22,
  title={Instability of martingale optimal transport in dimension $d \geq 2$},
  author={Br{\"u}ckerhoff, Martin and Juillet, Nicolas},
  journal={Electronic Communications in Probability},
  volume={27},
  pages={1--10},
  year={2022},
  publisher={The Institute of Mathematical Statistics and the Bernoulli Society}
}

@article{blanchet2019quantifying,
  title={Quantifying distributional model risk via optimal transport},
  author={Blanchet, Jose and Murthy, Karthyek},
  journal={Mathematics of Operations Research},
  volume={44},
  number={2},
  pages={565--600},
  year={2019},
  publisher={INFORMS}
}

@article{gao2023distributionally,
  title   = {Distributionally Robust Stochastic Optimization with Wasserstein Distance},
  author  = {Gao, Rui and Kleywegt, Anton J.},
  journal = {Mathematics of Operations Research},
  volume  = {48},
  number  = {2},
  pages   = {603--655},
  year    = {2023},
  doi     = {10.1287/moor.2022.1275}
}

@article{jordan1998variational,
  title={The variational formulation of the Fokker--Planck equation},
  author={Jordan, Richard and Kinderlehrer, David and Otto, Felix},
  journal={SIAM journal on mathematical analysis},
  volume={29},
  number={1},
  pages={1--17},
  year={1998},
  publisher={SIAM}
}

@book{ambrosio2005gradient,
  title={Gradient flows: in metric spaces and in the space of probability measures},
  author={Ambrosio, Luigi and Gigli, Nicola and Savar{\'e}, Giuseppe},
  year={2005},
  publisher={Springer}
}
\end{document}